%% file: subgraph-poisson-arxiv-aap.tex
\documentclass[twosided,11pt]{article}

\usepackage{graphicx, subfigure}
\usepackage[top=1in,bottom=1in,left=1in,right=1in]{geometry}
\usepackage[sort&compress]{natbib} 
\usepackage{amsfonts, amsmath, amssymb, amsthm, constants}

\input{macros-aap}

\begin{document}
\thispagestyle{empty}

% TITLE
\noindent {\sc \LARGE Community Detection in Sparse Random Networks}

% AUTHORS
\bigskip 
\noindent {\Large
Nicolas Verzelen\footnote{INRA, UMR 729 MISTEA, F-34060 Montpellier, FRANCE}
and
Ery Arias-Castro\footnote{Department of Mathematics, University of California, San Diego, USA}} \\[.05in]

% ABSTRACT
\bigskip
\noindent
We consider the problem of detecting a tight community in a sparse random network.  This is formalized as testing for the existence of a dense random subgraph in a random graph.  Under the null hypothesis, the graph is a realization of an Erd\"os-R\'enyi graph on $N$ vertices and with connection probability $p_0$; under the alternative, there is an unknown subgraph on $n$ vertices where the connection probability is $p_1 > p_0$.  In \citep{subgragh_detection}, we focused on the asymptotically {\em dense} regime where $p_0$ is large enough that $np_0>(n/N)^{o(1)}$.  We consider here the asymptotically {\em sparse} regime where $p_0$ is small enough that $np_0<(n/N)^{c_0}$ for some $c_0>0$.  As before, we derive information theoretic lower bounds, and also establish the performance of various tests.  Compared to our previous work \citep{subgragh_detection}, the arguments for the lower bounds are based on the same technology, but are substantially more technical in the details; also, 
the 
methods we study are different: besides a variant of the scan statistic, we study other tests statistics such as the size of the largest connected component, the number of triangles, and the number of subtrees of a given size.  Our detection bounds are sharp, except in the Poisson regime where we were not able to fully characterize the constant arising in the bound.


\section{Introduction}
\label{sec:intro}

Community detection refers to the problem of identifying communities in networks, e.g., circles of friends in social networks, or groups of genes in graphs of gene co-occurrences \citep{PhysRevE.80.056117,PhysRevE.69.026113,Bickel15122009, Newman06062006,PhysRevE.74.016110,Girvan11062002}.  
Although fueled by the increasing importance of graph models and network structures in applications, and the emergence of large-scale social networks on the Internet, the topic is much older in the social sciences, and the algorithmic aspect is very closely related to graph partitioning, a longstanding area in computer science. 
We refer the reader to the comprehensive survey paper of \cite{Santo201075} for more examples and references.

By community detection we mean, here, something slightly different.  Indeed, instead of aiming at extracting the community (or communities) from within the network, we simply focus on deciding whether or not there is a community at all.  Therefore, instead of considering a problem of graph partitioning, or clustering, we consider a problem of testing statistical hypotheses.  
We observe an undirected graph $\cG = (\cE, \cV)$ with $N := |\cV|$ nodes.  Without loss of generality, we take $\cV = [N] := \{1, \dots, N\}$.  The corresponding adjacency matrix is denoted $\bW = (W_{i,j}) \in \{0,1\}^{N \times N}$, where $W_{i,j} = 1$ if, and only if, $(i,j) \in \cE$, meaning there is an edge between nodes $i, j \in \cV$.  Note that $\bW$ is symmetric, and we assume that $W_{ii} = 0$ for all $i$.  Under the null hypothesis, the graph $\cG$ is a realization of $\bbG(N, p_0)$, the Erd\"os-R\'enyi random graph on $N$ nodes with probability of connection $p_0 \in (0,1)$; equivalently, the upper diagonal entries of $\bW$ are independent and identically distributed with $\P(W_{i,j} = 1) = p_0$ for any $i \neq j$.  Under the alternative, there is a subset of nodes indexed by $S \subset    \cV$ such that $\P(W_{i,j} = 1) = p_1$ for any $i,j \in S$ with $i \neq j$, while $\P(W_{i,j} = 1) = p_0$ for any other pair of nodes $i \neq j$.  We assume that $p_1 > p_0$, implying that the connectivity is 
stronger 
between nodes in $S$, so that $S$ is an assortative community.  
The subset $S$ is not known, although in most of the paper we assume that its size $n := |S|$ is known. 
Let $H_0$ denote the null hypothesis, which consists of $\bbG(N, p_0)$ and is therefore simple.  And let $H_S$ denote the alternative where $S$ is the anomalous subset of nodes.  We are testing $H_0$ versus $H_1 := \bigcup_{|S|=n} H_S$. 
We consider an asymptotic setting where 
\beq \label{asymp}
N \to \infty, \quad n = n(N) \to \infty, \quad n/N \to 0, \quad n/\log N \to \infty,
\eeq 
meaning the graph is large in size, and the subgraph is comparatively small, but not too small.  Also, the probabilities of connection, $p_0 = p_0(N)$ and $p_1 = p_1(N)$, may change with $N$ --- in fact, they will tend to zero in most of the paper.

Despite its potential relevance to applications, this problem has received considerably less attention.  We mention the work of \cite{wang2008spatial} who, in a somewhat different model, propose a test based on a statistic similar to the modularity of \cite{PhysRevE.69.026113}; the test is evaluated via simulations.  \cite{MR2460888} consider the problem of detecting a clique, a problem that we addressed in detail in our previous paper \citep{subgragh_detection}, and which is a direct extension of the `planted clique problem' \citep{MR2735341,MR1662795,dekel}.
\cite{rukhin2012limiting} consider a test based on the maximum number of edges among the subgraphs induced by the neighborhoods of the vertices in the graph; they obtain the limiting distribution of this statistic in the same model we consider here, with $p_0$ and $p_1$ fixed, and $n$ is a power of $N$, and in the process show that their test reduces to the test based on the maximum degree.  Closer in spirit to our own work, \cite{1109.0898} study this testing problem in the case where $p_0$ and $p_1$ are fixed.  A dynamic setting is considered in \citep{MR2758643,mongiovinetspot,6380528} where the goal is to detect changes in the graph structure over time.

\subsection{Hypothesis testing}

 We start with some concepts related to hypothesis testing.
We refer the reader to \citep{TSH} for a thorough introduction to the subject.
A test $\phi$ is a function that takes $\bW$ as input and returns $\phi =1$ to claim there is a community in the network, and $\phi=0$ otherwise.  The (worst-case) risk of a test $\phi$ is defined as
\beq \label{gamma}
\gamma_N(\phi) = \P_0(\phi = 1) + \max_{|S| = n} \P_S(\phi = 0) \ ,
\eeq
where $\P_0$ is the distribution under the null $H_0$ and $\P_S$ is the distribution under $H_S$, the alternative where $S$ is anomalous.  We say that a sequence of tests $(\phi_N)$ for a sequence of problems $(\bW_N)$ is asymptotically powerful (resp.~powerless) if $\gamma_N(\phi_N) \to 0$ (resp.~$\to 1$).  We will often speak of a test being powerful or powerless when in fact referring to a sequence of tests and its asymptotic power properties.  Then, practically speaking, a test is asymptotically powerless if it does not perform substantially better than any method that ignores the adjacency matrix $\bW$, i.e., guessing.  We say that the hypotheses merge asymptotically if 
\[
\gamma_N^* := \inf_{\phi} \gamma_N(\phi) \to 1 \ ,
\]
and that the hypotheses separate completely asymptotically if $\gamma_N^* \to 0$, which is equivalent to saying that there exists a sequence of asymptotically powerful tests.  Note that if $\liminf \gamma_N^* > 0$, no sequence of tests is asymptotically powerful, which includes the special case where the two hypotheses are contiguous.

Our general objective is to derive the detection boundary for the problem of community detection. On the one hand, we want to characterize the range of parameters $(n,N,p_0,p_1)$ such that either all tests are asymptotically powerless $(\gamma_N^*\to 1)$ or no test is asymptotically powerful $(\lim\inf \gamma_N^*> 0)$. On the other hand, we want to introduce asymptotically minimax optimal tests, that is tests $\phi$ satisfying  $\gamma_N(\phi)\to 0$ whenever $\gamma_N(\phi)\to 0$ or  $\lim\sup \gamma_N^*<1$ whenever $\lim\sup \gamma_N^*<1$.

%\subsubsection*{Our previous work}
%\medskip 
\subsection{Our previous work}
We recently considered this testing problem in \citep{subgragh_detection}, focusing on the {\em dense} regime where $\log(1 \vee (np_0)^{-1}) = o(\log(N/n))$ or equivalently $p_0\geq n^{-1}(n/N)^{o(1)}$.  (For $a,b \in \bbR$, $a \wedge b$ denotes the minimum of $a$ and $b$ and $a \vee b$ denotes their maximum.)  We obtained information theoretic lower bounds, and we proposed and analyzed a number of methods, both when $p_0$ is known and when it is unknown.  (None of the methods we considered require knowledge of $p_1$.)  In particular, a combination of the total degree test based on 
\beq \label{total-stat}
\tot := \sum_{1 \leq i < j \leq N} W_{i,j} \ ,
\eeq
and the scan test based on 
\beq \label{scan-stat}
\scan := \max_{|S| = n} W_S, \qquad W_S := \sum_{i,j\in S, i < j} W_{i,j} \ ,
\eeq
was found to be asymptotically minimax optimal when $p_0$ is known and when $n$ is not too small, specifically $n/\log N \to \infty$.  This extends the results that \cite{1109.0898} obtained for $p_0$ and $p_1$ fixed (and $p_0$ known).  In that same paper, we also proposed and studied a convex relaxation of the scan test, based on the largest $n$-sparse eigenvalue of $\bW^2$, inspired by related work of \cite{berthet}.

%\subsection{In this paper}
\subsection{Contribution}

Continuing our work, in the present paper we focus on the {\em sparse} regime where 
\beq \label{sparse}
p_0\le \frac{1}{n} \left(\frac{n}{N}\right)^{c_0} \text{ for some constant $c_0 > 0$.}
\eeq
Obviously, \eqref{sparse} implies that $n p_0 \le 1$.
We define
\beq \label{lambda}
\lambda_0 = N p_0, \qquad \lambda_1 = n p_1,
\eeq
and note that $\lambda_0$ and $\lambda_1$ may vary with $N$. 
Our results can be summarized as follows.

\medskip

\noindent 
{\bf Regime 1: $\lambda_0=(N/n)^\alpha$ with fixed $0<\alpha<1$.} 
Compared to the setting in our previous work \citep{subgragh_detection}, the total degree test \eqref{total-stat} remains a contender, scanning over subsets of size exactly $n$ as in \eqref{scan-stat} does not seem to be optimal anymore, all the more so when $p_0$ is small.  Instead, we scan over subsets of a wider range of sizes, using
\begin{equation}\label{new-scan}
W_n^\ddag= \sup_{k=n/u_N}^{n}\frac{W_k^*}{k} \ ,
\end{equation}
where $u_N=\log\log(N/n)$.
We call this the broad scan test.  In analogy with our previous results in \citep{subgragh_detection}, we find that a combination of the total degree test \eqref{total-stat} and the broad scan test based on \eqref{new-scan} is asymptotically optimal when $\lambda_0\to \infty$, in the following sense. 
Suppose %$\lambda_0=(N/n)^\alpha$ with $0<\alpha < 1$ --- the case $\lim\inf \alpha \geq 1$ being settled in \citep{subgragh_detection} --- and consider 
$n=N^{\kappa}$ with $0<\kappa<1$. 
When $\kappa>\frac{1+\alpha}{2+\alpha}$, the total degree test is asymptotically powerful when $\lambda_1\gg \frac{N^{(1+\alpha)/2}}{n^{1+\alpha}}$ and the two hypotheses merge asymptotically when $\lambda_1\ll \frac{N^{(1+\alpha)/2}}{n^{1+\alpha}}$. 
(For two sequences of reals, $(a_N)$ and $(b_N)$, we write $a_N \ll b_N$ to mean that $a_N = o(b_N)$.)
When $\kappa<\frac{1+\alpha}{2+\alpha}$, that is for smaller $n$, there exists a sequence of increasing functions $\psi_n$ (defined in Theorem \ref{thm:broad}) such that the broad scan test is asymptotically powerful when $\lim\inf (1-\alpha)\psi_n(\lambda_1)>1$ and the hypotheses merge asymptotically when $\lim\sup (1-\alpha)\psi_n(\lambda_1)<1$. Furthermore, as $n \to \infty$, $\psi_n(\lambda) \asymp \lambda$ when $\lambda\geq 1$ remains fixed, while $\psi_n(1)\to 1$, and $\psi_n(\lambda)\sim \lambda/2$ for $\lambda\to \infty$.  
As a consequence, the  broad scan test is asymptotically powerful when $\lambda_1$ is larger than (up to a numerical) $(1-\alpha)^{-1}$.
See Table \ref{tab:polynomial} for a visual summary. 
(For two real sequences, $(a_N)$ and $(b_N)$, we write $a_N \prec b_N$ to mean that $a_N = O(b_N)$, and $a_N \asymp b_N$ when $a_N \prec b_N$ and $a_N \succ b_N$.)

\begin{table}
\caption{
Detection boundary and near-optimal algorithms  in the regime $\lambda_0=(N/n)^{\alpha}$ with $0<\alpha<1$ and $n=N^{\kappa}$ with $0<\kappa<1$.
Here, `undetectable' means that that the hypotheses merge asymptotically, while `detectable' means that there exists an asymptotically powerful test. 
}
\label{tab:polynomial}
\centering
\medskip
\def\arraystretch{2}
\begin{tabular}{ c ||c | c }
 $\kappa$ &  $\kappa<\frac{1+\alpha}{2+\alpha}$  &  $\kappa> \frac{1+\alpha}{2+\alpha}$\\ \hline 
Undetectable  &  $\lambda_1 \prec (1-\alpha)^{-1}$;  Exact Eq. in \eqref{lower4} &  $\lambda_1 \ll \frac{N^{(1+\alpha)/2}}{n^{1+\alpha}}$\\ 
Detectable & $\lambda_1 \succ (1-\alpha)^{-1}$; Exact Eq. in \eqref{broad1}&   $\lambda_1 \gg \frac{N^{(1+\alpha)/2}}{n^{1+\alpha}}$ \\[.05in] 
\hline 
Optimal test  & {\sc Broad Scan test} & {\sc Total Degree test}
\end{tabular}

\end{table}

%We also consider the test based on the size of the largest connected component, a test based on counting the number of subtrees of a certain (logarithmic) size, the test based on the number of triangles, and few others including spectral methods.

\medskip

When $N^{-o(1)}\leq \lambda_0\leq (N/n)^{o(1)}$ and $n=N^{\kappa}$ with $1/2<\kappa<1$, the total degree test is optimal, in the sense that it is asymptotically powerful for $\lambda_1^2/\lambda_0\gg n^2/N$, while the hypotheses merge asymptotically for $\lambda_1^2/\lambda_0\ll n^2/N$. This is why we assume in the remainder of this discussion that $n=N^{\kappa}$ with $0<\kappa<1/2$.

\medskip

\noindent 
{\bf Regime 2: $\lambda_0\to \infty$ with $\log(\lambda_0)=o[\log(N/n)]$.}   When $\kappa<\tfrac{1}{2}$,  the broad scan test is asymptotically powerful when $\lim\inf\lambda_1 >1$ and the hypotheses merge asymptotically when $\lim\sup \lambda_1 <1$. See the first line of Table \ref{tab:divergence} for a visual summary.

\medskip

\noindent
{\bf Regime 3: $\lambda_0>0$ and $\lambda_1>0$ are fixed.} 
The Poissonian regime where $\lambda_0$ and $\lambda_1$ are assumed fixed is depicted on Figure~\ref{fig:poisson}.
 When $\lambda_1>1$, the broad scan test is asymptotically powerful. When $\lambda_0>e$ and $\lambda_1<1$, no test is able to fully separate the hypotheses. 
In fact, for any fixed $(\lambda_0,\lambda_1)$  a test based on the number of triangles has some nontrivial power (depending on $(\lambda_0,\lambda_1)$), implying that the two hypotheses do not completely merge in this case. 
The case  where $ \lambda_0 < e$ is not completely settled. No test is able to fully separate the hypotheses if $\lambda_1 <\sqrt{\lambda_0/e}$.  The largest connected component test is optimal up to a constant when $\lambda_0 < 1$ and a test based on counting subtrees of a certain size bridges the gap in constants for $1 \le \lambda_0<e$, but not completely. 
When $\lambda_0$ is bounded from above and $\lambda_1=o(1)$, the two hypotheses merge asymptotically.

\begin{figure} 
\begin{center}
\includegraphics[width=9cm]{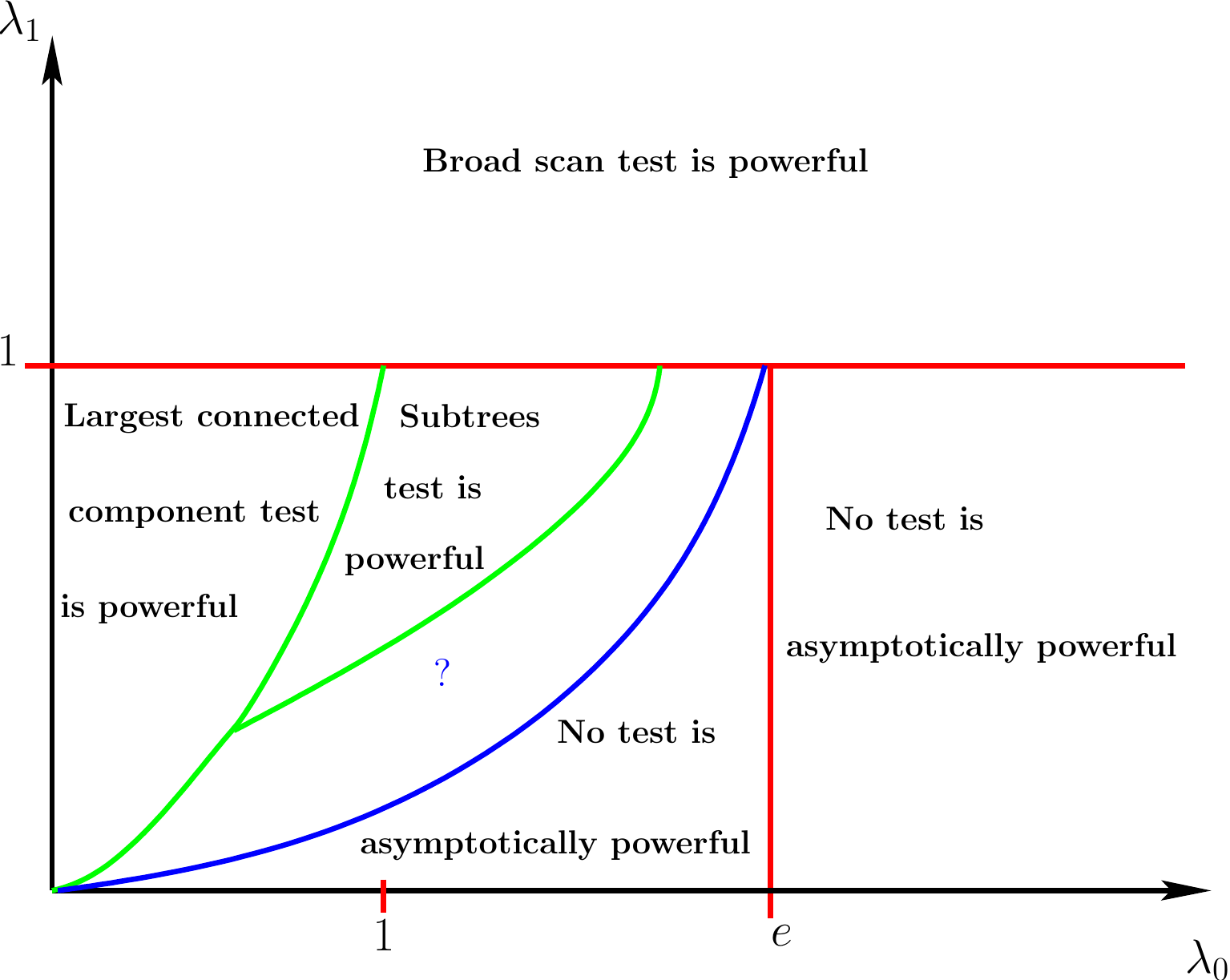} 
\caption{Detection diagram  in the poissonian asymptotic where $\lambda_0$ and $\lambda_1$ are fixed and $n=N^{\kappa}$ with $0<\kappa<1/2$. }
\label{fig:poisson}
\end{center}
\end{figure}

\medskip

\noindent 
{\bf Regime 4: $\lambda_0=o(1)$ with $\log(1/\lambda_0)=o[\log(N)]$.} Finally, when $\lambda_0 \to 0$, the largest connected component test is asymptotically optimal.  See Table~\ref{tab:divergence}.

\begin{table}
\caption{Detection boundary and near-optimal algorithms  in the regimes $\lambda_0\to \infty $  and $\lambda_0\to 0$ and $n=N^{\kappa}$ with $0<\kappa<1/2$.
For $1/2<\kappa<1$, the detection boundary accurs at $\lambda_1 \asymp  N^{1/2}/n^{2}$ and is achieved by the total degree test.
}
\label{tab:divergence}
\centering
\medskip
\def\arraystretch{2}
\begin{tabular}{c||c|c}
$\lambda_0$ & $1\ll \lambda_0\ll \left(\frac{N}{n}\right)^{o(1)}$ & $\frac{1}{N^{o(1)}}\leq \lambda_0=o(1)$ \\ \hline
Undetectable  & $\lim\sup \lambda_1<1$ & $\lim\sup \frac{\log(\lambda_1^{-1})}{\log(\lambda_0^{-1})}>\kappa$ \\ 
Detectable & $\lim\inf \lambda_1>1$ & $\lim\inf \frac{\log(\lambda_1^{-1})}{\log(\lambda_0^{-1})}<\kappa $ \\[.05in]  \hline
Optimal test & {\sc Largest CC test} &  {\sc Broad Scan test}
\end{tabular}
\end{table}

%\begin{tabular}{ c ||c| c|c | c }
%  &  Undetectable &   Detectable  & Optimal test  \\ \hline \hline
%$1\ll \lambda_0\ll \left(\frac{N}{n}\right)^{o(1)}$  & $\lim\sup \lambda_1<1$ & $\lim\inf \lambda_1>1$ &  {\sc Broad Scan test} \\  \hline 
%$\frac{1}{N^{o(1)}}\leq \lambda_0=o(1)$&  $\lim\sup \frac{\log(\lambda_1^{-1})}{\log(\lambda_0^{-1})}>\kappa$  & $\lim\inf \frac{\log(\lambda_1^{-1})}{\log(\lambda_0^{-1})}<\kappa $& {\sc Largest CC test}
%\end{tabular}

%\medskip
%
%\nv{do we remove the two next sentences?}
%\eac{YES, if we are to write a little paper on computationally feasible methods, which I would like to do, complete with experiments whenever feasible.}
%We also discuss tests that can be computed in polynomial time.  Besides the total degree test, the test based on the largest connected component and the number of triangle test, already mentioned, we discuss the maximum degree test, a test based on counting simple cycles of given (small) length, and also some spectral methods.  We find that, in the regime where $\lambda_0 \to \infty$, no test among these seems to come close to the broad scan test. 

\subsection{Methodology for the lower bounds}
Compared to our previous work \citep{subgragh_detection}, the derivation of the various lower bounds here rely on the same general approach.  Let $\bbG(N,p_0; n, p_1)$ denote the random graph obtained by choosing $S$ uniformly at random among subsets of nodes of size $n$, and then generating the graph under the alternative with $S$ being the anomalous subset.  When deriving a lower bound, we first reduce the composite alternative to a simple alternative, by testing $H_0 : \bbG(N, p_0)$ versus $\bar{H}_1 := \bbG(N,p_0;n, p_1)$.  Let $L$ denote the corresponding likelihood ratio, i.e., $L = \sum_{|S| = n} L_S/\binom{N}{n}$, where $L_S$ is the likelihood ratio for testing $H_0$ versus $H_S$.
Then these hypotheses merge in the asymptote if, and only if, $L \to 1$ in probability under $H_0$.  A variant of the so-called `truncated likelihood' method, introduced by \cite{1109.0898}, consists in proving that $\E_0(\tilde L) \to 1$ and $\E_0(\tilde L^2) \to 1$, where $\tilde L$ is a truncated likelihood of the form 
$\tilde L = \sum_{|S| = n} L_S \1_{\Gamma_S}/\binom{N}{n}$,
where $\Gamma_S$ is a carefully chosen event.
(For a set or event $A$, $\1_A$ denotes the indicator function of $A$.) 
An important difference with our previous work is the more delicate choice of $\Gamma_S$, which here relies more directly on properties of the graph under consideration.
We mention that we use a variant to show that $H_0$ and $\bar{H}_1$ do {\em not} separate in the limit.  
This could be shown by proving that the two graph models $\bbG(N, p_0)$ and $\bbG(N,p_0;n, p_1)$ are contiguous.  The `small subgraph conditioning' method of \cite{MR1262985,MR1151355} --- see the more recent exposition in \citep{MR1725006} --- was designed for that purpose.  For example, this is the method that \cite{mossel2012stochastic} use to compare a Erd\"os-R\'enyi graph with a stochastic block model\footnote{This is a popular model of a network with communities, also known as the planted partition model.  In this model, the nodes belong to blocks: nodes in the same block connect with some probability $p_{\rm in}$, while nodes in different blocks connect with probability $p_{\rm out}$.} with two blocks of equal size.  
This method does not seem directly applicable in the situations that we consider here, in part because the second moment of the likelihood ratio, meaning $\E[L^2]$, tends to infinity at the limit of detection.

\subsection{Content}
The remaining of the paper is organized as follow.  
In \secref{prelim} we introduce some notation and some concepts in probability and statistics, including concepts related to hypothesis testing and some basic results on the binomial distribution.
In \secref{tests} we study some tests that are near-optimal in different regimes.
In \secref{lower} we state and prove information theoretic lower bounds on the difficulty of the detection problem.
%In \secref{poly} we study various tests that run in polynomial-time.  
In \secref{discussion} we discuss the situations where $p_0$ and/or $n$ are unknown, as well as open problems.
\secref{aux} contains some proofs and technical derivations.

\section{Preliminaries}
\label{sec:prelim}

In this section, we first define some general assumptions and some notation, although more notation will be introduced as needed.    We then list some general results that will be used multiple times throughout the paper.

\subsection{Assumptions and notation} \label{sec:notation}
We recall that $N \to \infty$ and the other parameters such as $n, p_0, p_1$ may change with $N$, and this dependency is left implicit.  
Unless otherwise specified, all the limits are with respect to $N \to \infty$.
We assume that $N^2 p_0 \to \infty$, for otherwise the graph (under the null hypothesis) is so sparse that number of edges remains bounded.  Similarly, we assume that $n^2 p_1 \to \infty$, for otherwise there is a non-vanishing chance that the community (under the alternative) does not contain any edges.  Throughout the paper, we assume that $n$ and $p_0$ are both known, and discuss the situation where they are unknown in \secref{discussion}.

Define 
\beq \label{a}
\alpha = \frac{\log \lambda_0}{\log(N/n)} \ ,
\eeq
which varies with $N$, and notice that $p_0= \frac{\lambda_0}N$ with $\lambda_0 = \big(\frac{N}n\big)^{\alpha}$. 
The dense regime considered in \citep{subgragh_detection} corresponds to $\liminf \alpha \ge 1$.  Here we focus on the sparse regime where $\limsup \alpha < 1$.  The case where $\alpha \to 0$ includes the Poisson regime where $\lambda_0$ is constant. 

Recall that $\cG = (\cV, \cE)$ is the (undirected, unweighted) graph that we observe, and for $S \subset \cV$, let $\cG_S$ denote the subgraph induced by $S$ in $\cG$.  

We use standard notation such as $a_N \sim b_N$ when $a_N/b_N \to 1$; $a_N = o(b_N)$ when $a_N/b_N \to 0$; $a_N = O(b_N)$, or equivalently $a_N\prec b_N$, when $\limsup_N |a_N/b_N| < \infty$; 
$a_N \asymp b_N$ when $a_N = O(b_N)$ and $b_N = O(a_N)$.
We extend this notation to random variables.  For example, if $A_N$ and $B_N$ are random variables, then $A_N \sim B_N$ if $A_N/B_N \to 1$ in probability.

For $x \in \bbR$, define $x_+=x\vee 0$ and $x_-=(-x) \vee 0$, which are the positive and negative parts of $x$.
For an integer $n$, let 
\beq \label{nn}
\nn = \binom{n}2 = \frac{n(n-1)}2 \ .
\eeq  

Because of its importance in describing the tails of the binomial distribution, the following function --- which is the relative entropy or Kullback-Leibler divergence of ${\rm Bern}(q)$ to ${\rm Bern}(p)$ --- will appear in our results:
\beq \label{H}
H_p(q) = q \log \left(\frac{q}p\right) + (1-q) \log \left(\frac{1 -q}{1 -p}\right), \quad p, q \in (0,1).
\eeq
We let $H(q)$ denote $H_{p_0}(q)$.

\subsection{Calibration of a test}
We say that the test that rejects for large values of a (real-valued) statistic $T = T_N(\bW_N)$ is asymptotically powerful if there is a critical value $t = t(N)$ such that the test $\{T \ge t\}$ has risk \eqref{gamma} tending to 0.  The choice of $t$ that makes this possible may depend on $p_1$.  In practice, $t$ is chosen to control the probability of type I error, which does not necessitate knowledge of $p_1$ as long as $T$ itself does not depend on $p_1$, which is the case of all the tests we consider here.  Similarly, we say that the test is asymptotically powerless if, for any sequence of reals $t = t(N)$, the risk of the test $\{T \ge t\}$ is at least 1 in the limit.

We prefer to leave the critical values implicit as their complicated expressions do not offer any insight into the theoretical
difficulty or the practice of testing for the presence of a dense subgraph. Indeed, if a method can run efficiently, then most practitioners will want to calibrate it by simulation (permutation or parametric bootstrap, when $p_0$ is unknown). Besides, the interested reader will be able to obtain the (theoretical) critical values by a cursory examination of the proofs.

\subsection{Some general results}

Remember the definition of the entropy function in \eqref{H}.  
The following is a simple concentration inequality for the binomial distribution.

\begin{lem}[Chernoff's bound] \label{lem:chernoff}
For any positive integer $n$, any $q, p \in (0,1)$, we have
\beq \label{chernoff}
\pr{\Bin(n, p) \ge q n} \leq \exp\left(- n H_p(q) \right).
\eeq
\end{lem}

Here are some asymptotics for the entropy function. 
\begin{lem} \label{lem:H}
Define $h(x) = x \log x - x + 1$.
For $0 < p \le q < 1$, we have
\[
0\leq H_{p}(q) - p \, h(q/p) \leq  O\Big(\frac{q^2}{1-q}\Big) \ .
\]
\end{lem}

The following are standard bounds on the binomial coefficients.  Recall that $e = \exp(1)$.

\begin{lem} \label{lem:binom}
For any integers $1 \le k \le n$,
\beq \label{binom}
\left(\frac{n}k\right)^k \le {n \choose k} \le \left(\frac{e n}k\right)^k \ .
\eeq
\end{lem}

Let ${\rm Hyp}(N, m, n)$ denotes the hypergeometric distribution counting the number of red balls in $n$ draws from an urn containing $m$ red balls out of~$N$.

\begin{lem} \label{lem:hyper}
${\rm Hyp}(N, m, n)$ is stochastically smaller than ${\rm Bin}(n, \rho)$, where $\rho := \frac{m}{N-m}$.
\end{lem}

\section{Some near-optimal tests} \label{sec:tests}

In this section we consider several tests and establish their performances.  We start by recalling the result we obtained for the total degree test, based on \eqref{total-stat}, in our previous work \citep{subgragh_detection}.  Recalling the definition of $\lambda_0$ and $\lambda_1$ in \eqref{lambda}, define 
\beq \label{zeta}
\zeta := \frac{(p_1 - p_0)^2}{p_0} \frac{n^4}{N^2} = \frac{\left(\lambda_1-\lambda_0n/N\right)^2}{\lambda_0}\frac{n^2}{N} \ .
\eeq

\begin{prp}[Total degree test] \label{prp:total}
The total degree test is asymptotically powerful if $\zeta \to \infty$, and asymptotically powerless if $\zeta \to 0$.
\end{prp}

In view of \prpref{total}, the setting becomes truly interesting when $\zeta \to 0$, which ensures that the naive total degree test is indeed powerless.

\subsection{The broad scan test} \label{sec:broad}

In the denser regimes that we considered in \citep{subgragh_detection}, the (standard) scan test based on $W_n^*$ defined in \eqref{scan-stat} played a major role.  In the sparser regimes we consider here, the broad scan test based on $W_n^\ddag$ defined in \eqref{new-scan} has more power.  
Assume that $\liminf \lambda_1 > 1$, so that $\cG_S$ is supercritical under $H_S$.  Then it is preferable to scan over the largest connected component in $\cG_S$ rather than scan $\cG_S$ itself.  

\begin{lem} \label{lem:Cmax}
For any $\lambda>1$, let $\eta_\lambda$ denote the smallest solution of the equation $\eta = \exp(\lambda(\eta-1))$.  Let $\cC_m$ denote a largest connected component in $\bbG(m, \lambda/m)$ and assume that $\lambda > 1$ is fixed.  Then, in probability, $|\cC_m| \sim (1-\eta_{\lambda}) m$ and $W_{\cC_m} \sim \frac{\lambda}2 (1 - \eta_{\lambda}^2) m$.
\end{lem}

\begin{proof}
The bounds on the number of vertices in the giant component is well-known \citep[Th.~4.8]{remco:lecture}, while the lower bound on the number of edges comes from \citep[Note 5]{Pittel2005127}. 
\end{proof}

By \lemref{Cmax}, most of the edges of $\cG_S$ lie in its giant component, which is of size roughly $(1-\eta_{\lambda_1}) n$.  This informally explains why a test based on $W^*_{n(1-\eta_{\lambda_1})}$ is more promising that the standard scan test based on $W_n^*$.  

In the details, the exact dependency of the optimal subset size to scan over seems rather intricate.  This is why in $W_n^\ddag$ we scan over subsets of size $n/u_N \le k \le n$.  (Recall that $u_N=\log\log(N/n)$, although the exact form of $u_N$ is not important.)
For any subset $S \subset \cV$, let 
\[
W_{k,S}^*=\max_{T\subset S, |T|=k} W_T \ .
\]  
Note that $W_{k,\cV}^* = W_k^*$ defined in \eqref{scan-stat}.
Recall the definition of the exponent $\alpha$ in \eqref{a}.

\begin{thm}[Broad scan test] \label{thm:broad}
The scan test based on $W_n^\ddag$ is asymptotically powerful if 
\beq \label{broad1}
\lim\sup \alpha \leq 1 \quad \text{ and } \quad \lim\inf \ (1-\alpha)\sup_{k=n/u_N}^n \frac{\E_S[W_{k,S}^*]}{k} > 1\ ;
\eeq
or 
\beq \label{broad2}
\alpha  \to 0 \quad \text{ and } \quad \lim\inf \lambda_1 > 1\ .
\eeq
\end{thm}

\medskip
Note that the quantity $\sup_{k=n/u_N}^n \E_S[W_{k,S}^*]/k$ does not depend on $p_0$ or $\alpha$. 
We shall prove in the next section that the power of the broad scan test is essentially optimal: if  
\[
\lim\sup \alpha < 1 \text{ and } \lim\sup \ (1-\alpha)\sup_{k=n/u_N}^n \E_S[W_{k,S}^*]/k < 1,\] 
or $\alpha  \to 0$ and  $\lim\sup \lambda_1 < 1$, then no test is asymptotically powerful (at least when $n^2=o(N)$, so that the total degree test is powerless). 
Regarding \eqref{broad1}, we could not get a closed-form expression of this supremum. 
Nevertheless, we show in the proof that
\beq \label{broad-lb}
\liminf \sup_{k=n/u_N}^n \frac{\E_S[W_{k,S}^*]}{k} \ge \liminf \ \frac{\lambda_1}2 (1 + \eta_{\lambda_1}) \ ,
\eeq
where $\eta_\lambda$ is defined in \lemref{Cmax}. Moreover, we show in Section \ref{sec:aux} the following upper bound.
\begin{lem}\label{lem:upper_bound_expectation_wk}
\beq  \label{broad-ub}
\liminf \sup_{k=n/u_N}^n \frac{\E_S[W_{k,S}^*]}{k} \le  \liminf  \ \frac{\lambda_1}{2} + 1+ \sqrt{1+\lambda_1}  \ . 
\eeq 
If $\lambda_1\to \infty$, then \[\sup_{k=n/u_N}^n \frac{\E_S[W_{k,S}^*]}{k} \sim \lambda_1/2 \ .\]

\end{lem}
Hence, assuming $\alpha$ and $\lambda_1$ are fixed and positive , the broad scan test is asymptotically powerful when $(1-\alpha) \frac{\lambda_1}2 (1 + \eta_{\lambda_1}) > 1$. In contrast, the scan test was proved to be asymptotically powerful when $(1-\alpha)\frac{\lambda_1}{2}>1$ \citep[Prop.~3]{subgragh_detection}, so that we have improved the bound by a factor larger than  $1+\eta_{\lambda_1}$ and smaller than  $1+2\lambda_1^{-1}(1+ \sqrt{1+\lambda_1})$. When $\alpha$ converges to one, it was proved in \citep{subgragh_detection} that the minimax detection boundary corresponds to $(1-\alpha)\lambda_1/2\sim 1$ (at least when $n^2=o(N)$). Thus, for $\alpha$ going to one, both the broad scan test and the scan test have comparable power and are essentially optimal. In the dense case, the broad scan test and the scan test have also comparable powers as shown by the next result which is the counterpart of \citep[Prop.~3]{subgragh_detection}.

\begin{prp} \label{prp:scan_broad}
Assume that $p_0$ is bounded away from one. 
The broad scan test is powerful if 
\[\liminf \frac{n H(p_1)}{2 \log (N/n)} > 1\ .\]
\end{prp}

The proof is essentially the same as the corresponding result for the scan test itself.  See \citep{subgragh_detection}.

\begin{proof}[Proof of \thmref{broad}]
First, we control $W_n^\ddag$ under the null hypothesis.
For any positive constant $c_0>0$, we shall prove that 
\begin{eqnarray}\label{eq:controlH0}
 \mathbb{P}_0\left[(1-\alpha)W_n^\ddag\geq 1+c_0\right] = o(1) \ . 
\end{eqnarray}
Under Conditions \eqref{broad1} and \eqref{broad2}, $\alpha$ is smaller than for $N$ large enough. 
Consider any integer $k \in [n/u_N, n]$, and let $q_k=2(1+c_0)/[(k-1)(1-\alpha)]$. 
Recall that $\kk = k(k-1)/2$.
Applying a union bound and Chernoff's bound (\lemref{chernoff}), we derive that
\beqn
\P_0\left[W_k^*\geq \frac{1+c_0}{1-\alpha}k\right]
&\leq& \binom{N}{k}\exp\left[-k^{(2)}H(q_k)\right] \\
&\leq& \exp\left[k \big\{\log(eN/k)-\frac{k-1}{2}H(q_k) \big\} \right]\ .
\eeqn
We apply \lemref{H} knowing that $q_k/p_0\rightarrow \infty$, and use the definition of $\alpha$ in \eqref{a}, to control the entropy as follows
\beqn
\frac{k-1}{2}H(q_k)
&\sim& \frac{k-1}{2} q_k \log\left[\frac{q_k}{p_0}\right] \\
&=& \frac{1+c_0}{1-\alpha} \big[\log(N/n) - \log \lambda_0 + O(\log u_N) \big] \\
&\sim& (1+c_0) \log(N/n) \ ,
\eeqn
since $\log(u_N)=o(\log(N/n))$.   
 Consequently, 
\[\P_0\left[W_k^*\geq \frac{1+c_0}{1-\alpha}k\right]\leq \exp\left[-kc_0\log(N/n) (1+o(1))\right]\ ,\]
where the $o(1)$ is uniform with respect to $k$. Applying a union bound, we conclude that 
\begin{eqnarray*}
 \mathbb{P}_0\left[(1-\alpha)W_n^\ddag\geq 1+c_0\right]\leq \sum_{k=n/u_N}^n \exp\left[-kc_0\log(N/n)(1+o(1))\right]= o(1)\ .
\end{eqnarray*}

We now lower bound $W_n^\ddag$ under the alternative hypothesis.  First, assume that \eqref{broad1} holds, so that there exists a positive constant $c$ and a sequence of integers $k_n \ge n/u_N$ such that $\E_S[W_{k_n,S}^*]\geq k_n(1+c)/(1-\alpha)$ eventually.  In particular, $\E_S[W_{k_n,S}^*] \to \infty$.  
We then use \eqref{Wk-lb} in the following concentration result for $W_{k, S}^*$.

\begin{lem}\label{lem:concentration_Wk}
For an integer $0 \le k \leq n$, define $\mu^*_{k,S} = \E_S[W^*_{k,S}]$.  We have the following deviation inequalities
\beq\label{Wk-ub}
\mathbb{P}_S\left[W^*_{k,S}\geq \mu^*_{k,S} + t\right] \leq \exp\left[-\frac{\log(2)}{4}\left\{t\wedge \frac{t^2}{8 \mu^*_{k,S}}\right\}\right]\ ,  
\quad \forall t>8\left(1\vee \sqrt{\mu^*_{k,S}}\right)\ ;
\eeq
\beq\label{Wk-lb}
\mathbb{P}_S\left[W^*_{k,S}\leq \mu^*_{k,S} - t\right] \leq \exp\left[-\log(2)\frac{t^2}{8 \mu^*_{k,S}}\right]\ , 
\quad \forall t>4 \sqrt{\mu^*_{k,S}}\ .
\eeq
\end{lem}

It follows from \lemref{concentration_Wk} that, with probability going to one under $\P_S$, 
\[W_n^\ddag\geq \frac{W^*_{k_n}}{k_n}\geq \frac{W^*_{k_n,S}}{k_n}\geq  \frac{1+c/2}{1-\alpha}\ .\]
Taking $c_0= c/4$ in \eqref{eq:controlH0} allows us to conclude that the test based on $W_n^\ddag$ with threshold $\frac{1+c/2}{1-\alpha}$ is asymptotically powerful.

\bigskip
Now, assume that \eqref{broad2} holds.  Because $W_n^\ddag$ is stochastically increasing in $\lambda_1$ under $\P_S$, we may assume that $\lambda_1 > 1$ is fixed.  We use a different strategy which amounts to scanning the largest connected component of $\cG_S$.  Let $\Cmax^S$ be a largest connected component of $\cG_S$.

For a small $c > 0$ to be chosen later, assume that $(1-c)n(1-\eta_{\lambda_1}) \le  |\Cmax^S| \le (1+c)n(1-\eta_{\lambda_1})$ and $W_{\Cmax^S} \ge (1-c) \frac{n \lambda_1}2 (1 - \eta_{\lambda_1}^2)$, which happens with high probability under $\P_S$ by \lemref{Cmax}. 
Note that, because $\lambda_1 > 1$, we have $\eta_{\lambda_1} < 1$, and therefore $|\Cmax^S| \asymp n$.  Consequently, when computing $W_n^\ddag$ we scan $\Cmax^S$, implying that
\[
W_n^\ddag 
\ge \frac{W_{\Cmax^S}}{|\Cmax^S|}
\geq \frac{(1-c)  \frac{\lambda_1}2 (1- \eta_{\lambda_1}^2) n }{(1+c)(1-\eta_{\lambda_1})n} 
\ge \frac{1-c}{1+c} \, \frac{\lambda_1}2 (1+\eta_{\lambda_1}). 
\label{eq:lower_bound_Wk}
\]
Since $c$ above may be taken as small as we wish, and in view of \eqref{eq:controlH0}, it suffices to show that 
$\lambda_1(1+\eta_{\lambda_1})> 2\ .$
Since $\eta_{\lambda}$ converges to one when $\lambda$ goes to one, we have $\lim_{\lambda\rightarrow 1}\lambda(1+\eta_{\lambda})=2$. Consequently, it suffices to show that the function $f:\lambda \mapsto \lambda(1+\eta_{\lambda})$ is increasing on $(1,\infty)$. By definition of $\eta_{\lambda}$, we have $\eta_{\lambda}<1/\lambda$ (since $e^{-\lambda}<1/\lambda$) and  $\eta'(\lambda)=\eta_{\lambda}(\eta_{\lambda}-1)/(1-\lambda\eta_{\lambda})$. 
Consequently, $f'(\lambda)= 2+ \frac{\eta_{\lambda}-1}{1-\lambda\eta_{\lambda}}$. Hence, $f'(\lambda)$ is positive if $\eta_{\lambda}<(2\lambda-1)^{-1}:=a_{\lambda}$. Recall that $\eta_{\lambda}$ is the smallest solution of the equation $x=\exp[\lambda(x-1)]$, the largest solution being $x=1$. Furthermore, we have $x\geq \exp[\lambda(x-1)]$ for any $x\in[\eta_{\lambda},1]$. 
To conclude, it suffices to prove $a_{\lambda}>e^{\lambda(a_{\lambda}-1)}$.  This last bound is equivalent to 
\[\lambda-\frac{1}{2}-\frac{1}{2(2\lambda-1)}-\log(2\lambda-1)>0\ .\]
The function on the LHS is null for $\lambda=1$. Furthermore, its derivative $\frac{4(\lambda-1)^2}{(2\lambda-1)^2}$
is positive for $\lambda>1$, which allows us to conclude.
\end{proof}

\begin{proof}[Proof of \lemref{concentration_Wk}]
The proof is based on moment bounds for functions of independent random variables due to \cite{MR2123200} that generalize the Efron-Stein inequality.

Recall that $\cG_S = (S, \cE_S)$ is the subgraph induced by $S$.
Fix some integer $k \in [0,n]$. For any $(i, j) \in \cE_S$, define the graph $\mathcal{G}_S^{(i,j)}$ by removing $(i,j)$ from the edge set of $\mathcal{G}_S$.  Let $W_{T}^{(i,j)}$ be defined as $W_{T}$ but computed on $\mathcal{G}_S^{(i,j)}$, and then let $W_{k,S}^{* (i,j)} = \max_{T \subset S, |T| = k} W_{T}^{(i,j)}$.  Observe that $0\leq W_{k,S}^{*}- W_{k,S}^{*(i,j)}\leq 1 $ and that $W_{k,S}^{*(i,j)}$ is a measurable function of $\cE_S^{(i,j)}$, the edges set of $\cG_S^{(i,j)}$.
Let $T^* \subset S$ be a subset of size $k$ such that $W_{k,S}^{*}=W_{T^*}$.  Then, we have
\[\sum_{(i,j) \in \cE_S} \big(W_{k,S}^{*}- W_{k,S}^{*(i,j)} \big) \leq \sum_{(i,j) \in \cE_S} \big(W_{T^*}- W_{T^*}^{(i,j)} \big) =W_{T^*}= W_{k,S}^{*}\ ,\]
where the first equality comes from the fact that $W_{T^*}- W_{T^*}^{(i,j)} = \IND{(i,j) \in \cE_T}$. 

Applying \citep[Cor.~1]{MR2123200}, we derive that, for any real $q\geq 2$,
\begin{eqnarray*}
 \left[\E_S\left\{\left(W^*_{k,S}-\E_S[W^*_{k,S}]\right)_+^q\right\}\right]^{1/q}&\leq& \sqrt{2q\E_S[W^*_{k,S}]} + q \ ;\\
\left[\E_S\left\{\left(W^*_{k,S}-\E_S[W^*_{k,S}]\right)_-^q\right\}\right]^{1/q}&\leq& \sqrt{2q\E_S[W^*_{k,S}]} \ .
\end{eqnarray*}
Take some $t>8(1\vee \sqrt{\E_S[W^*_{k,S}]})$. For any $q\geq 2$, we have by Markov's inequality
\[ \mathbb{P}_S\left[W^*_{k,S}\geq \E_S[W^*_{k,S}]+ t\right]\leq \left(\frac{ \sqrt{2q\E_S[W^*_{k,S}]} + q}{t}\right)^q\ . \]
The choice $q= \frac{t}{4}\wedge \frac{t^2}{32\E_S[W^*_{k,S}]}$ is larger than $2$ and leads to \eqref{Wk-ub}. 
Similarly, if take some $t>4 \sqrt{\E_S[W^*_{k,S}]}$, and apply Markov's inequality, we get
\[ \mathbb{P}_S\left[W^*_{k,S}\le \E_S[W^*_{k,S}] - t\right]\leq \left(\frac{ \sqrt{2q\E_S[W^*_{k,S}]}}{t}\right)^q\ . \]
The choice $q=\frac{t^2}{8\E_S[W^*_{k,S}]} \ge 2$ leads to \eqref{Wk-lb}.
\end{proof}

\subsection{The largest connected component} \label{sec:CC}

This test rejects for large values of the size (number of nodes) of the largest connected component in $\cG$, which we denoted $\Cmax$.  

\subsubsection{Subcritical regime} 
We first study that test in the subcritical regime where $\limsup \lambda_0 < 1$.
Define 
\beq \label{I}
I_{\lambda}= \lambda -1 - \log(\lambda)\ .
\eeq

\begin{thm}[Subcritical largest connected component test]\label{thm:CC-sub}
Assume that $\log\log(N)=o(\log n)$, $\lim\sup \lambda_0 <1$, and $I^{-1}_{\lambda_0}\log(N)\rightarrow \infty$.
  The largest connected  component test is asymptotically powerful when $\lim\inf \lambda_1>1$ or
\beq\label{eq:hyp_CC-sub}
\lambda_0\leq \lambda_1e^{1-\lambda_1}\text{ for $n$ large enough}\quad \text{ and }\quad  \lim \inf \ \frac{I_{\lambda_0}}{\lambda_0+ I_{\lambda_1}-\lambda_0e^{I_{\lambda_1}}}\frac{\log(n)}{\log(N)}> 1\ . 
\eeq
If we further assume that $n^2=o(N)$, then the largest connected component test is asymptotically powerless  when $\lambda_1<1$ for all $n$ and 
\beq\label{eq:hyp_CC-sub-min}
 \lambda_0\geq \lambda_1e^{1-\lambda_1}\text{ for $n$ large enough}\quad \text{ or }\quad  \lim \sup \ \frac{I_{\lambda_0}}{\lambda_0+ I_{\lambda_1}-\lambda_0e^{I_{\lambda_1}}}\frac{\log(n)}{\log(N)}< 1\ . 
\eeq
\end{thm}

If we  assume that both $\lambda_0$ and $\lambda_1$ go to zero, then Condition \eqref{eq:hyp_CC-sub} is equivalent to 
\beq \label{eq:hyp_CC-sub_lambda_small}
\lim \inf \ \frac{I_{\lambda_0}}{I_{\lambda_1}}\frac{\log(n)}{\log(N)}> 1\ ,
\eeq
which corresponds to the optimal detection boundary in this setting, as shown in \thmref{lower}.

The technical hypothesis $\log\log(N)=o(\log n)$ is only used for convenience when analyzing the critical behavior $\lambda_1\rightarrow 1$.  The condition $I^{-1}_{\lambda_0}\log(N)\rightarrow \infty$ implies that $\lambda_0$ can only converge to zero slower than any power of $N$.  Although it is possible to analyze the test in the very sparse setting where $\lambda_0$ goes to zero polynomially fast, we did not do so to keep the exposition focused on the more interesting regimes.

\begin{proof}[Proof of \thmref{CC-sub}]
That the test is powerful when $\liminf \lambda_1 > 1$ derives from the well-known phase transition phenomenon of Erd\"os-R\'enyi graphs.  
Let $\cC_m$ denote a largest connected component of $\bbG(m, \lambda/m)$ and assume that $\lambda \in (0, \infty)$ is fixed.  By  \citep[Th.~4.8, Th.~4.4, Th.~4.5]{remco:lecture} in probability, we know that 
\[
|\cC_m| \sim 
\begin{cases}
I_\lambda^{-1} \log m, & \text{if } \lambda < 1 \ ; \\
(1-\eta_\lambda) m, & \text{if } \lambda > 1 \ ,
\end{cases} 
\]
where $\eta_{\lambda}$ is defined as in Lemma \ref{lem:Cmax}. When $\lambda > 1$,  the result is actually contained in \lemref{Cmax}. 

Hence, under the null with $\limsup \lambda_0 < 1$, the largest connected component of $\cG$ is of order $\log(N)$ with probability going to one.  Under the alternative $H_S$ with $\lim\inf \lambda_1 >1$, the graph $\mathcal{G}_S$ contains a giant connected component whose size of order $n$ with probability going to one.  Recalling that $\log(N)=o(n)$ allows us to conclude. 

\medskip

Now suppose that \eqref{eq:hyp_CC-sub} holds. 
We assume that the sequence $\lambda_1$ is always smaller or equal to $1$, that $I_{\lambda_1}^{-1}=O\left( \log(n)/\log(N)\right)$  and that $\log(I_{\lambda_1}^{-1} \vee 1)=o(\log n)$, meaning that $\lambda_1$ does not converge too fast to 1.  We may do so while keeping Condition \eqref{eq:hyp_CC-sub} true because the distribution of $|\Cmax|$ under $\P_S$ is stochastically increasing with $\lambda_1$, because 
$\lim\sup \lambda_0<1$, $I_{\lambda_1}+\lambda_0-\lambda_0e^{I_{\lambda_1}}\sim I_{\lambda_1}(1-\lambda_0)$ for $\lambda_1\to 1$, and because $\log\log(N)=o(\log n)$.

By hypothesis \eqref{eq:hyp_CC-sub}, there exists a constant $c'>0$, such that 
\[
\tau := \lim \inf\frac{I_{\lambda_0}\log(n)}{(I_{\lambda_1}+\lambda_0-\lambda_0e^{I_{\lambda_1}})\log(N)}\geq 1+c' \ .
\]

To upper-bound the size of $\Cmax$ under $\P_0$, we use the following. 
 
\begin{lem}\label{lem:subcritic_cluster_upper}
Let $\cC_m$ denote a largest connected component of $\bbG(m, \lambda/m)$ and assume that $\lambda < 1$ for all $m$ and $\log[I_{\lambda}^{-1}\vee 1]=o(\log(m))$.  Then, for any sequence $u_m$ satisfying 
\[\lim\inf \frac{u_m I_{\lambda}}{\log m}>1\ ,\]
we have
\[ \mathbb{P}(|\mathcal{C}_m|\geq u_m)=o(1)\ .\]
\end{lem}
\begin{proof}
This lemma is a slightly modified version of \citep[Th.~4.4]{remco:lecture}, the main difference being that $\lambda$ was fixed in the original statement. Details are omitted. 
\end{proof}

Define $c= (c'\wedge 1)/4$. 
Applying \lemref{subcritic_cluster_upper}, $|\Cmax|\leq t_0 := I_{\lambda_0}^{-1}\log(N)(1+c)$, with probability going to one under $\P_0$. 

We now need to lower-bound the size of $\Cmax$ under $\P_S$.
Define
\beqn
k_0&=&  (1-c)\log(n)\big[I_{\lambda_1}+\lambda_0-\lambda_0e^{I_{\lambda_1}} \big]^{-1}\ , \quad k= \lceil k_0\rceil\ ,  \\
q_0&= & (1-c)\log(n)\frac{1-\lambda_0e^{I_{\lambda_1}}}{I_{\lambda_1}+\lambda_0-\lambda_0e^{I_{\lambda_1}} }\ , \quad  q=\lfloor q_0\rfloor \ . 
\eeqn
The denominator of $k_0$ is positive  since $\lambda_0e^{I_{\lambda_1}}\leq 1$ and 
\beq\label{eq:lb_denominator}
I_{\lambda_1}+\lambda_0-\lambda_0e^{I_{\lambda_1}}\geq  I_{\lambda_1}+e^{-I_{\lambda_1}}\left(1-e^{I_{\lambda_1}}\right)= I_{e^{-I_{\lambda_1}}}>0\ .
\eeq
We note that $k = O(\log n)$, unless the denominator of $k_0$ goes to zero, which is only possible when $I_{\lambda_1}$ goes to zero (implying $\lambda_1 \to 1$), in which case  
\beq\label{eq:upper_bound_k}
k \sim \log(n)[I_{\lambda_1}(1-\lambda_0)]^{-1} = O\left[I_{\lambda_1}^{-1}\vee 1\right]\log(n)= O\left[\log(N)\right]\ ,
\eeq
since, in this case, \eqref{eq:hyp_CC-sub} implies that $I_{\lambda_1}^{-1}=O\left( \log(n)/\log(N)\right)$, and $\lim\sup \lambda_0<1$ by assumption.  So \eqref{eq:upper_bound_k} holds in any case.

We shall prove that among the connected components of $\cG_S$ of size larger than $q$, there exists at least one component whose size in $\cG$ is larger than $k$.
By definition of $c$, we have  $\liminf k/t_0\geq \tau (1-c)/(1+c) \ge (1+c')(1-c)/(1+c) > 1$, and the connected component test is therefore powerful.
The main arguments  rely on the second moment method and on the comparison between cluster sizes and branching processes.
Before that, recall that $t_0\to \infty$, so that $\log(n)I^{-1}_{e^{-I_{\lambda_1}}}\succ k_0\to \infty$, which in turn implies $I_{\lambda_1}=o\left(\log(n)\right)$.

\begin{lem}\label{lem:subcritic_cluster_lower}
Fix any $c>0$. Consider the distribution $\bbG(m, \lambda/m)$ and assume that $\lambda$ satisfies
\[
\lim\sup \lambda \le 1,\qquad \log\left[I^{-1}_{\lambda}\vee 1\right]= o\left(\log(m)\right)\ , \qquad I_{\lambda}^{-1}\log m\rightarrow \infty\ .
\]
For any sequence $q= a\log(m)$ with $a\leq I_{\lambda}^{-1}(1-c)$, let $Z_{\geq q}$ denote the number of nodes belonging to a connected component whose size is larger than $q$. With probability going to one, we have 
\beq \label{eq:lower_bound_Z}
Z_{\geq q}\geq m^{1-aI_{\lambda}-o(1)}\ . 
\eeq
\end{lem}

\begin{proof}
This lemma is a simple extension of the second moment method argument (Equations (4.3.34) and (4.3.35)) in the proof of \citep[Th.~4.5]{remco:lecture}, where $\lambda$ is fixed, while here it may vary with $m$, and in particular, may converge to 1. We leave the details to the reader. 
\end{proof}

Observe that
\[
\frac{q}{(1-c) I_{\lambda_1}^{-1}\log(n)}\leq \frac{I_{\lambda_1}-\lambda_0I_{\lambda_1}e^{I_{\lambda_1}}}{I_{\lambda_1}+\lambda_0-\lambda_0e^{I_{\lambda_1}}} \leq  1 - \lambda_0 \, \frac{1 - e^{I_{\lambda_1}} + I_{\lambda_1}e^{I_{\lambda_1}}}{I_{\lambda_1}+\lambda_0-\lambda_0e^{I_{\lambda_1}}}  \le 1 \ ,
\]
using the fact that $xe^{x}-e^x+1\geq 0$ for any $x \ge 0$.
Thus, we can apply Lemma \ref{lem:subcritic_cluster_lower} to $\cG_S$. 
And by Lemma \ref{lem:subcritic_cluster_upper}, the largest connected component of $\cG_S$ has size smaller than $2I_{\lambda_1}^{-1}\log(n)$ with probability tending to one.  Hence, $\cG_S$ contains more than 
\[
\frac{n^{1 +o(1)} e^{-qI_{\lambda_1}}}{2 I_{\lambda_1}^{-1}\log n} = ne^{-qI_{\lambda_1}-o(\log(n))} 
\]
connected components of size larger than $q$.  (We used the fact that $\log(I_{\lambda_1}^{-1} \vee 1)=o(\log n)$.)  
If $k_0-q_0\leq 1$, then applying Lemma  \ref{lem:subcritic_cluster_lower} to $q+2$ (instead of $q$) allows us to conclude that there exists a connected component of size at least $k$. This is why we assume in the following  that $\lim\inf k_0-q_0>1$.  By definition of $k_0$ and $q_0$, $k_0-q_0\geq 1$, implies that 
\[\log(n)\lambda_0\geq \frac{1}{1-c} e^{-I_{\lambda_1}}\left(I_{\lambda_1}+\lambda_0-\lambda_0e^{I_{\lambda_1}}\right)\geq \frac{1}{1-c}e^{-I_{\lambda_1}}I_{e^{-I_{\lambda_1}}} \]
by \eqref{eq:lb_denominator}.
Thus, $\lim\inf k_0-q_0>1$ implies that for $n$ large enough  $\log(n)\lambda_0\geq \lambda_1 I_{\lambda_1}e$ and consequently 
\beq\label{eq:upper_bound_I0}
I_{\lambda_0}\leq O(1)-\log(\lambda_0) \leq  o\left(\log(n)\right)+ I_{\lambda_1}+\log\left[I^{-1}_{e^{-I_{\lambda_1}}}\right]= o(\log(n))
\eeq
since $I_{\lambda_1}=o(\log(n))$,  $-\log(I_{\lambda_1})\leq o(\log(n))$ and  $I^{-1}_{e^{-I_{\lambda}}}=O\left[(e^{-I_{\lambda}} -1)^{-2}\right]=O\left[I_{\lambda}^{-2}\right]$.

Let  $\{\cC_S^{(i)},\ i\in \cI\}$ denote the collection of connected components of size larger than $q$ in $\cG_S$. For any such component $\cC_S^{(i)}$, we extract any subconnected component $\tilde{\cC}_S^{(i)}$ of size $q$. Recall that, with probability going to one: 
\beq\label{eq:lower_bound_cI}
|\cI|\geq n^{1-o(1)}e^{-qI_{\lambda_1}} \ .
\eeq
For any node $x$, let $\cC(x)$ denote the connected component of $x$ in $\cG$, and let $\cC_{-S}(x)$ denote the connected component of $x$ in the graph $\cG_{-S}$ where all the edges in $\cG_S$ have been removed.  Then, let \[U_i:=\bigcup_{x\in \tilde{\cC}_S^{(i)}}\cC_{-S}(x)\ ,\ i\in \cI\ ;\quad \quad V= \sum_{i\in \cI}\IND{ |U_i|\geq k}\ .\]
Since $V\geq 1$ implies that the largest connected component of $\cG$ is larger than $k$, it suffices to prove that $V$ is larger than one with probability going to one.
Observe that conditionally to $|\cI|$, the distribution of $(|U_i|$, $i\in \cI$) is independent of $\cG_S$.  Again, we use a second moment method based on a stochastic comparison between connected components and binomial branching processes.

\begin{lem}\label{lem:second_moment}
The following bounds hold
\begin{eqnarray}\nonumber
\P_S[|U_i|\geq k]& \geq& \left(\frac{k}{k-q}\right)^{k-q}e^{-\lambda_0q -I_{\lambda_0}(k-q)}n^{-o(1)} \ , \\
\Var_S[V|\cG_{S}]&\leq& |\cI|\P_S[|U_i|\geq k]+ \frac{|\cI|^2 q^2}{N}\E_S[|U_i|\IND{U_i\geq k}] \ ,\label{eq:ub_varV}\\
\P_S[|U_i|\geq k]& \leq& \E_S[|U_i|\IND{U_i\geq k}] \leq   \left(\frac{k}{k-q}\right)^{k-q}e^{-\lambda_0q -I_{\lambda_0}(k-q)}n^{o(1)}\label{eq:ub_Ui}\ .
\end{eqnarray}
\end{lem}

Before proceeding to the proof of Lemma \ref{lem:second_moment}, we finish proving that $V\geq 1$ with probability going to one.
Let  define $\mu_k:= \left(\frac{k}{k-q}\right)^{k-q}e^{-\lambda_0q -I_{\lambda_0}(k-q)}$.
Applying  Chebyshev inequality, we derive from Lemma \ref{lem:second_moment}
\beqn
V \geq |\cI| \mu_k n^{-o(1)} - O_{\P_S}\left[\left(|\cI| \mu_k\right)^{1/2}n^{o(1)}\right]- O_{\P_S}\left[|\cI| (\mu_k/N)^{1/2}n^{o(1)}\right]\ .
\eeqn	
In order to conclude, we only to need to prove that $|\cI|\mu_k\geq n^{c-o(1)}$ since  $\left(|\cI| \mu_k\right)^{1/2}/|\cI| (\mu_k/N)^{1/2} = \sqrt{N/|\cI|} \ge 1$.	
%First, we consider $|\cI|\mu_k$. 
Relying on \eqref{eq:lower_bound_cI}, we derive
\beqn
|\cI|\mu_k&\geq & n^{1-o(1)} \left(\frac{k}{k-q}\right)^{k-q}e^{-\lambda_0q-qI_{\lambda_1} -I_{\lambda_0}(k-q)}\\
&\geq & n^{1-o(1)}   \left(\frac{k_0}{k_0-q_0}\right)^{k_0-q_0}e^{-\lambda_0q_0-q_0I_{\lambda_1} -I_{\lambda_0}(k_0-q_0)-2I_{\lambda_0}}\\
&\geq & n^{1-o(1)} \lambda_0^{-(k_0-q_0)}e^{-\lambda_0q_0-k_0I_{\lambda_1}-I_{\lambda_0}(k_0-q_0)} \\
&\geq & n^{1-o(1)}e^{-k_0\lambda_0-k_0I_{\lambda_1}}e^{k_0-q_0}\\
&\geq & n^{1-o(1)}\exp\left[-k_0\left(\lambda_0+ I_{\lambda_1}- \lambda_0 e^{I_{\lambda_1}}\right) \right]= n^{c-o(1)}\ ,
\eeqn
where we use  \eqref{eq:upper_bound_I0} and $\frac{k_0}{k_0-q_0}= \lambda_0^{-1}e^{-I_{\lambda_1}}$ in the third line, the definition $I_{\lambda_0}=\lambda_0-\log(\lambda_0)-1$ in the fourth line, and the definitions of $k_0$ and $q_0$ in the last line.

\begin{proof}[Proof of Lemma \ref{lem:second_moment}]

We shall need the two following lemmas.

\begin{lem}[Upper bound on the cluster sizes] \label{lem:stochastic_domination}
Consider the distribution $\bbG(m, \lambda/m)$ and a collection $\cJ$ of nodes. For each $k\geq |\cJ|$, 
\beqn
\P\left[|\cup_{x\in \cJ}\cC(x)|\geq k\right]\leq \P_{m,\lambda/m}\left[T_1+\ldots + T_{|\cJ|}\geq k\right]\ ,
\eeqn
where $T_1,T_2,\ldots$ denote the total progenies of i.i.d.~binomial branching processes with parameters $m$ and $\lambda/m$.
For each $|\cJ|\leq k\leq m$, 
\beqn
\P\left[|\cup_{x\in \cJ}C(x)|\geq k\right]\geq \P_{m-k,\lambda/m}\left[T_1+\ldots + T_{|\cJ|}\geq k\right]\ ,
\eeqn
where $T_1,T_2,\ldots$ denote the total progenies of i.i.d. binomial branching processes with parameters $m-k$ and $\lambda/m$.
\end{lem}

Lemma \ref{lem:stochastic_domination}  is a slightly modified version of \citep[Th.~4.2 and 4.3]{remco:lecture}, the only difference being that $|\cJ|=1$ in the original statement. The proof is left to the reader. The following result is proved in \cite[Sec.~3.5]{remco:lecture}.

\begin{lem}[Law of the total progeny]\label{lem:total_progeny}
 Let $T_1,\ldots, T_r$ denote the total progenies of $r$ i.i.d.~branching processes with offspring distribution $X$. Then,
\[\P\left[T_1+\ldots + T_r= k\right]= \frac{r}{k}\P\left[X_1+\ldots +X_k=k-r\right]\ , \]
where $(X_i)$, $i=1,\ldots, k$ are i.i.d. copies of of $X$.
\end{lem}

Consider any  subset $\cJ$ of node of size $q$. The distribution $|U_i|=|\bigcup_{x\in \tilde{\cC}_S^{(i)}}\cC_{-S}(x)|$ is stochastically dominated by the distribution of $Z:=|\bigcup_{x\in\cJ} \cC(x)|$ under the null hypothesis.
Let $T_q$ be sum of the total progenies of $q$ independent binomial branching processes with parameters $N-n+q-k$ and $p_0$. 
By Lemma \ref{lem:stochastic_domination}, we derive 
\[ \P_S[|U_i|\geq k] \geq \P_0[Z\geq k] \geq \P_{N-n+q-k,p_0 }[T_q\geq k]\geq \P_{N-n+q-k,p_0 }[T_q= k]\ .\]
Let $X_1,X_2,\ldots $ denote independent binomial random variables with parameters $N-n+q-k$ and $p_0$. Relying on Lemma \ref{lem:total_progeny} and the lower bound
$\binom{s}{r}\geq \frac{(s-r)^r}{r!}\geq (re)^{-1}\left(\frac{(s-r)e}{r}\right)^r$, we derive 
\beqn 
\P_{N-n+q-k,p_0 }[T_q= k]&=   & \frac{q}{k}\P_{N-n+q-k,p_0 }[X_1+\ldots + X_k= k-q]\\
& = & \frac{q}{k}\binom{k(N-n+q-k)}{k-q}p_0^{k-q}(1-p_0)^{k(N-n+q-k)-k+q} \\
&\succ& \frac{q}{k^2}\left[\frac{ek(N-n-2(q-k))}{k-q}\right]^{k-q} \left(\frac{\lambda_0}{N}\right)^{k-q} e^{-\lambda_0k-kO(n/N)}\\
&\succ & \frac{q}{k^2}e^{-I_{\lambda_0}(k-q)}e^{-\lambda_0 q}\left(\frac{k}{k-q}\right)^{k-q}e^{-kO(n/N)}\\
&\succ &  \left(\frac{k}{k-q}\right)^{k-q}e^{-\lambda_0q -I_{\lambda_0}(k-q)} n^{o(1)}\ ,
\eeqn
where  \eqref{eq:upper_bound_k} with $n\log(N)/N=o(\log(n))$ in the last line.\\

Let us now prove \eqref{eq:ub_Ui}.  The first inequality is Markov's.
For the second, by Lemma \ref{lem:stochastic_domination}, 
$U_i$ is stochastically dominated by $\tilde{T}_q$, the sum of the total progenies of $q$ independent binomial branching  processes with parameters $N$ and $p_0$, so that
\[
\E_S\left[|U_i|\IND{U_i\geq k}\right] = \sum_{r=k}^N \P_{S}[U_i\geq  r] \le \sum_{r=k}^\infty \P_{N,p_0}[\tilde{T}_q\geq   r] \le \sum_{r=k}^{\infty} r \P_{N,p_0}[\tilde{T}_q=   r] \ .
\]
We use Lemma \ref{lem:total_progeny} to control the deviation of $\tilde{T}_q$. Below $X_1,X_2,\ldots$ denote independent binomial random variables with parameter $N$ and $p_0$.  
\begin{eqnarray}\nonumber 
\sum_{r=k}^{\infty} r \P_{N,p_0}[\tilde{T}_q=   r] 
&\leq & \sum_{r=k}^{\infty}
r \frac{q}{r}\P_{N,p_0}[X_1+\ldots + X_r= r-q]\\ \label{eq:upU1}
&\leq & \sum_{r=k}^{\infty} q\exp\left[- NrH_{p_0}\left(\frac{r-q}{Nr}\right)\right]\ ,
\end{eqnarray}
by Chernoff inequality since 
\beqn
\frac{r-q}{Nr}\geq \frac{k-q}{Nk}\geq \frac{k_0-q_0}{Nk_0}= \frac{\lambda_0e^{I_{\lambda_1}}}{N}> \frac{\lambda_0}{N} = p_0\ .
\eeqn
By Lemma \ref{lem:H},  $H_{p_0}(a)\geq a\log(a/p_0)-a+p_0$. Thus, we arrive at
\begin{eqnarray}
\E_S\left[|U_i|\IND{U_i\geq k}\right] &\leq & \sum_{r=k}^{\infty} q\exp\left[ -(r-q)\log\left(\frac{r-q}{r\lambda_0}\right) +r-q - r\lambda_0 \right] \nonumber \\
&\leq & q\sum_{r=k}^{\infty} \exp[A_r]\ ,
\label{eq:upper_Ui}
\end{eqnarray}
where $A_r:=  -(r-q) I_{\lambda_0}- q\lambda_0 -(r-q)\log\left(\frac{r-q}{r}\right)$.
Differentiating the function $A_r$ with respect to $r$, we get 
\beqn
\frac{dA_r}{dr}&=& -I_{\lambda_0}-\log\left(\frac{r-q}{r}\right)-1+\frac{r-q}{r}\leq -I_{\lambda_0}-\log\left(\frac{k-q}{k}\right)-1+\frac{k-q}{k}\\
&\leq & -I_{\lambda_0}-\log\left(\frac{k_0-q_0}{k_0}\right)-1+\frac{k_0-q_0}{k_0}= -\lambda_0 -I_{\lambda_1}+\lambda_0e^{I_{\lambda_1}}\ ,
\eeqn 
which is negative as argued below the definition of $k$. Consequently, $A_r$ is a decreasing function of $r$.
 Define $r_1$ as the smallest integer such that $\log((r-q)/r)\geq -I_{\lambda_0}/2$. Since $\lim\sup \lambda_0<1$, it follows $r_1=O(q)$. Coming back to \eqref{eq:upper_Ui}, we derive
\begin{eqnarray}\nonumber
 \E_S\left[|U_i|\IND{U_i\geq k}\right] &\leq & q(r_1-k)_+\exp[A_k]  +  q\sum_{r=r_1}^{\infty}\exp[A_r]\\ \nonumber
&\leq & q e^{A_k}\left[(r_1-k)_+ +\sum_{r=r_1}^{\infty} e^{-(r-k) \left[I_{\lambda_0}- \log((r-q)/r)\right]} \right]\\ \nonumber
&\leq & q e^{A_k}\left[(r_1-k)_+ +\sum_{r=r_1}^{\infty} e^{-(r-k) I_{\lambda_0}/2} \right]\\ \label{eq:upper_bound_AK}
&\leq & e^{A_k} O(k^2) \ ,
\end{eqnarray}
since $\lim\sup\lambda_0<1$. From \eqref{eq:upper_bound_k}, we know that $k=O(\log(N))= n^{o(1)}$, which allows us to prove \eqref{eq:ub_Ui}.

\medskip
Turning to the proof of \eqref{eq:ub_varV}, we have the decomposition
\begin{eqnarray}\nonumber
\Var_S[V|\cG_{S}]&\leq& |\cI|\P_S[U_i\geq k]+\sum_{i\neq i'\in \cI}\left\{\P_S\left[|U_i|\geq k ,\, |U_{i'}|\geq k\right]- \P^2_S\left[|U_i|\geq k\right]  \right\} \\
&\leq & |\cI|\P_S[U_i\geq k]+|\cI|^2 \P_S\left[ |U_i|\geq k \ ,\ U_i\cap U_{i'}\neq \emptyset \right] \nonumber \\ && + |\cI|^2\left\{ \P_S\left[ |U_i|\geq 
k ,\,  |U_{i'}|\geq k\ , U_i\cap U_{i'}=\emptyset \right]- \P^2_S\left[|U_i|\geq k\right] \right\}\ . \label{eq:decomposition_variance}
\end{eqnarray}
The last term is nonpositive. Indeed, 
\beqn
 \lefteqn{\P_S\left[|U_i|\geq k ,\, |U_{i'}|\geq k\ , U_i\cap U_{i'}=\emptyset \right]- \P^2_S\left[|U_i|\geq k\right]} &&\\
&=& \sum_{r=k}^{N}\P_S\left[ |U_i|=r \right]\big(\P_S\left[ |U_{i'}|\geq k ,\ U_i\cap U_{i'}=\emptyset \, \big| \, |U_i|=r  \right]-   \P_S\left[|U_{i'}|\geq k \right]\big)\\
&\leq & \sum_{r=k}^{N}\P_S\left[ |U_i|=r \right]\big(\P_S\left[ |U_{i'}|\geq k  \, \big| \ U_i\cap U_{i'}=\emptyset\ , \, |U_i|=r  \right]-   \P_S\left[|U_{i'}|\geq k \right]\big)
\ ,
\eeqn
where the last difference is negative, as the graph is now smaller once we condition  on $|U_i|\geq 1$ and $U_i\cap U_{i'}=\emptyset$.   Consider the second term in \eqref{eq:decomposition_variance}:
\[
\P_S\left[  |U_i|\geq k\ ,\ U_i\cap U_{i'}\neq \emptyset \right]= \sum_{r=k}^{N}\P_S[|U_i|=r]\P_S[U_i\cap U_{i'}\neq \emptyset \, | \, |U_{i}|=r] \ . 
\]
By symmetry and a union bound, we derive 
$$\P_S[U_i\cap U_{i'}\neq \emptyset \, | \, |U_{i}|=r]\leq  q^2 \P_S[y\in \cC_{-S}(x) \, | \, |U_{i}|=r] \ ,$$ 
for some $x\in  \tilde{\cC}_S^{(i)}$ and $y\in \tilde{\cC}_S^{(i')}$. 
Since the graph $\cG_{-S}$ is not symmetric, the probability that a fixed  node $z$ belongs to $\cC_{-S}(x)$ conditionally to $|\cC_{-S}(x)|$ is smaller for  $z\in S\setminus\{i\}$ than for $z\in S^c$. It follows that
$$\P_S[y\in \cC_{-S}(x) \, | \, |U_{i}|=r]\leq \E_S\left[\left.\frac{|\cC_{-S}(x)|-1}{N-1} \right| \, |U_{i}|=r\right] \ .$$ 
Since $|\cC_{-S}(x)|\leq r$,  we  conclude
\[\P_S\left[  |U_i|\geq k\ ,\ U_i\cap U_{i'}\neq \emptyset \right]\leq \sum_{r=k}^{N}\P_S[|U_i|=r] \frac{q^2r}{N}= \frac{q^2}{N}\E_{S}[|U_i|\IND{U_i\geq k}]\ .\]

\end{proof}

Let us continue with the proof of \thmref{CC-sub}, now assuming that $\lambda_1 < 1$, that Condition \eqref{eq:hyp_CC-sub-min} holds, and that $n^2 = o(N)$.
We assume in the sequel that $I_{\lambda_1}\leq -\log(\lambda_0)$, meaning that $\lambda_1$ is not too small. We may do so while keeping Condition \eqref{eq:hyp_CC-sub-min} true, because the distribution of $|\cC_{\max}|$ under $\P_S$ is increasing with respect to $\lambda_1$ and because for $I_{\lambda_1}= -\log(\lambda_0)$, \eqref{eq:hyp_CC-sub-min}  is equivalent to $\lim\sup \log(n)/\log(N)<1$, which is always true since $n^2=o(N)$. Similarly, we assume that $I_{\lambda_1}=o(\log(n))$ while keeping Condition \eqref{eq:hyp_CC-sub-min} true since for $I_{\lambda_1}$ going to infinity, \eqref{eq:hyp_CC-sub-min} is equivalent to $\lim\sup \frac{I_{\lambda_0}\log(n)}{I_{\lambda_1}\log(N)}<1$ and since $I^{-1}_{\lambda_0}\log(N)\to \infty$. By Condition \eqref{eq:hyp_CC-sub-min}, there exists a constant $c>0$ such that 
\beq\label{eq:condition_minoration}
 \lim \sup \ \frac{I_{\lambda_0}}{\lambda_0+ I_{\lambda_1}-\lambda_0e^{I_{\lambda_1}}}\frac{\log(n)}{\log(N)}< 1-c\ .
\eeq 

We shall prove that with probability $\P_S$ going to one, the largest connected component of $\cG$ does not intersect $S$. As the distribution of the statistic under the alternative dominates the distribution under the null, this will imply that the largest connected component test is asymptotically powerless. \\
Denote $\cA$ the event that, for all $(x,y) \in S$, there is no path between $x$ and $y$ with all other nodes in $S^c$.
For any subset $T$, denote $\cC_T(x)$ the connected component of $x$ in $\cG_T$, and recall that $\cC(x)$ is a shorthand for $\cC_{\cV}(x)$.
By symmetry, we have 
\[\P_S[\cA^c]\leq n^2 \P_0[y\in \cC_{-S}(x)]\leq \P_0[y\in \cC(x)]\ , \]
since the probability of the edges outside $\cG_S$ under $\P_S$ is the same as under $\P_0$. Again, by symmetry 
\[\P_0[y\in \cC(x)]= \E_{0}[\P_0[y\in C(x)]\, | \, |\cC(x)|]\leq \E_{0}\left[\frac{|\cC(x)|}{N-1}\right]\leq \frac{1}{(N-1)(1-\lambda_0)}\ , \]
as the expected size of a cluster is dominated by the expected progeny of a branching process with parameters $N$ and $p_0$ (Lemma \ref{lem:stochastic_domination})
and the expected progeny of a subcritical branching process having mean offspring $\mu<1$ is $(1-\mu)^{-1}$ \citep[Th.~3.5]{remco:lecture}. Thus,
\beq\label{eq:upper_bound_proba_A}
\P_S[\cA^c]=O( n^2/N)=o(1)\ .
\eeq 
Define \beq\label{eq:defik} k:=(1-c)^{1/2}\log(N)I_{\lambda_0}^{-1}\ .\eeq
Since $\lim\sup \lambda_0<1$ and since $\log\log(N)=o[\log(n)]$, it follows that $k\asymp\log(N)=n^{o(1)}$.
By Lemma \ref{lem:subcritic_cluster_lower}, $|\cC_{\max}|$ is larger or equal to $k$ with probability $\P_S$ (and $\P_0$) going to one. Thus, it suffices to  prove that 
$\P_{S}[\vee_{x\in S}|\cC(x)|\geq k] \to 0$. Observe that
\[\P_{S}[\vee_{x\in S}|\cC(x)|\geq k]\leq n \P_{S}\left[\left\{|C(x)|\geq k\right\}\cap \cA\right] +\P_S[\cA^c]\ ,
\]
 so that, by \eqref{eq:upper_bound_proba_A}, we only need to prove that  $n\P_{S}\left[\left\{|C(x)|\geq k\right\}\cap \cA\right]=o(1)$. Under the event $\cA$, $\cC(x)\cap S$ is exactly the connected component $\cC_S(x)$ of $x$ in $\cG_S$.  
 Furthermore, $\cC(x)$ is the union of $\cC_{-S}(y)$  over $y\in \cC_S(x)$.
Consequently, we have the decomposition
\beqn 
\P_{S}\left[\left\{|\cC(x)|\geq k\right\}\cap \cA\right] 
&\leq& \P_S[|\cC_S(x)|\geq k] \\
&& + \sum_{q=1}^{k-1}\P_S[|\cC_{S}(x)|=q]\P_S\left[ \left.\cB_q \, \right| |\cC_S(x)|=q\right]\ ,
\eeqn
where $\cB_q:= \{|\cup_{y\in \cC_{S}(x)} \cC_{-S}(y)| \geq k\}$.
By Lemma \ref{lem:stochastic_domination}, the distribution of $|\cC_{S}(x)|$ is stochastically dominated by the total progeny distribution of a binomial branching process with parameters $(n,\lambda_1/n)$. Denote by $\cJ$ any set of nodes of size $q$. Since, conditionally to $ |\cC_S(x)|=q$, the event $\cB_q$ is increasing and only depends on the edges outside $\cG_S$, we have
\[\P_S\left[ \left.\cB_q \right| |\cC_S(x)|=q\right]\leq \P_0\left[|\cup_{y\in \cJ}\cC(y)|\geq k\right]\ ,\]
which is in turn, by Lemma \ref{lem:stochastic_domination}, smaller than the probability that the total progeny of $q$ independent branching processes with parameters $(N,\lambda_0/N)$ is larger than $k$. 
Relying on the law of the total progeny of branching processes (Lemma \ref{lem:total_progeny}) and \lemref{stochastic_domination}, we get
\beqn
\P_S[|\cC_{S}(x)|=q]
&\leq & \frac{1}{q}\P\left[\mathrm{Bin}(nq,\lambda_1/n)= q-1\right] \ ,\\
\P_S\left[ \left.\cB_q \, \right| |\cC_S(x)|=q\right]& \leq &\sum_{r=k}^{\infty}\frac{q}{r}\P\left[\mathrm{Bin}(Nr,\lambda_0/N)=  r-q\right]\ .
\eeqn
Working out the density of the binomial random variable, we derive
\beqn	
\P_S[|\cC_{S}(x)|=q]&\leq & \binom{nq}{q-1}p_1^{q-1}(1-p_1)^{nq-q+1} \prec \frac{1}{\lambda_1}e^{-I_{\lambda_1}q}\ ,\\
\eeqn
and for $q\leq (1-\lambda_0)k$, we get 
\beqn
\P_S\left[ \left.\cB_q \, \right| |\cC_S(x)|=q\right]& \leq& \frac{q}{k} \exp\left[-N k H_{p_0}\left(\frac{k-q}{N k}\right)\right]\ , 
\eeqn
which is exaclty the term \eqref{eq:upU1}, which has been proved in \eqref{eq:upper_bound_AK} to be smaller than 
\[ O(k^2)\left(\frac{k-q}{k}\right)^{k-q}  e^{-(k-q)I_{\lambda_0}-q\lambda_0} \ .\]
 Let define 
\[B_{\ell}:= e^{-I_{\lambda_1}\ell-\ell\lambda_0-(k-\ell)I_{\lambda_0}}\left(\frac{k}{k-\ell}\right)^{k-\ell}\]
 Gathering all these bounds, we get 
 \beqn
\P_{S}\left[\left\{|C(x)|\geq k\right\}\cap \cA\right]&\prec& \frac{e^{-I_{\lambda_1}k}}{\lambda_1}+ \sum_{q=\lceil (1-\lambda_0)k\rceil  }^{k-1}\frac{e^{-I_{\lambda_1}q}}{\lambda_1} +O\left(\frac{k^2}{\lambda_1}\right)\sum_{q=1}^{\lfloor(1-\lambda_0)k\rfloor}
B_q\\
&\prec & \frac{k^3}{\lambda_1}\big[e^{-I_{\lambda_1}(1-\lambda_0)k}+ \bigvee_{q=1}^k B_q\big]
\prec  n^{o(1)} \sup_{q\in[0;k]} B_q\ ,
\eeqn
where we observe that  $e^{-I_{\lambda_1}(1-\lambda_0)k}=B_{(1-\lambda_0)k}$  and we use $k=n^{o(1)}$ and $I_{\lambda_1}=o(\log(n))$.
 By differentiating $\log(B_q)$ as a function of $q$, we obtain the maximum
\beqn
\sup_{q\in[0;k]}B_q \leq \left\{\begin{array}{ccc}
e^{-kI_{\lambda_0}}\ ,&\text{ if }& \lambda_0e^{I_{\lambda_1}}>1\ ;\\
e^{-I_{\lambda_1}k}\exp\left[\lambda_0k(e^{I_{\lambda_1}}-1)\right]\ , & \text{ else .}&
                                                                                                     \end{array}\right.
\eeqn
Recall that we assume  $\lambda_0e^{I_{\lambda_1}}\leq 1$ so that 
\beqn
\P_{S}\left[\left\{|C(x)|\geq k\right\}\cap \cA\right]&\prec& n^{o(1)}\frac{1}{\lambda_1}\exp\left[-k\left\{\lambda_0+ I_{\lambda_1}- \lambda_0e^{I_{\lambda_1}}\right\}\right] \\
&\prec& n^{-(1-c)^{-1/2}+o(1)}\ ,
\eeqn
by definition \eqref{eq:defik} of $k$ and Condition \eqref{eq:condition_minoration}. We conclude that 
\[n \P_{S}\left[\left\{|C(x)|\geq k\right\}\cap \cA\right]=o(1).\]
\end{proof}

\subsubsection{Supercritical regime} 
We now briefly discuss the behavior of the largest connected component test in the supercritical regime where $\liminf \lambda_0 > 1$.  When $\lambda_0 - \log N \to \infty$, the graph $\cG$ is connected with probability tending to one under the null and under any alternative \citep[Th.~5.5]{remco:lecture}, which renders the test completely useless.  We focus on the case where $\lambda_0$ is fixed for the sake of simplicity.  In that regime, we find that, in that case, the test performs roughly as well as the total degree test --- compare \prpref{total}. 

\begin{prp}[Supercritical largest connected component test] \label{prp:CC-sup}
The largest connected component test is asymptotically powerful when $\lambda_1 > \lambda_0 > 1$ are fixed and $n^2/N \to \infty$.
\end{prp}

\begin{proof}
We keep the same notation.  Under $\P_0$, we have $|\Cmax| = (1 - \eta_{\lambda_0}) N + O(\sqrt{N})$ \citep[Th.~4.16]{remco:lecture}.  Hereafter, assume that we are under $\P_S$.  Then, by the same token, $|\Cmax^{S^\comp}| = (1 - \eta_{\lambda_0}) (N-n) + O(\sqrt{N-n})$ and $|\Cmax^{S}| = (1 - \eta_{\lambda_1}) n + O(\sqrt{n})$.  Given $\cG_S$ and $\cG_{S^\comp}$, the probability that $\Cmax^{S^\comp}$ and $\Cmax^{S}$ are connected in $\cG$ is equal to $1 - (1-p_0)^{|\Cmax^{S^\comp}| \ |\Cmax^{S}|} \to 1$ in probability, since $p_0 |\Cmax^{S^\comp}| \, |\Cmax^{S}| \asymp n$ in probability.  Hence, with probability tending to one,
\beqn
|\Cmax| 
&\ge& |\Cmax^{S^\comp}| + |\Cmax^{S}| \\
&=& (1 - \eta_{\lambda_0}) (N-n) + O(\sqrt{N}) + (1 - \eta_{\lambda_1}) n + O(\sqrt{n}) \\
&=& (1 - \eta_{\lambda_0}) N + (\eta_{\lambda_0} - \eta_{\lambda_1}) n + O(\sqrt{N}) \ ,
\eeqn
with $\eta_{\lambda_0} - \eta_{\lambda_1} > 0$ since $\lambda_1 > \lambda_0 > 1$ and $\eta_\lambda$ is strictly decreasing.  Hence, because $n \gg \sqrt{N}$ by assumption, the test that rejects when $|\Cmax| \ge (1 - \eta_{\lambda_0}) N + \frac12 (\eta_{\lambda_0} - \eta_{\lambda_1}) n$.
\end{proof}

When $\lambda_0 > 1$ is fixed, the largest connected component is of size $|\Cmax|$ satisfying 
\[\frac{|\Cmax| - (1-\eta_{\lambda_0}) N}{\sqrt{N}} \to \cN(0,1), \quad \text{under }\P_0 \ ,\]
by \citep[Th.~4.16]{remco:lecture}, while $|\Cmax|$ increases by at most $n$ under the alternative, so the test is powerless when $n = o(\sqrt{N})$.

\subsection{The number of $k$-trees} \label{sec:k-tree}

We consider the test that rejects for large values of $\tree_k$, the number of subtrees of size $k$.  This test will partially bridge the gap in constants between what the broad scan test and largest connected component test can achieve in the regime where $\lambda_0$ is constant.  
Recall the definition of $I_\lambda$ in \eqref{I}.

\begin{thm}\label{thm:k-tree}
Assume that $\lambda_1$ and $\lambda_0$ are both fixed, with $0 < \sqrt{\lambda_0/e} < \lambda_1 < 1$, and that
\beq \label{k-tree}
\limsup \frac{\log(N/n^2)}{\log n} < \frac{I_{\frac{\lambda_0}{\lambda_1e}}- I_{\sqrt{\frac{\lambda_0}{e}}}}{\big(1 - \frac{\lambda_0}{\lambda_1 e}\big) I_{\frac{\sqrt{\lambda_1}}{e}}} \ .
\eeq  
Then there is a constant $c > 0$ such that the test based on $\tree_k$ with $k = c \log n$ is asymptotically powerful.
\end{thm}

Thus, even in the supercritical  Poissonian regime with $1<\lambda_0<e$, there exist subcritical communities $\lambda_1<1$ that are asymptotically detectable with probability going to one. The condition $\lambda_1>\sqrt{\lambda_0/e}$ will be shown to be minimal in Theorem \ref{thm:lower-no}. Condition \eqref{k-tree} essentially requires that $n^2/N$ does not converge too fast to zero.   In particular, when $n = N^\kappa$, \eqref{k-tree} translates into an upper bound on $\kappa$.  We show later in Theorem \ref{thm:lower-no} that such an upper bound is unavoidable, for when $\kappa$ is too small, no test is asymptotically powerful.  Nevertheless, Condition \eqref{k-tree} is in all likelihood not optimal.

\begin{proof}[Proof of \thmref{k-tree}]
Let $\tree_k$ denote the number of subtrees of size $k$.
We first compute the expectation of $\tree_k$ under $\P_0$ using Cayley's formula.
Since $k^2=o(n)=o(N)$ and $k\rightarrow \infty$, we derive
\begin{eqnarray*}
 \E_0[\tree_k]
 &=&\sum_{|C| = k} \P_0[\text{$\cG_C$ is a tree}] \\
 &= &\binom{N}{k}k^{k-2}p_0^{k-1}(1-p_0)^{k^{(2)}-k+1} \\
&\sim& N(\lambda_0e)^k\frac{1}{\sqrt{2\pi}k^{5/2}\lambda_0}\ ,
\end{eqnarray*}
where we used the fact any $k$-tree has exactly $k-1$ edges.  The last line comes from an application of Stirling's formula.
We then bound the variance of $\tree_k$ under $\P_0$ in the following lemma, whose lengthy proof is postponed to \secref{Nkt_var0}.

\begin{lem}\label{lem:Nkt_var0}
When $\lambda_0 < e$, we have
\[ 
\Var_0[\tree_k] \prec \frac{N}{k\lambda_0}(e\lambda_0)^{k}e^{2k\sqrt{\lambda_0/e}} \ .
\]
\end{lem}

By Chebyshev's inequality, under $\P_0$, 
\[
\tree_k = \E_0 \tree_k + O\big(\Var_0(\tree_k)\big)^{1/2} \ .
\]

\medskip 
Fix $S \subset \cV$ of size $|S| = n$, and let $q$ be an integer between $1$ and $k$ chosen later.  We let $\tree_{k,S^\comp}$ denote the number of $k$-trees in $\cG_{S^\comp}$, and let $\tree_{k,S,q}$ as the number of subsets $C$ of size $k$ such that $|C\cap S|=q$ and both $\cG_{C\cap S}$ and $\cG_{C}$ are trees. 
We have $\tree_k\geq \tree_{k,S^\comp}+ \tree_{k,S,q}$. 
Therefore, by Chebyshev's inequality, under $\P_S$ 
\[
\tree_k \ge \E_S\big(\tree_{k,S^\comp}\big)+ \E_S\big({\tree_{k,S,q}}\big) + O\big(\Var_S(\tree_{k,S^\comp})\big)^{1/2} + O\big(\Var_S(\tree_{k,S,q})\big)^{1/2} \ .
\]

Noting that $\cG_{S^\comp} \sim \bbG(N-n, p_0)$, and letting $\lambda_0' = (N-n) p_0$, \lemref{Nkt_var0} implies that
\[
\Var_S[\tree_{k,S^\comp}] \prec \frac{N-n}{k\lambda_0'}(e\lambda_0')^{k}e^{2k\sqrt{\lambda_0'/e}} \sim \frac{N}{k\lambda_0}(e\lambda_0)^{k}e^{2k\sqrt{\lambda_0/e}} \ ,  
\]
because $nk=o(N)$. 
Thus,  we only need to show that, for a careful choice of $q$, 
\begin{eqnarray}
\E_S[\tree_{k,S,q}]& \gg &\E_0[\tree_k] - \E_S[\tree_{k,S^\comp}]\ , \label{cheb1}\\
\E_S^2[\tree_{k,S,q}] &\gg& \Var_S[\tree_{k,S,q}] \label{cheb2} \ , \\
\E_S^2[\tree_{k,S,q}] &\gg& \frac{N}{k\lambda_0}(e\lambda_0)^{k}e^{2k\sqrt{\lambda_0/e}} \succ \Var_0[\tree_k] \label{cheb3} \ .
\end{eqnarray}
From now on, let $q=k-\lfloor \frac{\lambda_0}{\lambda_1 e}k\rfloor$.

We use the following lemma, whose lengthy proof is postponed to \secref{alternative_tree}.
\begin{lem}\label{lem:alternative_tree}
When $q=k-\lfloor \frac{\lambda_0}{\lambda_1 e}k\rfloor$, we have
\beq \label{eq:upper_ES}
\mathbb{E}_S[\tree_{k,S,q}] \succ n \frac{\lambda_1^{k-1}e^{2k-q}}{k^{3}}\succ n (e\lambda_1)^k e^{\frac{\lambda_0}{\lambda_1e}k} \frac{1}{\lambda_1k^{3}}\ ,
\eeq
and
\beq \label{eq:upper_VARS}
\Var_S[\tree_{k,S,q}] \prec nk^2 \lambda_1^{2k-q-1}e^{4k-2q}  e^{2\frac{\sqrt{\lambda_1}}{e}q} +\frac{k^7n^2}{N}\lambda_1^{2k-2 }\lambda_0e^{4k-2q}\ .
\eeq
\end{lem}

We first prove \eqref{cheb1}, bounding 
\beqn
 \E_0[\tree_k]-\E_S[\tree_{k,S^\comp}] &= &\left(\binom{N}{k}-\binom{N-n}{k}\right) k^{k-2}p_0^{k-1}(1-p_0)^{k^{(2)}-k+1} \\
&\le &  \left(N^k- (N-n-k)^k\right) \frac{k^{k-2}}{k!} \left(\frac{\lambda_0}N\right)^{k-1} \\
&\prec & n (\lambda_0e)^k k^{-5/2}\ ,
\eeqn
since $[1-(n+k)/N]^k= 1+kn/N+o(kn/N)$ by the fact that $k=o(n)$ and $kn=o(N)$.  We also used Stirling's formula again.  
Using this bound together with \eqref{eq:upper_ES}, we derive
\beqn
\frac{\E_S[\tree_{k,S,q}]}{\E_0[\tree_k]-\E_S[\tree_{k,S^\comp}]} \succ \frac{\lambda_0}{k^{1/2}\lambda_1}\left(\frac{\lambda_1}{\lambda_0}\right)^k e^{k\frac{\lambda_0}{\lambda_1e}}= \frac{\lambda_0}{k^{1/2}\lambda_1} \exp\left[k I_{\frac{\lambda_0}{\lambda_1e}}\right] \to \infty\ ,
\eeqn
since $\lambda_0$ and $\lambda_1$ are fixed such that $\lambda_0/\lambda_1e<1$, implying that $I_{\frac{\lambda_0}{\lambda_1e}}>0$ is fixed.

Second, we prove \eqref{cheb2}.  Using \eqref{eq:upper_ES} and \eqref{eq:upper_VARS}, we have 
\beqn
\frac{\Var_S[\tree_{k,S,q}]}{\mathbb{E}^2_S[\tree_{k,S,q}]} 
&\prec& \frac{k^8}{n}\lambda_1^{-q} e^{2\frac{\sqrt{\lambda_1}}{e}q} +\frac{k^{13}}{N} \\
&\prec& \frac{k^8}{n} \exp\left[2 k \big(1 - \frac{\lambda_0}{\lambda_1 e}\big) I_{\frac{\sqrt{\lambda_1}}{e}}\right] +\frac{k^{13}}{N} \ ,
\eeqn
and the RHS goes to $0$ as long as 
\[
\limsup  \frac{k}{\log(n)} <  \frac1{2 \big(1 - \frac{\lambda_0}{\lambda_1 e}\big) I_{\frac{\sqrt{\lambda_1}}{e}}} \ .
\]

Finally, we prove \eqref{cheb3}.   Using \lemref{Nkt_var0} and \eqref{eq:upper_ES}, we have
\beqn
\frac{\Var_0[\tree_k]}{\E_S^2[\tree_{k,S,q}]} 
&\prec& \frac{N k^5}{n^2} \left(\frac{\lambda_1^2}{e \lambda_0}\right)^k e^{2k\sqrt{\frac{\lambda_0}{e}} -4k + 2q} \\
&\prec& \frac{N k^5}{n^2} \exp\left[2k\big(I_{\sqrt{\frac{\lambda_0}{e}}} - I_{\frac{\lambda_0}{\lambda_1e}}\big)\right] \ .
\eeqn
Note that $I_{\sqrt{\frac{\lambda_0}{e}}} - I_{\frac{\lambda_0}{\lambda_1e}} < 0$ is fixed, since our assumptions imply that $\frac{\lambda_0}{\lambda_1e} < \sqrt{\frac{\lambda_0}{e}} <1$ and the function $I_\lambda$ is decreasing on $(0,1)$. Thus, the RHS above goes to $0$ as long as 
\[
\liminf \frac{k}{\log(N/n^2)} >  \frac1{2 \big(I_{\frac{\lambda_0}{\lambda_1e}}- I_{\sqrt{\frac{\lambda_0}{e}}}\big)} \ .
\]
\end{proof}

\subsection{The number of triangles} \label{sec:triangle}
We recall that this test is based on the number $T$ of triangles in $\cG$.  This is an emblematic test among those based on counting patterns, as it is the simplest and the least costly to compute.  As such, the number of triangles in a graph is an important topological characteristic, with applications in the study of real-life networks.  For example, \cite{Maslov2004529} use the number of triangles to quantify the amount of clustering in the Internet. 

\begin{prp} \label{prp:triangle}
The triangle test is asymptotically powerful if 
\beq \label{triangle1}
\lim\sup \lambda_0 <\infty \quad \text{ and } \quad \lambda_1 \to \infty  \ ;
\eeq
or
\beq \label{triangle2}
\lim\inf \lambda_0>0 \quad ,\quad  \lambda_0<N/n   \quad \text{ and } \quad \frac{\lambda_1^2}{\lambda_0} \gg 1\vee \left(\frac{\lambda_0}{\sqrt{N}}\right)^{2/3} \ .
\eeq
When $\lambda_0$ and $\lambda_1$ are fixed, $T$ converges in distribution towards a Poisson distribution with parameter $\lambda_0^3/6$ under the null and $(\lambda_0^3+\lambda_1^3)/6$ under the alternative hypothesis.
In particular, the test is {\em not} asymptotically powerless if 
\beq \label{triangle3}
\limsup \lambda_0 < \infty \quad \text{ and } \quad \liminf \lambda_1 > 0 \ .
\eeq
\end{prp}

\begin{proof}[Proof of \prpref{triangle}]
Let $T$ be the number of triangles in $\cG$.  For $S \subset \cV$, let $T_S$ denote the number of triangles in $\cG_S$.  We have $T \ge T_{S^\comp} + T_S$.

The following result is based on \citep[Th.~4.1, 4.10]{MR1864966}.  We use it multiple times below without explicitly saying so.

\begin{lem} \label{lem:triangle}
Let $T_m$ be the number of triangles in $\bbG(m, \lambda/m)$.  Fixing $ \lambda > 0$ while $m \to \infty$, we have $T_m \Rightarrow {\rm Poisson}(\lambda^3/6)$.  If instead $\lambda = \lambda_m \to \infty$ with $\log \lambda = o(\log m)$, then $\frac{T_m - \mu}{\sqrt{\mu}} \Rightarrow \cN(0,1)$ where $\mu := \E T = \binom{m}3 (\lambda/m)^3 \sim \lambda^3/6$.
\end{lem}

Assume that \eqref{triangle1} holds.  Applying \lemref{triangle}, $T = O_P(1)$ under $\P_0$, while $T \ge T_S \to \infty$ under $\P_S$.  (For the latter, we use the fact that $T$ is stochastically increasing in $\lambda_1$.) 

\medskip 
Assume that $\lambda_0$ and $\lambda_1$ are fixed. Applying \lemref{triangle}, $T \Rightarrow {\rm Poisson}(\lambda_0^3/6)$  under $\P_0$, while under $\P_S$, $T_{S^\comp}+T_S \Rightarrow {\rm Poisson}((\lambda_0^3+\lambda_1)^3/6)$ since $T_{S^\comp} \sim \bbG(N-n, p_0)$ and $T_S \sim \bbG(n, p_1)$ are independent, and $n=o(N)$. Define $T_{S,S^\comp}:= T - T_{S}-T_{S^\comp}$ as the number of triangles in $\cG$ with nodes both in $S$ and $S^\comp$.  We have
\[E_{S}\left[T_{S,S^{\comp}}\right]\leq N^2np_0^3+ n^2 N p_1p^2_0\leq \frac{n}{N}\lambda_0^3+ \frac{n}{N}\lambda_1\lambda_0^2=o(1)\ ,\]
so that $T_{S,S^{\comp}}=o_{\P_S}(1)$, and by Slutsky's theorem, $T \Rightarrow {\rm Poisson}((\lambda_0^3+\lambda_1)^3/6)$ under $\mathbb{P}_S$.

\medskip
Assume that \eqref{triangle3} holds.  By considering a subsequence if needed, we may assume that $\lambda_0 < \infty$ is fixed.  And since $T$ is stochastically increasing in $\lambda_1$ under the alternative, we may assume that $\lambda_1 > 0$ is fixed.  We have proved above that $T \Rightarrow {\rm Poisson}(\lambda_0^3/6)$ under $\P_0$, $T  \Rightarrow {\rm Poisson}(\lambda_0^3/6 + \lambda_1^3/6)$ the alternative; hence the test $\{T \ge 1\}$ has risk
\[\P_0(T \ge 1) + \P_S(T_S = 0) \to 1 - e^{-\lambda_0^3/6} + e^{-\lambda_0^3/6-\lambda_1^3/6} < 1 \ .\]

\medskip
Finally, assume that \eqref{triangle2} holds.  Using Chebyshev's inequality, to prove that the test based on $T$ is powerful it suffices to show that
\beq \label{cheb-power}
\frac{\E_S T - \E_0 T}{\sqrt{\Var_S(T) \vee \Var_0(T)}} \to \infty \ .
\eeq
Straightforward calculations show that $\E_0 T = \binom{N}3 p_0^3$, and
\beqn
\Var_0(T) 
&=& \big[3 (N-3) (1-p_0) p_0^2 + (1-p_0^3) \big] {N \choose 3} p_0^3 \\
&\asymp& N^4 p_0^5 + N^3 p_0^3 \ .
\eeqn
And carefully counting the number of triplets with 2 or 3 vertices in $S$ gives
\[
\E_S T = {N-n \choose 3}p_0^3 + {n \choose 1}{N-n \choose 2} p_0^3  + {n \choose 2}{N-n \choose 1} p_0^2 p_1 + {n \choose 3} p_1^3 \ ,
\]
while counting pairs of triplets with a certain number of vertices in $S$, shared or not, we arrive at the rough estimate
\[
\Var_S(T) \asymp N^4 p_0^5 + n^2 N^2 p_0^4 p_1 + n^3 N p_0^2 p_1^3 + n^4 p_1^5 + \E_S T \ .
\]
Note that 
\begin{eqnarray}
\E_S T - \E_0 T 
&=& {n \choose 2}{N-n \choose 1} p_0^2 (p_1 - p_0) + {n \choose 3} (p_1^3 - p_0^3) \nonumber\\
&\asymp& n^2 (p_1 - p_0) \big[N p_0^2 + n p_1^2\big] \nonumber\\ 
&\succ& Nn^2p^2_0p_1 + n^3p_1^3 \ ,\label{eq:lower_expected_T}
\end{eqnarray}
since by condition \eqref{triangle2}, $np_0\leq 1$ and $np_1=\lambda_1\gg 1$. 
and
\[
\Var_0(T) \prec \Var_S(T) \asymp N^4 p_0^5 + n^2 N^2 p_0^4 p_1 + n^3 N p_0^2 p_1^3 + n^4 p_1^5 + N^3 p_0^3 + n^3 p_1^3 \ . 
\]
We only need to prove that the square root of this last expression is much smaller than \eqref{eq:lower_expected_T}. 
Since $(np_1)^2\gg Np_0$ and $np_1\to \infty$, we first derive that 
\[n^3 N p_0^2 p_1^3 + n^4 p_1^5 + N^3 p_0^3 + n^3 p_1^3 =o\left[ (np_1)^6\right]\ .\]
Similarly, we get  $n^2 N^2 p_0^4 p_1=o\left[ n^4p_1^2N^2p_0^4\right]$. Finally, \eqref{triangle2} entails that $\lambda_1^2\gg \lambda_0(\lambda_0/\sqrt{N})^{2/3}$ which is equivalent to  $N^4 p_0^5=o\left[ (np_1)^6\right]$. 
\end{proof}

\section{Information theoretic lower bounds} \label{sec:lower}

In this section we state and prove lower bounds on the risk of {\rm any} test whatsoever.  In most cases, we find sufficient conditions under which the null and alternative hypotheses merge asymptotically, meaning that all tests are asymptotically powerless.  In other cases, we find sufficient conditions under which no test is asymptotically powerful.

To derive lower bounds, it is standard to reduce a composite hypothesis to a simple hypothesis.  This is done by putting a prior on the set of distributions that define the hypothesis.  In our setting, we assume that $p_0$ is known so that the null hypothesis is simple, corresponding to the Erd\"os-R\'enyi model $\bbG(N,p_0)$.  The alternative $H_1 := \bigcup_{|S|=n} H_S$ is composite and `parametrized' by subsets of nodes of size $n$.  We choose as prior the uniform distribution over these subsets, leading to the simple hypothesis $\bar{H}_1$ comprising of $\bbG(N,p_0;n,p_1)$ defined earlier.  
The corresponding risk for $H_0$ versus $\bar H_1$ is  
\[
\bar{\gamma}_N(\phi) =  \P_0(\phi = 1) + \frac1{\binom{N}n} \sum_{|S| = n} \P_S(\phi = 0) \ .
\]
Note that $\gamma_N(\phi) \ge \bar \gamma_N(\phi)$ for any test $\phi$.  
Our choice of prior was guided by invariance considerations: the problem is invariant with respect to a relabeling of the nodes.  In our setting, this implies that $\gamma_N^* = \bar \gamma_N^*$, or equivalently, that there exists a test invariant with respect to permutation of the nodes that minimizes the worst-case risk \citep[Lem.~8.4.1]{TSH}.
Once we have a simple versus simple hypothesis testing problem, we can express the risk in closed form using the corresponding likelihood ratio.  Let $\bar{\P}_1$ denote the distribution of $\bW$ under $\bar H_1$, meaning $\bbG(N,p_0;n,p_1)$.  
The likelihood ratio for testing $\P_0$ versus $\bar{\P}_1$ is
\beq \label{L}
L = \frac1{{N \choose n}} \sum_{|S| = n} L_S \ , 
\eeq
where $L_S$ is the likelihood for testing $\P_0$ versus $\P_S$.
Then the test $\phi^* = \{L > 1\}$ is the unique test that minimizes $\bar \gamma_N$, and
\[
\bar \gamma_N(\phi^*) = \bar \gamma_N^* = 1 - \frac12 \E_0 |L - 1| \ .
\]
For each subset $S \subset \cV$ of size $n$, let $\Gamma_S$ be a decreasing event, i.e., a decreasing subset of adjacency matrices, and define the truncated likelihood as
\beq \label{Lt}
\tilde L = \frac1{\binom{N}{n}} \sum_{|S| = n} L_S \, \1_{\Gamma_S} \ .
\eeq
We have
\beqn
\E_0|L-1| 
&\le& \E_0|\Lt-1|+\E_0 (L-\Lt) \\
&\le& \sqrt{\E_0[\Lt^2] - 1 + 2 (1 - \E_0[\Lt])} + (1-\E_0[\Lt]) \ ,
\eeqn
using the Cauchy-Schwarz inequality and the fact that $\E_0 L = 1$ since it is a likelihood.  Hence, for all tests to be asymptotically powerless, it suffices that $\lim\sup \E_0[\Lt^2] \le 1 $ and $\lim\inf\E_0[\Lt] \ge 1$.  Note that
\[
\E_0[\Lt] = \frac1{\binom{N}n} \sum_{|S|=n} \P_S(\Gamma_S) \ .
\]
In all our examples, $\P_S(\Gamma_S)$ is only a function of $|S|$, and since all the sets we consider have same size $n$, $\E_0[\Lt] \to  1 $ is equivalent to $\P_S(\Gamma_S) \to 1$.

\subsection{All tests are asymptotically powerless} \label{sec:powerless}

We start with some sufficient conditions under which all tests are asymptotically powerless.  Recall $\alpha$ in \eqref{a} and $\zeta$ in \eqref{zeta}.  We require that $\zeta \to 0$ below to prevent the total degree test from having any power (see \prpref{total}).

\begin{thm} \label{thm:lower}
Assume that $\zeta \to 0$.  Then all tests are asymptotically powerless in any of the following situations:
\beq \label{lower1}
\lambda_0 \to 0, \quad \lambda_1 \to 0, \quad \limsup \frac{I_{\lambda_0}}{I_{\lambda_1}} \frac{\log n}{\log N}<  1 \ ;
\eeq
\beq \label{lower2}
0 < \liminf \lambda_0 \le \limsup \lambda_0 < \infty, \quad \lambda_1 \to 0 \ ;
\eeq
\beq \label{lower3}
\lambda_0 \to \infty \text{ with } \alpha \to 0, \quad \limsup \lambda_1 < 1 \ ;
\eeq
\beq \label{lower4}
0 < \liminf \alpha \le \limsup \alpha < 1, \quad \limsup \ (1-\alpha)\sup_{k=n/u_N}^n \frac{\E_S[W_{k,S}^*]}{k} < 1 \ .
\eeq
\end{thm}

We recall here the first few steps that we took in \citep{subgragh_detection} to derive analogous lower bounds in the denser regime where $\liminf \alpha \ge 1$. 
We start with some general identities.  We have
\beq \label{Ldef}
L_S := \exp(\theta W_S - \Lambda(\theta) \nn)\ ,
\eeq
with
\beq \label{theta-def}
\theta := \theta_{p_1}, \quad \theta_q := \log\left(\frac{q (1-p_0)}{p_0 (1-q)}\right) \ ,
\eeq
and
\[
\Lambda(\theta) := \log(1 -p_0 +p_0 e^\theta)\ ,
\]
which is the cumulant generating function of ${\rm Bern}(p_0)$. 

In all cases, the events $\Gamma_S$ satisfy 
\beq \label{Gamma}
\Gamma_S \subset \bigcap_{k > \kmin} \{W_T \leq w_k,\ \forall T\subset S\text{ such that }|T|=k \} \ ,
\eeq
where $\kmin$ and $w_k$ vary according to the specific setting. 

To prove that $\E_0 \Lt^2 \le 1 + o(1)$, we proceed as follows.  
We have
\beqn
\E_0 \big[\Lt^2 \big] 
&=& \frac1{\binom{N}{n}^2} \sum_{|S_1|=n} \sum_{|S_2|=n} \E_0 \left(L_{S_1} L_{S_2} \1_{\Gamma_{S_1}} \1_{\Gamma_{S_2}} \right) \\
&=& \frac1{\binom{N}{n}^2} \sum_{|S_1|=n} \sum_{|S_2|=n}  \E_0 \left[\exp\left(\theta (W_{S_1} + W_{S_2}) - 2 \Lambda(\theta) \nn\right) \1_{\Gamma_{S_1} \cap \Gamma_{S_2}} \right] \ .
\eeqn
Define 
\[
W_{S \times T} = \sum_{i \in S, j \in T} W_{i,j}\ ,
\]
and note that $W_S = \frac12 W_{S \times S}$.
We use the decomposition
\beq \label{decomp}
W_{S_1} + W_{S_2} = W_{S_1 \times (S_1 \setminus S_2)} + W_{S_2 \times (S_2 \setminus S_1)} + 2 W_{S_1 \cap S_2}\ ,
\eeq
the independence of the random variables on the RHS of \eqref{decomp} and FKG inequality to get
\beqn \label{Lt-bound}
\lefteqn{\E_0 \left(e^{\theta (W_{S_1} + W_{S_2}) - 2 \Lambda(\theta) \nn } \1_{\Gamma_{S_1} \cap \Gamma_{S_2}} \right) } & &\\
&= & \E_0\left(e^{2 \theta W_{S_1 \cap S_2} - 2 \Lambda(\theta) K^{(2)}} \E_0\left[e^{\theta (W_{S_1 \times (S_1 \setminus S_2)}+W_{S_2 \times (S_2 \setminus S_1)})  - 2\frac{\Lambda(\theta)}2 (n-K)(n+K-1)}\1_{\Gamma_{S_1} \cap \Gamma_{S_2}}\big|\cG_{S_1\cap S_2}\right]\right)\\
 &\leq & {\rm I} \cdot {\rm II} \cdot {\rm III}\ ,
\eeqn
where $K=|S_1\cap S_2|$, 
\[
{\rm I} := \E_0 \left[ \exp\left(\theta W_{S_1 \times (S_1 \setminus S_2)} - \frac{\Lambda(\theta)}2 (n-K)(n+K-1)\right) \right] = 1\ , 
\]
\[
{\rm II} := \E_0 \left[ \exp\left(\theta W_{S_2 \times (S_2 \setminus S_1)} - \frac{\Lambda(\theta)}2 (n-K)(n+K-1)\right) \right] = 1\ , 
\]
\[
{\rm III} := \E_0 \left[ \exp\left(2 \theta W_{S_1 \cap S_2} - 2 \Lambda(\theta) K^{(2)} \right) \1_{\Gamma_{S_1} \cap \Gamma_{S_2}} \right] \ . 
\]
In
The first two equalities are due to the fact that the likelihood integrates to one. 

Assuming that $\zeta \to 0$, we prove that all tests are asymptotically powerless in the following settings:
\beq \label{lower-a}
\limsup \lambda_0 < \infty, \quad \lambda_1^2 = o(\lambda_0) \ ;
\eeq
\beq \label{lower-b}
\lambda_0 \to 0, \quad \lambda_1 \to 0, \quad \lim\sup \frac{I_{\lambda_0}\log(n)}{I_{\lambda_1}\log(N)}<1, \quad n^2=o(N) \ ;
\eeq
\beq \label{lower-c}
\limsup \lambda_1 < 1,\quad \lambda_0\to\infty\ ,  \quad \limsup \alpha <1 \ ;
\eeq
\begin{multline}
\label{lower-d}
\lim\inf \lambda_1\geq 1, \quad 0 < \liminf \alpha \le \limsup \alpha < 1, \\ 
\limsup \ (1-\alpha)\sup_{k=n/u_N}^n \frac{\E_S[W_{k,S}^*]}{k} < 1 \ .
\end{multline}
This implies \thmref{lower}.  Indeed, \eqref{lower-a} includes \eqref{lower2}.  Assume that \eqref{lower1} holds.  Consider any subsequence $n^2/N$ converging to $x \in \mathbb{R}^+\cup\{\infty\}$.  If $x=0$, then \eqref{lower-b} holds.  If $x \ne 0$, then $\zeta\to 0$ implies that $(\lambda_1-\lambda_0 n/N)^2/\lambda_0=o(1)$. If, in addition, $\lambda_1\geq 2\lambda_0 n/N$, this implies  that $\lambda_1^2/\lambda_0=o(1)$. If, otherwise, $\lambda_1\leq 2\lambda_0n/N$, then $\lambda_1^2/\lambda_0\leq 4 \lambda_0(n/N)^2=o(1)$ since $\lambda_0=o(1)$. Thus, in both cases, \eqref{lower-a} holds.  Finally, \eqref{lower-c} includes \eqref{lower3} and also \eqref{lower4} when $\limsup \lambda_1 < 1$, while \eqref{lower-d} includes \eqref{lower4} when $\liminf \lambda_1 \ge 1$.  We note that \eqref{lower-d} implies that $\limsup \lambda_1 < \infty$ because of \eqref{broad-lb}.

\subsubsection{Proof of \thmref{lower} under (\ref{lower-a})}

The arguments here are very similar to those used in \citep{subgragh_detection}, except for the choice of events $\Gamma_S$.
Define
\[\Gamma_S:= \{\text{$\cG_S$ is a forest}\}\ .\]
When $\Gamma_S$ holds, for any $T \subset S$, $\cG_T$ is also a forest, and since any forest $\cF$ with $k$ nodes and $t$ connected components (therefore all trees) has exactly $k - t \le k$ edges, we have $W_T \le |T|$.  Hence, \eqref{Gamma} holds with $w_k := k$.  

\begin{lem}\label{lem:gamma_poisson}
$\P_S(\Gamma_S)$ is independent of $S$ of size $n$, and $\P_S(\Gamma_S)\to 1$.  
\end{lem}

\begin{proof}
The expected number of cycles of size $k$ in $\cG_S$ under $\P_S$ is equal to
\begin{equation}
\frac{n!}{(n-k)!2k} \, p_1^k \leq \frac{\lambda_1^k}{2k} \ .\label{eq:numberloops}
\end{equation}
Summing \eqref{eq:numberloops} over $k$, we see that the expected number of cycles in $\mathcal{G}_S$ under $\P_S$ is smaller than $\lambda_1^3/(1-\lambda_1) = o(1)$. Hence, with probability going to one under $\P_S$, $\cG_S$ has no cycles and is therefore a forest. 
\end{proof}

In order to conclude, we only need to prove that  $\lim\sup\mathbb{E}_0[\tilde{L}^2]\leq 1$.  We start from \eqref{Lt-bound} and we recall that $K=|S_1\cap S_2|$.
We take $\kmin$ as the largest integer $k$ satisfying. 
\[\frac{2}{k-3}\geq \frac{p_1^2(1-p_0)}{p_0(1-p_1)^2}\ , \]
with the convention $2/0=\infty$, so that $\kmin \ge 3$. 
Let $q_k=2/(k-1)$. 
Recall that $\rho = n/(N-n)$ and define $k_0 = \lceil b n \rho \rceil$, where $b \rightarrow \infty$ satisfies $b^2 \zeta \to 0$. 

\bitem
\item When $K \leq  \kmin$, we will use the obvious bound:
\beqn
{\rm III} 
& \le& \E_0 \exp\left(2 \theta W_{S_1 \cap S_2} - 2 \Lambda(\theta) K^{(2)} \right) = \exp\left(\Delta K^{(2)}\right), 
\eeqn
where
\beq \label{Delta}
\Delta := \Lambda(2 \theta) - 2 \Lambda(\theta) = \log\left(1 + \frac{(p_1 -p_0)^2}{p_0 (1 -p_0)}\right).
\eeq

\item When $K > \kmin$, we use a different bound.  Noting that $\Gamma_{S_1} \cap \Gamma_{S_2} \subset \{W_{S_1 \cap S_2} \le w_K\}$, for any $\xi \in (0, 2 \theta)$, we have
\beqn
{\rm III} 
&\le& \E_0 \left[\exp\left(\xi W_{S_1 \cap S_2} + (2\theta -\xi) w_K - 2 \Lambda(\theta) K^{(2)} \right) \IND{W_{S_1 \cap S_2} \le w_K}\right] \\
&\le& \E_0 \left[\exp\left(\xi W_{S_1 \cap S_2} + (2\theta -\xi) w_K - 2 \Lambda(\theta) K^{(2)} \right)\right]\ ,
\eeqn
 so that
\[ 
{\rm III} \le \exp\left(\Delta_K K^{(2)}\right),
\] 
where
\beq \label{Deltak}
\Delta_k := \min_{\xi \in [0, 2 \theta]} \Lambda(\xi) + (2\theta -\xi) q_k - 2\Lambda(\theta)\ .
\eeq
\eitem

Using the fact that $\mathbb{E}_0[\tilde{L}^2] \leq \E[{\rm III}]$ where the expectation is taken with respect to $K$, we have
\beqn
\mathbb{E}_0[\tilde{L}^2]
&\leq& \E\left[\IND{K \le k_0} \exp\left(\Delta K^{(2)}\right)\right] \\
&& \quad + \E\left[\IND{k_0+1 \le K \le \kmin} \exp\left(\Delta K^{(2)}\right)\right] \\ 
&& \quad \quad +  \E\left[\IND{\kmin+1 \le K \le n} \exp\left(\Delta_K K^{(2)}\right)\right] \\
&=& A_1 + A_2 + A_3 \ ,
\eeqn
where the expectation is with respect to $K \sim {\rm Hyp}(N, n, n)$.
By \lemref{hyper}, $K$ is stochastically bounded by $\Bin(n, \rho)$.  Hence, using Chernoff's bound (see \lemref{chernoff}), we have
\beq \label{Kbound}
\P(K \ge k) 
\le \P({\rm Hyp}(N, n, n) \ge k) 
\le \P(\Bin(n, \rho) \ge k) 
\le \exp\left(- n H_{\rho}(k/n) \right) .
\eeq

\bitem
\item When $K \le k_0$, we proceed as follows.  If $k_0 = 1$, we simply have
\[
A_1 = \P(K \le 1) \le 1\ .
\]
If $k_0 \ge 2$, we use the expression \eqref{Delta} of $\Delta$ to derive
\[
A_1 \leq \exp\left[\Delta k_0^2\right] \leq \exp\left[O(1)\frac{(p_1-p_0)^2}{p_0(1-p_0)}\frac{b^2 n^4}{N^2}\right] = \exp\left[O(b^2 \zeta) \right]  = 1 + o(1) \ .
\]

\item When $k_0 + 1 \le K \leq \kmin$, we use \eqref{Kbound} and \lemref{H}, to get
\begin{eqnarray*}
A_2 
\notag &\le &\sum_{k = k_0 + 1}^{\kmin} \exp\left[\Delta \frac{k(k-1)}2 - nH_{\rho}\left(\frac{k}{n}\right)\right]\\
\label{A2} &\leq & \sum_{k = k_0 + 1}^{\kmin} \exp\left[k\left(\Delta\frac{k-1}{2}-\log\left(\frac{k}{n\rho}\right)+1\right)\right] \ .
\end{eqnarray*}
The last sum is equal to zero if $\kmin\le k_0$; therefore, assume that $\kmin>k_0$.  
For $a > 0$ fixed, the function $f(x) = a x - \log x$ is decreasing on $(0, 1/a)$ and increasing on $(1/a, \infty)$.  Therefore, for $k_0 + 1 \le k \le \kmin$,
\[
\Delta \frac{k-1}2 - \log\left(\frac{k}{n\rho}\right) \le \max_{\ell\in\{k_0,k_{\min}\}} \left\{\Delta\frac{\ell-1}{2}-\log\left(\frac{\ell N}{n^2}\right)\right\}\ .
\]
We know that $\Delta (k_0 -1) = o(1)$, so that
\[
\Delta \frac{k_0-1}2 - \log\left(\frac{k_0}{n\rho}\right) \le o(1) - \log b \to -\infty \ .
\]
Therefore, it suffices to show that 
\[\frac{\kmin-1}{2}\Delta-\log\left(\frac{\kmin}{n\rho}\right)\rightarrow -\infty \ .\]
 If $\kmin>3$, observe that 
\[\frac{\kmin-1}{2}\Delta\leq \left(1+\frac{\kmin-3}{2}\right)\log\left(1+\frac{2}{\kmin-3}(1+o(1))\right) \leq \frac32 \log 3 +o(1)\ ,\]
while $\log(\kmin/(n\rho))\geq \log(k_{0}/(n\rho))\rightarrow \infty$.
If we have $\kmin=3$, then we have
\beqn
\Delta-\log\left(\frac{3}{n\rho}\right)\leq \log(p_1^2/p_0)-\log(N/n^2)+O(1)\leq \log\left(\frac{\lambda_1^2}{\lambda_0}\right)+O(1)\rightarrow -\infty \ ,
\eeqn
because of \eqref{lower-a}.

\item When $\kmin < K \le n$, we need to bound $\Delta_K$.  
Remember the definition of the entropy function $H_q$ in \eqref{H}, and that $H(q)$ is short for $H_{p_0}(q)$.  It is well-known that $H$ is the Fenchel-Legendre transform of $\Lambda$; more specifically, for $q \in (p_0,1)$, 
\beq \label{HL}
H(q) = \sup_{\theta \ge 0} [q \theta - \Lambda(\theta)] = q \theta_q - \Lambda(\theta_q) \ .
\eeq 
Hence, the minimum of $\Lambda(\xi) + (2\theta -\xi) q_k - 2\Lambda(\theta)$ over $\xi> 0$ is achieved at $\xi=\theta_{q_k}$ as soon as $2\theta\geq \theta_{q_k}$.  Moreover, by definition of $\theta$ in \eqref{theta-def}, our choice of $q_k$, and the fact that $k \ge \kmin$, we have
\[
2 \theta - \theta_{q_k} = \log \left( \frac{p_1^2 (1-p_0)}{p_0 (1-p_1)^2} \frac2{k-3} \right) \ge 0 \ .
\]
Hence, we have
\begin{eqnarray}
\Delta_k&=& - H(q_k)+2\theta q_k-2\Lambda(\theta)\nonumber \\
& = & -2H_{p_1}(q_k)+ H(q_k)\ . \label{eq:delta_k}
\end{eqnarray}
Using the definition of the entropy and the fact that $p_0 = o(1)$, we therefore have 
\beqn 
\Delta_k &=& q_k\log\left(\frac{p_1^2}{q_k p_0}\right)+ (1-q_k)\log\left(\frac{(1-p_1)^2}{(1-q_k)(1-p_0)}\right) \\
&\leq & \frac{2}{k-1} \Bigg(\log\left(\frac{\lambda_1^2 N (k-1)}{2 \lambda_0 n^2}\right)+ O(1)\Bigg) \ ,
\eeqn
where the $O(1)$ is uniform in $k$.
Hence, starting from the bound we got when bounding $A_2$, we have
\beqn
A_3 
&\leq & \sum_{k=\kmin+1}^{n} \exp\left[k\left(\Delta_k \frac{k-1}{2}-\log\left(\frac{k}{n\rho}\right)+1\right)\right] \\
&\leq & \sum_{k=\kmin+1}^{n} \exp\left[k\left\{\log\left(\frac{\lambda_1^2}{\lambda_0}\right)+\log\left(\frac{N(k-1)}{2 n^2}\right)-\log\left(\frac{Nk}{n^2}\right) + O(1)\right\}\right]\\ &\leq &
\sum_{k=\kmin+1}^{n} \exp\left[k\left\{\log\left(\frac{\lambda_1^2}{\lambda_0}\right)+O(1)\right\}\right]=o(1)\ , 
\eeqn
since $\lambda_1^2/\lambda_0 = o(1)$.
\eitem

This concludes the proof of \thmref{lower} under \eqref{lower-a}.

\subsubsection{Proof of \thmref{lower} under (\ref{lower-b})} \label{sec:lower-b}

Let $c$ be a positive constant that will be chosen small later on.
Define 
\[f_n:= \left(1+c\right)I_{\lambda_1}^{-1}\log(n)\ .\]
We consider the event 
\[
\Gamma_S = \{\text{$\mathcal{G}_S$ is a forest} \}\cap \{|\cC_{\max, S}| \le f_n\} \ .
\]
When $\Gamma_S$ holds, for any $T \subset S$, $\cG_T$ is also a forest, with $|T|-W_{T}$ connected components.  Since the size of each connected component is at most $f_n$, there are at least $\lceil |T|/f_n \rceil$ connected components.  Hence, \eqref{Gamma} holds with $w_k =k-\lceil \frac{k}{f_n}\rceil$.

\begin{lem}\label{lem:gamma_poisson_subcritic}
$\P_S(\Gamma_S)$ is independent of $S$ of size $n$, and $\P_S(\Gamma_S)\to 1$.  
\end{lem}

\begin{proof}
This is a straightforward consequence of Lemmas~\ref{lem:subcritic_cluster_upper} and~\ref{lem:gamma_poisson}.
\end{proof}

To conclude, it suffices to show that $\mathbb{E}_0[\tilde{L}^2] \le 1 + o(1)$.  
For this, we will need the following.

\begin{lem}\label{lem:combinatorics_forest}
Let $F_{k,j}$ stand for the number of forests with $j$ trees on $k$ labelled vertices.
For any $k\geq 2$ and any $j\leq k$,  $F_{k,j}\leq k^{k-2}$.
\end{lem}
\begin{proof}
Fix  $k\geq 2$.  By Cayley's formula, we have $F_{k,1}=k^{k-2}$.  Therefore, it suffices to prove that $F_{k,j} \ge F_{k,j+1}$ for all $j \ge 1$. 
If we take a forest with $j$ trees and erase any of its $k-j$ edges, we obtain a forest with $j+1$ trees.  And there are exactly $\sum_{s \ne t} k_s k_t$ such ways of obtaining a given forest with $j+1$ trees of sizes $k_1 \le \cdots \le k_{j+1}$.  Since
\[
\sum_{s \ne t} k_s k_t \geq k_1 (k - k_1) \geq k-1\ ,
\]
it follows that $F_{k,j}(k-j)\geq F_{k,j+1}(k-1)$. Thus, $F_{k,j} \ge F_{k,j+1}$.
\end{proof}

Starting from \eqref{Lt-bound}, and using the fact that, under $\Gamma_{S_1} \cap \Gamma_{S_2}$, $\cG_{S_1 \cap S_2}$ is a forest with $W_{S_1 \cap S_2} \le w_K$ edges, we have
\beqn 
\mathbb{E}_0\big[\tilde{L}^2\big]
&\leq& \E_0 \left(\exp\left(2 \theta W_{S_1 \cap S_2} - 2 \Lambda(\theta) K^{(2)} \right) \1_{\{\text{$\cG_{S_1 \cap S_2}$ is a forest}, \, W_{S_1 \cap S_2} \le w_K\}}\right)\ .
\eeqn 
Note that the exponential term is smaller than $1$ when $|S_1\cap S_2|\leq 1$.
Recall that $\rho = \frac{m}{N-m}$ and that $\Lambda(\theta)= \log\big[(1-p_0)/(1-p_1)\big]$.  
We derive
\begin{eqnarray}
 \mathbb{E}_0\big[\tilde{L}^2\big]-1 &\leq& \sum_{k=2}^n\sum_{i=1}^{w_k} \mathbb{P}\big[K=k,W_{S_1\cap S_2}=i , \text{$\cG_{S_1 \cap S_2}$ is a forest} \big]\exp\big[2i\theta-2\Lambda(\theta) k^{(2)} \big] \notag \\
&\leq &  \sum_{k=2}^n\sum_{i=1}^{w_k}\binom{n}{k}\rho^k F_{k,k-i}\, \frac{p_1^{2i}}{p_0^{i}}\left(\frac{1-p_0}{1-p_1}\right)^{2(i-k^{(2)})} \notag \\
&\prec&  \sum_{k=2}^n\sum_{i=1}^{w_k} \left(\frac{n^2}{N}\right)^{k-i}\left(\frac{\lambda_1^2}{\lambda_0}\right)^{i}\frac{F_{k,k-i} \,\binom{n}{k}}{n^k}\notag \\
&\prec&  \sum_{k=2}^n\sum_{i=1}^{w_k} \left(\frac{n^2e}{N}\right)^{k-i}\left(\frac{\lambda_1^2e}{\lambda_0}\right)^{i}\frac{1}{k^2} \notag \\
&\prec&  \sum_{j=1}^\infty  \left(\frac{n^2e}{N}\right)^{j} \ \sum_{i=1}^{j \lfloor  f_n\rfloor }\left(\frac{\lambda_1^2e}{\lambda_0}\right)^{i}\frac{1}{(i+j)^2} \notag \\
&\prec&  \sum_{j=1}^\infty  \left(\frac{n^2e}{N} \big[1 \vee \frac{\lambda_1^2e}{\lambda_0}\big]^{f_n} \right)^{j} \ . \label{2nd-moment-lower-b}
\end{eqnarray}
In the second inequality, we used the fact that $K$ is stochastically bounded by $\Bin(n, \rho)$ (see \lemref{hyper}).  
In the third inequality, we used the fact that $p_0 < p_1$ and $i \le w_k < k$, 
as well as the fact that $n^2 = o(N)$, which implies that $\rho^k \sim (n/N)^k$.
In the fourth inequality, we used \lemref{combinatorics_forest} and the lower bound $k! \ge (k/e)^k$.
The fifth inequality comes from a change of variables and uses the definition of $w_k$. When $\lambda_1^2 e \le \lambda_0$, since $n^2 = o(N)$, this sum is $O(n^2/N)$.
When $\lambda_1^2 e > \lambda_0$, this sum is equal to
\beq \label{2nd-moment-bound}
\frac1{e^{A_n -1} -1}, \qquad A_n := \log\left(\frac{N}{n^2}\right)- f_n\log\left(\frac{\lambda_1^2e}{\lambda_0}\right)\ .
\eeq
So it suffices to show that $A_n \to \infty$.
Since we are working under \eqref{lower-b}, there is $c > 0$ such that, eventually, 
\[
\frac{I_{\lambda_0} \log n}{I_{\lambda_1} \log N} \le \frac{1-c}{1+c} \ . 
\]
Then, using the fact that $\lambda_0\vee\lambda_1=o(1)$, we have
\beqn
f_n \log\left(\frac{\lambda_1^2e}{\lambda_0}\right) 
&=& (1+c) \frac{\log n}{I_{\lambda_1}} \big(2 \lambda_1 - 2 I_{\lambda_1} + I_{\lambda_0} - \lambda_0\big) \\
&\le& - (1+c+o(1)) \log (n^2) + (1-c) \log N \\
&\le& \log(N/n^2) - c \log(N) \ ,
\eeqn
eventually.  This implies that $A_n \ge -1 + c \log N \to \infty$.

This concludes the proof of \thmref{lower} under \eqref{lower-b}.

\subsubsection{Proof of \thmref{lower} under (\ref{lower-c})}
Recall that $\rho = n/(N-n)$ and define $k_0 = \lceil b n \rho \rceil$, where $b \rightarrow \infty$ satisfies $b^2 \zeta \to 0$.  
Let $k_{\min}$ be the integer part of $1+\frac{2}{1-\alpha}\big(1\vee \frac{n^{2-\alpha}}{N^{1-\alpha}}\big)$.
Define 
\[
\Gamma_S = \bigcap_{k=k_{\min}+1}^n\{W_T\leq w_k,\ \forall T\subset S\text{ such that }|T|=k \}\ , 
\]
where $w_k := k$ here.  

\begin{lem}\label{lem:control_gamma}
For any $k> k_{\min}$ and any subset $S$ of size $n$, we have  $\mathbb{P}_S[\Gamma_S]\to 1$.
\end{lem}
This takes care of the first moment.  In order to conclude, it suffices to control the second moment, specifically, to prove that $\overline{\lim }\mathbb{E}\big[\tilde{L}^2\big]\leq 1$.
Arguing as before, we have 
\beqn
\mathbb{E}_0[\tilde{L}^2]
&\leq& \E\left[\IND{K \le k_0} \exp\left(\Delta K^{(2)}\right)\right] \\
&& \quad + \E\left[\IND{k_0+1 \le K \le \kmin} \exp\left(\Delta K^{(2)}\right)\right] \\ 
&& \quad \quad +  \E_0 \left[\IND{k_0+1 \le K \le k_{\min}} \exp\left(2 \theta W_{S_1 \cap S_2} - 2 \Lambda(\theta) K^{(2)} \right) \IND{W_{S_1 \cap S_2} \le w_K}\right] \\
&=& A_1 + A_2 + A_3 \ .
\eeqn

\bitem
\item Arguing exactly as we did before, we have $A_1 = 1+o(1)$.
\item Arguing as before, we also have
\beqn
A_2 &\leq& \sum_{k=k_0+1}^{k_{\min}}\exp\left[k\left(\Delta\frac{k-1}{2}-\log\left(\frac{k}{n\rho}\right)+1\right)\right]\\
&\leq &\sum_{k=k_0+1}^{k_{\min}}\exp\left[k\left( 1+o(1)+\max_{\ell\in\{k_0,k_{\min}\}} \left\{\Delta\frac{\ell-1}{2}-\log\left(\frac{\ell N}{n^2}\right)\right\}\right)\right]\ .
\eeqn
First, we have $\Delta(k_0-1)/2-\log(k_0N/n^2)\rightarrow -\infty$. 
This is true if $k_0 = 1$, and when $k_0 > 1$, we have $N/n^2 \le b$, so that 
\[\frac{(p_1-p_0)^2}{p_0(1-p_0)} \sim \frac{N^2}{n^4} \zeta = \frac{N^2}{n^4 b^2} \, b^2 \zeta \to 0\ ,\]
by definition of $b$, and therefore
\[
\Delta \frac{k_0-1}2 \asymp \frac{N^2}{n^4} \zeta \, \frac{b n^2}N \le b^2 \zeta \to 0 \ .
\]
We also have $\Delta(k_{\min}-1)/2-\log(k_{\min} N/n^2)\rightarrow -\infty$.  To show this, we divide the analysis into two cases.  When $N^{1-\alpha}\leq n^{2-\alpha}$, this results from 
\[\Delta \frac{k_{\min}-1}{2} \le (1+o(1)) \frac{n^{2-\alpha}}{(1-\alpha) N^{1-\alpha}} \frac{p_1^2}{p_0} = (1+o(1)) \frac{\lambda_1^2}{1-\alpha}=O(1) \ ,\]
together with
\beq \label{kmin-infty1}
\log\left(\frac{k_{\min}N}{n^2}\right) \ge \log\left(\frac{2N^\alpha}{(1-\alpha) n^\alpha}\right) \ge \alpha \log(N/n) \to \infty\ ,
\eeq
where we used the definition of $k_{\min}$ and the fact that $\lambda_0 = (N/n)^\alpha$.
When $N^{1-\alpha}\geq n^{2-\alpha}$, this results from
\begin{eqnarray*}
\Delta \frac{k_{\min}-1}{2} 
&\le& \frac{1}{2}\lfloor \frac{2}{1-\alpha}\rfloor  \log\left(1+\frac{p_1^2}{p_0}\right) + o(1) \\
&\leq& \frac{1}{2}\lfloor \frac{2}{1-\alpha}\rfloor \log\left[1+\lambda_1^2\frac{N^{1-\alpha}}{n^{2-\alpha}}\right] +o(1)\\
&\leq& \frac{1}{2}\lfloor \frac{2}{1-\alpha}\rfloor \log\left[(1+\lambda_1^2)\frac{N^{1-\alpha}}{n^{2-\alpha}}\right] +o(1)\\
&\leq & \frac{1}{1-\alpha}\log(1+\lambda^2_1) +o(1) + \log(N/n^2) \left\{ \begin{array}{ccc}  -\frac{\alpha}{1-\alpha}\log\left(n\right) & \text{ if } & \alpha\geq 1/3\\
                                                       - \alpha \log(N/n) & \text{ if } & \alpha<1/3
                                                      \end{array}\right.
 \ ,
\end{eqnarray*}
where in the last line, we have used the identity $\lfloor 2/(1-\alpha)\rfloor =1$ for $\alpha<1/3$.

And we also have
\beq \label{kmin-infty2}
\log\left(\frac{k_{\min}N}{n^2}\right) \ge \log\left(N/n^2\right)\ ,
\eeq
so that
\[
\Delta \frac{k_{\min}-1}{2} - \log\left(\frac{k_{\min}N}{n^2}\right) \le \frac{1}{1-\alpha}\log(1+\lambda^2_1) - \left\{ \begin{array}{ccc}  -\frac{\alpha}{1-\alpha}\log\left(n\right) & \text{ if } & \alpha\geq 1/3\\
                                                       - \alpha \log(N/n) & \text{ if } & \alpha<1/3
                                                      \end{array}\right.
\ ,
\]
which goes to $-\infty$ since $\lambda_1 = O(1)$ and $\alpha \log(N/n)=\lambda_0\to \infty$. Hence, we have $A_2=o(1)$. 

\item It remains to prove that $A_3= o(1)$. If we assume that $p_1\leq 2p_0$, then $\Delta_k\leq \Delta \leq p_0(1+o(1))$ and we can prove that $A_3=o(1)$ arguing as for $A_2$ above:
\beqn
A_3&\leq& \sum_{k=k_{\min}+1}^{n}\exp\left[k\left(\Delta\frac{k-1}{2}-\log\left(\frac{k}{n\rho}\right)+1\right)\right]\\
&\leq &\sum_{k=k_{\min}+1}^{n}\exp\left[k\left( 1+o(1)+\max_{\ell\in\{k_{\min}+1,n\}} \left\{\Delta\frac{\ell-1}{2}-\log\left(\frac{\ell N}{n^2}\right)\right\}\right)\right]\\
&\leq & \sum_{k=k_{\min}+1}^{n}\exp\left[k\left( 1+o(1)+ \Delta\frac{n}{2}-\log\left(\frac{k_{\min} N}{n^2}\right)\right)\right]\ .
\eeqn
On one hand, we have $\Delta n\prec np_0=(n/N)^{1-\alpha}=o(1)$. On the other hand, $\log(k_{\min} N/n^2) \to \infty$.  Indeed, when $N^{1-\alpha}\leq n^{2-\alpha}$, we have \eqref{kmin-infty1}; and when $N^{1-\alpha} > n^{2-\alpha}$, then $N/n^2 > n^{\alpha/(1-\alpha)} \to \infty$ and we use \eqref{kmin-infty2}.
We conclude that $A_3=o(1)$ when $p_1 \le 2 p_0$.
In the following, we suppose that $p_1\geq 2p_0$.
Leaving $w_k$ unspecified, so we can use the same arguments later, we have 
\beqn
A_3 &=  &\E_0 \left[\IND{k_0+1 \le K \le k_{\min}} \exp\left(2 \theta W_{S_1 \cap S_2} - 2 \Lambda(\theta) K^{(2)} \right) \IND{W_{S_1 \cap S_2} \le w_K}\right]\\
&=& 
\sum_{k=k_{\min}+1}^n\sum_{i=1}^{w_k } \mathbb{P}_0\left[|S_1\cap S_2|=k,W_{S_1\cap S_2}=i\right]\exp\left[2i\theta -2 k^{(2)} \Lambda(\theta)\right]\\
& \le & \sum_{k=k_{\min}+1}^n\sum_{i=1}^{w_k} \binom{n}{k}\rho^k \binom{k^{(2)}}{i} p_0^i (1-p_0)^{k^{(2)}-i} \exp\left[2i\log\left(\frac{p_1}{p_0}\right)+2(k^{(2)}-i)\log\left(\frac{1-p_1}{1-p_0}\right)\right] \\
&:=&  \sum_{k=k_{\min}+1}^n \sum_{i=1}^{w_k} B_{i,k}\ .
\eeqn
Furthermore, since $0 < 1-p_0 < 1$ and $1-p_1 < 1-p_0$, we have
\beq
B_{i,k} \leq \binom{n}{k}\rho^k \binom{k^{(2)}}{i} p_0^i (p_1/p_0)^{2i} \leq e^{o(k)} \left(\frac{en^2}{kN}\right)^k \left(\frac{e p_1^2 k^{(2)}}{p_0 i}\right)^i \label{Aik-bound}\ ,
\eeq
using the standard bound $\binom{n}{k} \le (e n/k)^k$. 

We now specify the calculations when $w_k = k$.
Considering the sums over $i=1,\ldots k/2$ and over $i=k/2+1,\ldots k$ separately, we get 
\begin{eqnarray*}
\sum_{i=1}^{k} B_{i,k}&\leq& e^{o(k)} \left(\frac{en^2}{kN}\right)^k\left[\sum_{i=1}^{\lfloor k/2\rfloor }\left(\frac{ep_1^2k^{(2)}}{p_0}\right)^{i}+ \sum_{\lfloor k/2\rfloor +1}^k \left(\frac{ep_1^2k^{(2)}}{p_0k/2}\right)^{i}\right]\\
&\le& e^{o(k)}\left(\frac{en^2}{kN}\right)^k \, k \left[1+ \left(\frac{ep_1^2k^{(2)}}{p_0}\right)^{k/2}+ \left(\frac{ep_1^2k^{(2)}}{p_0k/2}\right)^{k}\right]\\
&\prec & e^{o(k)}\left[\left(\frac{en^2}{k N}\right)^k + \left(\frac{e^{3/2}n^2p_1}{N \sqrt{2 p_0}}\right)^k+\left(\frac{e^{2}n^2p^2_1}{N p_0}\right)^k \right]\ .
\end{eqnarray*}
First, $\frac{en^2}{k N} \le \frac{en^2}{k_0 N} = o(1)$ by definition of $k_0$.
Next, $\frac{n^2p_1}{N\sqrt{p_0}} \le \frac{2 (p_1 - p_0)}{\sqrt{p_0}} \frac{n^2}N = 2 \sqrt{\zeta} \to 0$, by the fact that $p_1 \ge 2 p_0$.
Finally, $n^2p_1^2/(Np_0) = \lambda_1^2/\lambda_0 \to 0$ since $\lambda_0\rightarrow \infty$ and $\lambda_1=O(1)$. Hence, we conclude that 
\beq \label{Bsum}
\sum_{k=k_{\min}+1}^{n}\sum_{i=1}^kB_{i,k}=o(1)\ .
\eeq 
This immediately implies that $A_3 = o(1)$.
\eitem 

This concludes the proof of \thmref{lower} under \eqref{lower-c}.

\medskip
\begin{proof}[Proof of Lemma \ref{lem:control_gamma}]
Let us consider the event 
\beqn
\Gamma'_S:= \{\text{no connected component of $\mathcal{G}_S$ has more than one cycle}\} 
\eeqn
Under $\Gamma'_S$, a connected component of $\mathcal{G}_S$ has at most as many edges as vertices. Consequently, $\Gamma_S\subset \Gamma'_S$ and we only need to prove that $\P_S(\Gamma'_S)\to1$. Since $\lim\sup \lambda_1<1$ and $\P_S(\Gamma'_S)$ is nondecreasing in $\lambda_1$, we may assume that $\lambda_1$ is fixed in $(0,1)$. 

As a warmup for what follows, we note that the number $\mathbf{L}_k$ of cycles of size $k$ in $\mathcal{G}_S$ satisfies 
\begin{equation*}
\mathbb{E}_S[\mathbf{L}_k]=p_1^k \frac{n!}{(n-k)!2k}\leq \frac{\lambda_1^k}{2k} \ ,
\end{equation*}
since there are $n!/[(n-k)!2k]$ potential cycles of size $k$.
Now, if a connected component contains (at least) two cycles, there are two possibilities:
\begin{itemize}
 \item The two cycles have at least one edge in common.  In that case, there is a cycle (say of length $k$) with a chord (say of length $s<k$). Let $\mathbf{L}'_{k,s}$ denote the number of such configurations,
There are $n!/[(n-k)!2k]$ potential cycles of size $k$. Given a cycle of size $k$, there are less than $\binom{k}{2}$ starting and ending nodes possible for the chord.  Once these two nodes are set, there remains less than $n!/(n-s+1)!$ possibilities for the other nodes on the chord. Thus, we have
\begin{eqnarray*}
\mathbb{E}_S[\mathbf{L}'_{k,s}]\leq  p_1^{k+s} \frac{n!}{(n-k)!2k}\binom{k}{2}\frac{n!}{(n-s+1)!}\leq \left(\frac{\lambda_1}{n}\right)^{k+s}kn^{k+s-1}
\le \lambda_1^{k+s}\frac{k}{n} \ .
\end{eqnarray*}
Summing this inequality over $s$ and $k$, we control the expected number of cycles with a chord:
\begin{eqnarray*}
 \sum_{k=3}^{\infty}\sum_{s=1}^{k-1}\mathbb{E}[\mathbf{L}'_{k,s}]\leq \frac{1}{n}\sum_{k=3}^{\infty}\frac{k\lambda_1^{k+1}}{1-\lambda_1} \asymp \frac1n  =o(1)\ ,
\end{eqnarray*}
since $\lim\sup \lambda_1<1$. Hence, this event occurs with probability going to 0.
\item The two cycles have no edge in common.  Since there are in the same connected component, there is a path that goes from a vertex in the first cycle to a vertex in the second cycle. Let us note $\mathbf{L}'_{k_1,k_2,s}$ the number of cycles of size $k_1$ and $k_2$ that do not share an edge and are connected by a path of length $s$.  Observe that there are less $\frac{n!}{(n-k_1)!2k_1}$ possible configurations for the first cycle, less than $\frac{n!}{(n-k_2)!2k_2}$ possible configurations for the second cycle, and less than $k_1k_2 \, n!/(n-s+1)!$ possibilities for the chord. Thus, we get
\beqn
\mathbb{E}\left[\mathbf{L}'_{k_1,k_2,s}\right]&\leq& p_1^{k_1+k_2+s} \frac{n!}{(n-k_1)!2k_1}\frac{n!}{(n-k_2)!2k_2}k_1k_2\frac{n!}{(n-s+1)!}\\
&\leq & \left(\frac{\lambda_1}{n}\right)^{k_1+k_2+s} n^{k_1+k_2+s-1}=
 \frac{\lambda_1^{k_1+k_2+s}}n\ ,
\eeqn
so that the expected number of such configurations is bounded as follows
\[\sum_{k_1\ge 3} \sum_{k_2 \ge 3} \sum_{s\ge1} \mathbb{E}\left[\mathbf{L}'_{k_1,k_2,s}\right]\leq \frac{1}n \sum_{k_1\ge 3} \sum_{k_2 \ge 3} \sum_{s\ge1} \lambda_1^{k_1+k_2+s} \asymp \frac1n = o(1)\ .\]
Hence, this second event occurs with probability going two zero
\end{itemize}
All in all, we have proved that $\P_S(\Gamma'_S)\to 1$, implying that $\P_S(\Gamma_S)\to 1$.
\end{proof}

\subsubsection{Proof of \thmref{lower} under (\ref{lower-d})}
We follow the arguments laid out for the case \eqref{lower-c}.  We define $\Gamma_S$ in the same way, except that $w_k := \big[k\frac{(1-c)^{1/2}}{1-\alpha}\big]$, where $c$ is a positive constant (to be chosen small later) such that $c < \alpha$ and, eventually,
\begin{equation}
\sup_{n/u_N <k\leq n}\frac1k \E_S[W_{k,S}^*]\leq \frac{1-2c}{1-\alpha} \ .
\label{eq:condition_sup_W_kS}
\end{equation}

\begin{lem}\label{lem:control_gamma2}
For any $k> k_{\min}$ and any subset $S$ of size $n$, we have  $\mathbb{P}_S[\Gamma_S]\to 1$.
\end{lem}

For the second moment, we proceed exactly as in the case \eqref{lower-c}, and we start from \eqref{Bsum}.  In fact, when $w_k \le k$, the proof is complete.  So we assume that $c$ is small enough that $w_k > k$, and bound the sum over $k+1\leq i \leq w_k$.  For $i > k$, we  use the bound \eqref{Aik-bound}, together with the fact that $\lambda_0 = (N/n)^\alpha$ and $k < i$, to derive
\beqn
B_{i,k} &\le& e^{o(k)} \left(\frac{en^2}{k N}\right)^k \left(\frac{e p_1^2 k^{(2)}}{p_0 i}\right)^i \\
&\le& e^{o(k)} \left(\frac{en^2}{k N}\right)^k \left(\frac{N^{1-\alpha} k}{n^{2-\alpha}} \frac{\lambda_1^2 e}2 \right)^i \\ 
&= & e^{o(k)+k} \left(\frac{n}{N}\right)^{k-i(1-\alpha)} \left( \frac{\lambda_1^2 e}2 \right)^i \left(\frac{n}{k}\right)^{k-i}\\
&\le & e^{o(k)+k} \left(\frac{n}{N}\right)^{k-i(1-\alpha)} \left( \frac{\lambda_1^2 e}2 \right)^i \ .
\eeqn
This allows us to control the sum 
\begin{eqnarray*}
\sum_{i=k+1}^{w_k} B_{i,k}
&\leq & w_k\, e^{o(k)+k} \left(\frac{n}{ N}\right)^{k - (1-\alpha) w_k}\left(\frac{\lambda_1^2 e}{2} \vee 1 \right)^{w_k}\\
&\prec & k e^{o(k)+k} \left(\frac{n}{ N}\right)^{k(1-(1-c)^{1/2})} \left(\frac{\lambda_1^2 e}{2} \vee 1 \right)^{ k\frac{(1-c)^{1/2}}{1-\alpha}}\\
&=& \exp\left[O(k) -k(1-(1-c)^{1/2}) \log(N/n)\right] \ ,
\end{eqnarray*}
where in the second line we used the fact that $w_k = O(k)$ since $\limsup a < 1$, and in the third line we used the fact that $\lambda_1 = O(1)$.  Thus, 
$$\sum_{k=k_{\min}+1}^n \sum_{i=k+1}^{w_k} B_{i,k}=o(1) \ ,$$ 
which together with \eqref{Bsum} allows us to conclude that 
$A_3=o(1)$.

This concludes the proof of \thmref{lower} under \eqref{lower-d}.

\medskip
\begin{proof}[Proof of \lemref{control_gamma2}]
Recall that $u_N = \log\log(N/n)$.
First we consider integers $k$ satisfying $k_{\min}+1 \le k<  n/u_N$. Define $\omega'_k= k(1-c)^{-1/2}\left(\frac{\lambda_1}{2}\vee 1\right)$ and $q'_k=\omega'_k/k^{(2)}$. Applying a union bound and Chernoff's bound for the binomial distribution, we derive that 
\beqn
 \mathbb{P}_S\left[W_{k,S}^*\geq \omega'_k\right]&\leq& \binom{n}{k} \P[\Bin(k^{(2)},p_1)\geq \omega'_k]\\
&\leq& \exp\left[k\left\{\log\left(\frac{ne}{k}\right)- \frac{k-1}{2}H_{p_1}(q'_k)\right\}\right]\ .
\eeqn
Since $k/n \le 1/u_N = o(1)$, and since $\lambda_1$ is bounded, we have $q'_k/p_1\rightarrow \infty$, so that
\beqn
\frac{k-1}{2}H_{p_1}(q'_k)
&\sim& \frac{k-1}{2} \, q'_k \log \left(\frac{q'_k}{p_1}\right) \\
&=& (1-c)^{-1/2}\left[\frac{\lambda_1}{2}\vee 1\right]\left[\log\left(\frac{n}{k-1}\right)+ \log\left\{(1-c)^{-1/2}\left(1\vee \frac{2}{\lambda_1}\right)\right\} \right]\\
&\geq & (1+o(1))(1-c)^{-1/2}\log\left(\frac{n}{k}\right)\ ,
\eeqn
and therefore, since $c \in (0,1)$ is fixed,
\[\log\left(\frac{ne}{k}\right)- \frac{k-1}{2}H_{p_1}(q'_k) \le 1 + \big[1 -  (1+o(1))(1-c)^{-1/2}\big] \log(u_N) \to -\infty \ .\]
We conclude that 
\[\sum_{k=k_{\min}+1}^{n/u_N}\mathbb{P}_S\left[W_{k,S}^*\geq \omega'_k\right]= o(1) \ .\]
Let us now prove that $\omega'_k\leq w_k$.  Indeed, this inequality holds if, and only if, $\lambda_1\leq 2(1-c)/(1-\alpha)$ and $c \le \alpha$.  The second inequality is by definition of $c$, while the first inequality is ensured by \eqref{eq:condition_sup_W_kS} since
\[\frac{\lambda_1}{2}\frac{n-1}{n} = \E_S[W_{n,S}^*/n]\leq\sup_{k\leq n}\E_S[W_{k,S}^*/k]\leq (1-2c)/(1-\alpha) \ .\]

\medskip
Let us turn to integers $k$ satisfying $k\geq n/u_N$. Let $c_0=(1-c)^{-1/2}-1$ and $t=c_0\E_S[W^*_{k,S}]$.  By taking any fixed subset $T \subset S$ of size $|T| = k$, we derive 
\beq \label{Wstar-lb}
\E_S[W^*_{k,S}] \ge \E_S[W_T] = p_1 \kk \ge \frac{\lambda_1}n (n/u_N)^{(2)} \asymp \frac{n}{u_N^2} \to \infty \ ,
\eeq 
so that $t$ satisfies the condition of Lemma \ref{lem:concentration_Wk} eventually.  Using that lemma, we derive that 
\beqn
\mathbb{P}_S\left[W^*_{k,S}\geq \E_S[W^*_{k,S}](1-c)^{-1/2}\right]\leq \exp\left[-\E_S[W_{k,S}^*]\frac{\log(2)}{4}c_0\left[1\wedge \frac{c_0}{8}\right]\right]
\eeqn
By Condition \eqref{eq:condition_sup_W_kS}, $w_k\geq \E_S[W^*_{k,S}](1-c)^{-1/2}$. Hence, there exists a positive constant $\kappa$, such that
\beqn
\sum_{k=n/u_N}^n\mathbb{P}_S\left[W^*_{k,S}\geq w_k\right]&\leq& \sum_{k=n/u_N}^n \exp\left[-\kappa \E_S[W_{k,S}^*]\right]\\
&\leq & n\exp\left[-\kappa\E_S\left[W_{\frac{n}{u_N},S}^*\right]\right]
\eeqn
Because of \eqref{Wstar-lb} and the fact that $\log(N)=o(n)$, we have
\[\E_S\left[W_{\frac{n}{u_N},S}^*\right] \succ \frac{n}{\log^2(n)}\ ,\]
and therefore the sum above goes to $0$.
\end{proof}

\subsection{No test is asymptotically powerful} \label{sec:not-powerful}

When $\lambda_0$ is bounded away from 0 and infinity, the triangle test has some non-negligible power as long as $\lambda_1$ is bounded away from 0 (see \secref{triangle}).  This motivates us to obtain sufficient conditions under which no test is asymptotically powerful.  

Our method is also based on bounding the first two moments of a truncated likelihood ratio $\Lt$.  Indeed, it is enough to show that $\liminf \E_0 \Lt > 0$ and $\liminf \E_0 [\Lt^2] < \infty$.  This comes from the following result.

\begin{lem} \label{lem:contig}
Let $\P_0$ and $\P_1$ be two probability distributions on the same probability space, with densities $f_0$ and $f_1$ with respect to some dominating measure.  
Let $\Gamma$ be any event and define the truncated likelihood ratio $\tilde L = L \,\1_\Gamma$, where $L = f_1/f_0$ is the likelihood ratio for testing $\P_0$ versus $\P_1$.  Then any test for $\P_0$ versus $\P_1$ has risk at least
\[\frac4{27} \frac{(\E_0 \tilde L)^3}{\E_0 [\tilde L^2]} \ , \]
where $\E_0$ denotes the expectation under $\P_0$, and by convention $0/0 = 0$.
\end{lem}

\begin{proof}
Assume $\E_0 \tilde L \ne 0$, for otherwise the result is immediate.  
The risk of the likelihood ratio test $\{L > 1\}$ --- which is the test that optimizes the risk --- is equal to
\[
B := 1 - \frac12 \E_0 |L - 1| = 1 - \E_0 (1 - L)_+ \ge 1 - \E_0 (1 - \tilde L)_+ \ ,
\]
since $\tilde L \le L$.
For any $t \in (0,1)$, we have
\[ \E_0 (1 - \tilde L)_+ \le (1-t) \P_0(\tilde L > t) + \P_0(\tilde L \le t) = 1 - t \P_0(\tilde L > t)\ .\]
Moreover, using the Cauchy-Schwarz inequality, we have for any $t > 0$
\beqn
\E_0 \tilde L &=& \E_0[ \tilde L \, \1_{\{\tilde L \le t\}}] + \E_0 [\tilde L \, \1_{\{\tilde L > t\}}] \\
& \le & t + \sqrt{\E_0 [\tilde L^2] \ \P_0(\tilde L > t)}\ ,
\eeqn
so that, taking $t < \E_0 \tilde L$, we have
\[
\P_0(\tilde L > t) \ge \frac{(\E_0 \tilde L - t)^2}{\E_0 \tilde L^2} \ .
\]
We conclude that
\[B \ge t \P_0(\tilde L > t) \ge t \ \frac{(\E_0 \tilde L - t)^2}{\E_0 \tilde L^2} \ ,\]
and optimizing this over $0 < t < \E_0 \tilde L$ yields the result.
\end{proof}

Since we only need to focus on the case where $\lambda_0$ is bounded from 0 and infinity, and where $\lambda_1$ is bounded from 0 (because the other cases are covered by \thmref{lower}), we may assume they are fixed without loss of generality.  In that case $\zeta \to 0$ is equivalent to $n^2/N \to 0$, which is what we assume in the following.

\begin{thm} \label{thm:lower-no}
Write $n=N^{\kappa}$ with $0<\kappa<1/2$, and assume that $\lambda_0$ and $\lambda_1$ are both fixed.
No test is asymptotically powerful in all the following situations:
\beq \label{lower-no1}
\lambda_1 < 1, \quad \lambda_1^2e \le \lambda_0 \ ;
\eeq
\beq \label{lower-no2}
\lambda_1 < 1, \quad \lambda_1^2 e > \lambda_0, \quad \frac{1-2\kappa }{\kappa }\frac{I_{\lambda_1}}{\log\left(\frac{e\lambda_1^2}{\lambda_0}\right)}>1\ .
\eeq
\end{thm}

\begin{proof}[Proof of \thmref{lower-no}]

We use the same truncation as in \secref{lower-b}, using the same notation $\Gamma_S$ and $f_n$ defined there, and still denote the resulting truncated likelihood by $\tilde L$.

For the first moment, by symmetry, 
\[\E_{0}[\tilde{L}]=\mathbb{P}_S[\Gamma_S] = \P_S[\text{$\mathcal{G}_S$ is a forest}, \ |\cC_{\max, S}| \le f_n] \ .\]
We already saw that $\P_S[|\cC_{\max, S}| \le f_n] \to 1$ \citep[Th.~4.4]{remco:lecture}. 
Consequently, 
\[\E_{0}[\tilde{L}] = \P_{S}[\cG_S\text{ is a forest}] + o(1)\ .\]
Of course, $\cG_S$ is a forest if, and only if, it has no cycles.
By \cite{takacs}, the number of cycles in $\cG_S$ converges weakly to a Poisson distribution with mean 
\[a(\lambda_1)= \frac{1}{2}\log\left(\frac{1}{1-\lambda_1}\right)- \frac{\lambda_1}{2}- \frac{\lambda_1^2}{4}\ ,\]
when $\lambda_1<1$ is fixed.
As a consequence, $\E_{0}[\tilde{L}]= \exp\left[-a(\lambda_1)\right]+o(1)$, which remains bounded away from zero.

For the second moment, we start from \eqref{2nd-moment-lower-b}:
\[
 \mathbb{E}_0\left[\tilde{L}^2\right]-1 
\prec  \sum_{j=1}^\infty  \left(\frac{n^2e}{N} \big[1 \vee \frac{\lambda_1^2e}{\lambda_0}\big]^{f_n} \right)^{j} \ ,
\]
with $f_n=(1+c) I_{\lambda_1}^{-1} \log n$ and $c$ is a small positive constant. 
Under \eqref{lower-no1}, we have $\lambda_1^2 e \le \lambda_0$ and the RHS is $O(n^2/N) = o(1)$.
Under \eqref{lower-no2}, we have $\lambda_1^2 e > \lambda_0$, and the RHS is, as before, equal to \eqref{2nd-moment-bound}.  Here we have
\[A_n = \big[1-2\kappa - (1+c) \frac{\kappa}{I_{\lambda_1}} \log\left(\frac{\lambda_1^2e}{\lambda_0}\right) \big] \log N \to \infty\ , \]
when \eqref{lower-no2} is satisfied and $c$ is small enough.
Hence, in any case, we found that $\mathbb{E}_0\big[\tilde{L}^2\big] \le 1 + o(1)$.

\end{proof}

\section{Discussion} \label{sec:discussion}

\subsection{Adapting to unknown $p_0$ and $n$}

In \citep{subgragh_detection}, we discussed in detail the case where $p_0$ is unknown.  In this situation, the total degree test is not applicable, and we replaced it with a test based on the difference between two estimates for the degree variance.  On the other hand, the scan test (based on \eqref{scan-stat}) can be calibrated in various ways without asymptotic loss of power --- for example, by plugging in the estimate $\hat p_0 = \frac{W}{\NN}$ in place of $p_0$.  We showed that a combination of degree variance test and the scan test are optimal when $p_0$ is unknown, so that the degree variance test can truly play the role of the total degree test in this situation.  
We believe this is the case here also.  In addition to that, the broad scan test (based on \eqref{new-scan}) can also be calibrated without asymptotic loss of power, and the same is true for all the other tests that we studied here, except for the largest connected component test in the supercritical regime.

We also discussed in \citep{subgragh_detection} the case where the size of the subgraph $n$ is unknown.  This only truly affects the broad scan test, whose definition itself depends on $n$.  As we argued in our previous paper, it suffices to apply the procedure to all possible $n$'s, meaning, consider the multiple test based on a combination of the statistics
\[
 W_n^\ddag,\ n= 1,\dots,N/2
\] 
with a Bonferroni correction.
The concentration inequalities that we obtained for $W_n^\ddag$ can accommodate an additional logarithmic factor that comes out of applying the union to control this statistic under $\P_0$, and from this we can immediately see that the test is asymptotically as powerful (up to first order).

\subsection{Open problems}  
The cases where $\lambda_0 \to 0$ and where $\liminf \lambda_0 \ge e$ are essentially resolved.  Indeed, in the first situation, the largest connected component test is asymptotically optimal by \thmref{CC-sub} and \thmref{lower} case \eqref{lower1}, while in the second situation the broad scan test is asymptotically optimal by \thmref{broad} and \thmref{lower} cases \eqref{lower3} and \eqref{lower4}, together with \thmref{lower-no}.  
The case where $0 < \lambda_0 < e$ is fixed is not completely resolved.  Since the triangle test has non-negligible power as soon as $\lambda_1$ is bounded away from 0, consider $\tau$ defined as the largest real such that no test for $\bbG(N, \frac{\lambda_0}N)$ versus $\bbG(N,\frac{\lambda_0}N;n,\frac{\lambda_1}n)$ is asymptotically powerful when $\limsup \lambda_1 < \tau$.  Theorems~\ref{thm:CC-sub} and~\ref{thm:k-tree} provide some upper bounds on $\tau$.

\begin{open}
Compute $\tau$ as a function of $\lambda_0$ and $\kappa := \limsup \frac{\log n}{\log N}$.
\end{open}

Although we proved that the broad scan test was asymptotically optimal when $\liminf \lambda_0 \ge e$, its performance was described only indirectly in terms of $\lambda_1$ in the case \eqref{broad1}.  

\begin{open}
Compute, as a function of $\lambda_1$, the limits inferior and superior of 
\[\sup_{k=n/u_N}^n \frac{\E_S[W_{k,S}^*]}{k} \ .\]
\end{open}

We also formulate an open problem that connects directly with the planted clique problem.  We saw that the broad scan test is powerful when $\lambda_1$ is sufficiently large, but we do not know how to compute it in polynomial time.  Is there a  polynomial-time test that can come close to that?

\begin{open}
Find a polynomial-time test that is asymptotically powerful for testing $\bbG(N, p_0)$ versus $\bbG(N,p_0;n,p_1)$ when $n^2/N = O(1)$, while $\lambda_0 \to \infty$ and $\lambda_1 = O(1)$.  
\end{open}

\section{Proofs of auxiliary results} \label{sec:aux}

\subsection{Proof of \lemref{upper_bound_expectation_wk}}

Fix $\epsilon>0$ and define $x:= 2\left[(1+\epsilon)+ \sqrt{(1+\epsilon)^2+\lambda_1(1+\epsilon)}\right]$.  First, we control the deviations of $W_{k,S}^*$. 
 Define $q_k= (\lambda_1+x)/(k-1)$ and notice that $q_k \ge p_1$ for $n/u_N \le k \le n$. Since $\log(1+t)\leq t$ for any $t > -1$, we have
\[H_{p_1}(q_k):= q_k\log\left(\frac{q_k}{p_1}\right)+ (1-q_k)\log\left(\frac{1-q_k}{1-p_1}\right)\geq  q_k\log\left(\frac{q_k}{p_1}\right) -q_k+p_1 \ .\]
Applying an union bound and Chernoff inequality \eqref{chernoff}, we control the deviations of $W_{k,S}^*$: 
\[
\P_S\left[W_{k,S}^*\geq k^{(2)}q_k\right] \leq \binom{n}{k}\exp\left[-k^{(2)}H_{p_1}(q_k)\right] \le \exp[k A_k] \ ,
\]
where
\[
A_k := \log\big(\frac{en}k\big)- \frac{k-1}{2}\left(q_k\log\left(\frac{q_k}{p_1}\right) -q_k+p_1\right) \ .
\]
Observe that $x$ is larger than $2$. As a consequence, we obtain
\beqn
A_k&=& 1+\log\left(\frac{n}{k}\right)- \frac{\lambda_1+x}{2}\log\left(\frac{n(\lambda_1+x)}{(k-1)\lambda_1}\right)+ \frac{\lambda_1+x}{2}- \frac{\lambda_1(k-1)}{2n}\\
&\leq & 1 +  \frac{x}{2}-\frac{\lambda_1+x}{2}\log\left(\frac{\lambda_1+x}{\lambda_1}\right)- \frac{\lambda_1}{2}\left[\frac{k-1}{n}-1-\log\left(\frac{k-1}{n}\right)\right]\\
&\leq & 1- \frac{x^2}{4(\lambda_1+x)}\ , 
\eeqn
where we used in the last line the inequalities $t -\log t -1\geq 0$ and $\log(1-t)\leq -t-t^2/2$, valid for any $t\geq 0$.   By definition of $x$, we have $x^2/(4(\lambda_1+x))=1+\epsilon$. In conclusion, we have proved that for any integer $k$ between $n/u_N$ and  $n$
\begin{equation}\label{eq:upper_control_wks}
 \P_S\left[\frac{W_{k,S}^*}{k}\geq  \frac{\lambda_1+x}{2} \right]\leq \exp\left[-k\epsilon\right] \ . 
\end{equation}
Let us now control the lower deviations of $\frac1k W_{k,S}^*$ using Lemma \ref{lem:concentration_Wk} 
\beqn
\P_S\left[\frac{W^*_{k,S}}{k}\leq \E_S\left[\frac{W^*_{k,S}}{k}\right]- \left(\E_S\left[\frac{W^*_{k,S}}{k}\right]\right)^{1/2}\frac{8}{k^{1/2}}\right]\leq 2^{-8}\ . 
\eeqn
For $k$ large enough, $\exp\left[-k \eps\right]\leq 1/2$, which therefore implies that 
\beq\label{eq:upper_expectation_Wk1}
\E_S\left[\frac{W^*_{k,S}}{k}\right] \le \left(\E_S\left[\frac{W^*_{k,S}}{k}\right]\right)^{1/2}\frac{8}{(n/u_N)^{1/2}} + \frac{\lambda_1+x}{2} \ ,
\eeq
since $k \ge n/u_N$.
Taking the supremum over $k$ and letting $n$ go to infinity, we conclude that 
\beqn
\lim\inf  \bigvee_{k=n/u_N}^n \E_S\left[\frac{W^*_{k,S}}{k}\right]\leq  \lim\inf  \frac{\lambda_1+ x}{2}= \lim\inf  \frac{\lambda_1}{2} + (1+\epsilon)+ \sqrt{(1+\epsilon)^2+\lambda_1(1+\epsilon)}\ .
\eeqn
Then letting $\epsilon$ going to zero allows us to prove the first result. 

Now assume that $\lambda_1\to \infty$. From \eqref{eq:upper_expectation_Wk1}, we deduce that 
\[\lim\sup \lambda_1^{-1}\bigvee_{k=n/u_N}^n \E_S\left[\frac{W^*_{k,S}}{k}\right]\leq \frac{1}{2}\ .\]
On the other hand, 
\[\bigvee_{k=n/u_N}^n \E_S\left[\frac{W^*_{k,S}}{k}\right]\geq \frac{\E_S [W^*_{n,S}]}{n}= \lambda_1\frac{n-1}{2n}\sim \frac{\lambda_1}{2}\ .\]
This concludes the proof.

\subsection{Some combinatorial results} \label{sec:comb}

We state and prove some combinatorial results.

\begin{lem}[Extension of Cayley's identity]\label{lem:cayley}
 The number $T_k^{(\ell)}$ of labelled trees of size $\ell$ containing a given labelled tree of size $k$ satisfies
\[T_k^{(\ell)}= k\ell^{\ell-k-1} \ . \]
The number $T_{k_1,\ldots k_r} ^{(\ell)}$ of labelled trees of size $\ell$ containing a given labelled forest with tree components  of size $k_1,\ldots, k_r$ satisfies
\[T_{k_1,\ldots, k_r}^{(\ell)}\leq \left(\frac{k}{r}\right)^r\ell^{\ell-k+r-1} (\ell-k+r -1)^{r-1}\ ,\]
with $k=\sum_{i=1}^r k_i$.
\end{lem}

\begin{proof}
The proof relies on the double counting argument of Pitman \citep{MR2569612}. Noting $\cT$ the fixed tree of size $k$, we count in two ways the number of labelled trees of size $\ell$ that contain $\cT$ and whose vertices outside $\cT$ have been ordered. Straightforwardly, we have $T_k^{(\ell)}(\ell-k)!$ such trees. Alternatively, we  consider the following way of building such a labelled ordered tree: 
\begin{enumerate}
 \item Start from $\cT$. 
 \item Choose any vertex $\tilde{u}_0$ among the original tree $\cT$ and any vertex $\tilde{v}_0$ among the $(\ell-k)$  remaining vertices. Add an edge between $\tilde{u}_0$ and $\tilde{v}_0$. Root the given tree --- now of size $k+1$ --- at $\tilde{v}_0$.
Consider all the $\ell-k-1$ remaining vertices as rooted trees of size 1.
\item Then, perform the iterative construction of Pitman.  At each step  $i=1,\ldots, \ell-k-1$, add an edge in the following way: choose any starting vertex $u_i$ among the $\ell$ vertices and note $\rho_i$ the root of the tree containing $u_i$. Choose any ending vertex $v_i$ among the $(\ell-k-i)$ roots other than $\rho_i$.  This so-obtained tree is rooted at $\rho_i$.
\item Let $v_{\ell-k}$ denote the root of the final tree.
\end{enumerate}
All in all,  we have $k\ell^{\ell-k-1}(\ell-k)!$ such constructions and the sequence $v_1, \dots, v_{\ell-k}$ obtained in Step 3 provides an ordering for the vertices not in $\cT$. 

\begin{lem}\label{lem:bijection}
 For any labelled tree $\tilde{\cT}$ of size $\ell$ that contains $\cT$ and whose vertices outside $\cT$ have been ordered, there exists one and only one construction of $\tilde{\cT}$ based on the algorithm above.
\end{lem}
Comparing the two counts  leads to the desired result.

\begin{proof}[Proof of \lemref{bijection}]
Let us slightly modify the iterative construction of Pitman by putting an orientation on the added edges: the first edge is oriented from $\tilde{v}_0$ to $\tilde{u}_0$. For any $i=1,\ldots ,\ell-k-1$, the edge between $u_i$ and $v_i$ is oriented from $u_i$ to $v_i$. The so-obtained partially oriented tree is noted $\overrightarrow{\cT}_{u,v}$. 

Observe that except for $v_{\ell-k}$ which has no parents, all other nodes $v_i$ have one and only one parent. Also, observe that except for the edge $\tilde{v}_0\rightarrow \tilde{u}_0$, all edges between nodes in the subtree $\cT$ and nodes in $\{v_1,\ldots v_{\ell-k}\}$  leave the subtree $\cT$. By a simple induction, this leads us to the following claim: 

\medskip
\noindent 
{\bf Claim 1}: All partially oriented tree $\overrightarrow{\cT}_{u,v}$ based on Pitman construction with sequences $(u,v) = (\tilde{u}_0,\tilde{v}_0,u_1,\dots,u_{\ell-k-1},v_1,\dots,v_{\ell-k-1},v_{\ell-k})$ satisfy the following property
\beqn
(P) \quad \left\{
\begin{array}{l}
\text{Any edge in $\cT$ is undirected,}\\
\text{Any edges on the unique path between $v_{\ell-k}$ and $\cT$ is oriented towards $\cT$,}\\ 
 \text{Any other edge (not in $\cT$) is oriented in the  opposite direction to  $\cT$}.
\end{array}
\right.
\eeqn

In fact, this property characterizes the oriented partially trees $\overrightarrow{\cT}_{u,v}$. \\

\noindent 
{\bf Claim 2}: Conversely, for any sequence $v=(v_1,\ldots, v_{\ell-k})$ and any partially oriented tree $\overrightarrow{\cT}$ of size $\ell$ satisfying $(P)$,
there exists a unique sequence, $(\tilde{u}_0,\tilde{v}_0, u_1,\dots,u_{\ell-k-1})$ such that $\overrightarrow{\cT}_{u,v} = \cT$. 

{\it Proof of Claim 2}. The uniqueness is straightforward. Given $\overrightarrow{\cT}$, define $\tilde{u}_0$ as the unique child in $\cT$ and define $\tilde{v}_0$ the parent of $\tilde{u}_0$. For any $i=1,\ldots, \ell-k-1$, denote $u_i$ the parent of $v_i$. These sequences are lawful for the Pitman construction. Indeed, at step $i$, $v_i$ is not in the same connected component as $u_i$ and $v_i$ is still a root of a connected component. 

Then, the lemma proceeds from the fact that for any tree $\cT$ of size $\ell$ and any sequence $v=(v_1,\ldots,v_{\ell-k})$, there exists one and only partially orientation of $\cT$ satisfying $(P)$.
\end{proof}

\medskip
We now prove the second part of \lemref{cayley}, relying on the same double counting argument. Write $k=k_1+\ldots +k_r$.
 Noting $\cF$ the fixed forest with (labelled) connected components $\cT_1, \dots, \cT_r$ of respective sizes $k_1,\ldots, k_r$, we count in two ways the number of labelled trees of size $\ell$ that contain $\cF$ and whose vertices outside $\cF$ have been ordered. Straightforwardly, we have $T_{k_1,\ldots, k_r}^{(\ell)}(\ell-k)!$ such trees. Alternatively, we consider the following Pitman construction:
\begin{enumerate}
 \item Start from $\cF$. 
 \item For any $j=1,\dots, r$, choose any vertex $w_j \in \cT_j$. Root $\cT_j$ at $w_j$.
Consider all the $\ell-k$ remaining vertices as rooted trees of size 1.
\item Then, perform the iterative construction of Pitman: at each step  $i=1,\ldots, \ell-k+r-1$, add an edge in the following way: choose any starting vertex $u_i$ among the $\ell$ vertices and note $\rho_i$ the root of the tree containing $u_i$. Choose any ending vertex $v_i$ among the remaining $(\ell-k+r -i)$ roots  other than $\rho_i$. The resulting tree is rooted at $\rho_i$.
\item Let $v_{\ell-k+r}$ denote the root of the final tree.
\end{enumerate}
All in all, we have $\big(\prod_{j=1}^r k_j\big)\ell^{\ell-k+r-1}(\ell-k+r -1)!$ such constructions.  And the sequence $v_1, \dots, v_{\ell-k+r}$ obtained in Step 3 provides an ordering of the vertices outside $\cF$ if we ignore the $w_j$'s in that sequence.  
\begin{lem}\label{lem:surjection}
Any tree that contains $\cF$ and whose vertices outside $\cF$ are ordered by the sequence $(t_1,\ldots,t_{\ell-k})$ can be constructed in this way. 
\end{lem}
Consequently we have 
\beqn
T_{k_1,\ldots, k_r}^{(\ell)}(\ell-k)!\leq \big(\prod_{i=1}^r k_i\big) \, \ell^{\ell-k+r-1}(\ell-k+r -1)! \ ,
\eeqn
from which we derive the (crude) bound 
\[T_{k_1,\ldots, k_r}^{(\ell)}\leq \left(\frac{k}{r}\right)^r\ell^{\ell-k+r-1} (\ell-k+r -1)^{r-1}\ .\]
\end{proof}
\begin{proof}[Proof of \lemref{surjection}]
Consider a tree $\cT^{\ell}$ that contains $\cF$ and whose vertices outside $\cF$ are ordered in the following sequence $(t_1,\ldots,t_{\ell-k})$.

\medskip\noindent 
{\bf Claim.}  There exists a (non-necessarily unique) orientation of the edges outside $\cF$ such that any node in $t_1,\ldots, t_{\ell-k-1}$ has exactly one parent, $t_{\ell-k}$ has no parent, and any tree $\cT_i$ in $\cF$ has exactly one parent. 

{\it Proof of Claim 1}: Collapse each of the trees $\cT_i$ into a single node, to obtain the tree $\cT^{\ell-k+r}$ with $\ell-k+r$ nodes. Then, we prove the result for $\cT^{\ell-k+r}$ by a simple induction on the number of nodes. 

 For any $i\in \{1,\ldots ,r\}$, define $\omega_i$ as the unique node in $\cT_{i}$. 
We define the sequence $v:=(\omega_1,\ldots,\omega_r,t_1,\ldots t_{\ell-k})$. Finally, we define $u_i$ as the unique parent of $v_i$ for any $i\leq \ell-k+r-1$. It is straightforward to check that these sequences $\omega$, $v$ and $u$ are lawful for the Pitman construction and allow to build $\cF$. 
\end{proof}

\subsection{Proof of \lemref{Nkt_var0}} \label{sec:Nkt_var0}
By definition, 
\[\Var_0[\tree_k]= \sum_{C_1,C_2} \big( \mathbb{P}_0[\cG_{C_1}\text{ and } \cG_{C_2}\text{ are trees}]-  \mathbb{P}_0[\cG_{C_1}\text{ is a tree}] \mathbb{P}_0[\cG_{C_2}\text{ is a tree}] \big) \ ,\]
where the sum ranges over subsets $C_1$, $C_2$ of size $k$.

In the sequel, we let $q = |C_1 \cap C_2|$ and let $r$ denote the number of connected components of $C_1 \cap C_2$.  Note that, when $q = 0$, the corresponding terms in the sum above are zero.  When $q \ge 1$, we define 
\[B_{r,q} = \mathbb{P}_0\big[\text{$\cG_{C_1}, \cG_{C_2}$ are trees and $\cG_{C_1\cap C_2}$ has $r$ connected components}\big] \ ,\]
so that
\[\mathbb{P}_0[\cG_{C_1}\text{ and } \cG_{C_2}\text{ are trees}] = \sum_{r=1}^q B_{r,q} \ .\]
Note that $\cG_{C_1\cap C_2}$ is a forest when $\cG_{C_1}\text{ and } \cG_{C_2}\text{ are trees}$.  

We derive $B_{1,q}$ first. Under the event $\{\cG_{C_1}, \cG_{C_2}\text{ and }\cG_{C_1\cap C_2}\text{ are trees}\}$, there are exactly $2k-1-q$ edges in $\cG_{C_1}\cup\cG_{C_2}$ among the potential 
$2k^{(2)}-q^{(2)}$ edges. Let us count the number of configurations compatible with this event. By Cayley's identity, there are $q^{q-2}$  configurations for the tree $\cG_{C_1\cap C_2}$. The tree $\cG_{C_1\cap C_2}$ being fixed, we apply \lemref{cayley} to derive that there are $qk^{k-q-1}$ configurations for $\cG_{C_1}$ and $qk^{k-q-1}$ configurations for $\cG_{C_2}$. All in all, we get
\beqn
B_{1,q} &= &q^{q-2}[qk^{k-q-1}]^2p_0^{2k-1-q}(1-p_0)^{2k^{(2)}-q^{(2)}-2k+1+q}\ .
\eeqn 
Then, we upper bound $B_{r,q}$ for $r\geq 2$.  Under the event defined in $B_{r,q}$, there are $2k-2-q+r$ edges in $\cG_{C_1}\cup\cG_{C_2}$ among the potential 
$2k^{(2)}-q^{(2)}$ edges. By \lemref{combinatorics_forest} 
, there are less than $q^{q-2}$ configurations for the forest $\cG_{C_1\cap C_2}$ with $r$ connected components. $\cG_{C_1\cap C_2}$  being fixed, \lemref{cayley} tells us that there are less than $\left(\frac{q}{r}\right)^rk^{k-q+r-1} (k-q+r -1)^{r-1}$ possible configurations to complete $\cG_{C_1}$ and (independently) for $\cG_{C_2}$. It then follows that
\beqn
B_{r,q}&\leq&  q^{q-2}\left[\left(\frac{q}{r}\right)^rk^{k-q+r-1} (k-q+r -1)^{r-1}\right]^2p_0^{2k-q+r-2}(1-p_0)^{2k^{(2)}-q^{(2)}-2k+q-r+2}\\
&\leq & B_{1,q} \left(\frac{p_0}{1-p_0}\right)^{r-1} k^{6r-4}\ ,
\eeqn
using the fact that $q \le k$.
Summing over $r$ leads to 
\beqn
\sum_{r=2}^q B_{r,q}\leq B_{1,q} \, k^{2}\sum_{r=2}^q \left[\frac{p_0 k^6}{1-p_0}\right]^{r-1} \le B_{1,q} \, \frac{p_0 k^8}{1-p_0-p_0k^6} = o\big(B_{1,q}\big)\ ,
\eeqn
since $p_0k^{8}=o(1)$. 
Thus, when $|C_1 \cap C_2| = q \ge 1$, we obtain
\beqn
\mathbb{P}_0[\cG_{C_1}\text{ and } \cG_{C_2}\text{ are trees}] &=& B_{1,q} + o\big(B_{1,q}\big) \\
&\prec&  q^{q}k^{2k-2q-2}p_0^{2k-1-q}\ .
\eeqn 
 
We can now bound the variance. The number of subsets $(C_1,C_2)$ of size $k$ such that $C_1\cap C_2=q$ equals $\binom{N}{k}\binom{k}{q}\binom{N-k}{k-q}$. Thus, we derive
\begin{eqnarray*}
 \Var_0[\tree_k]&\prec& \sum_{q=1}^k \binom{N}{k}\binom{k}{q}\binom{N-k}{k-q}q^qk^{2k-2q-2}p_0^{2k-q-1}\\
&\prec& N  \sum_{q=1}^k \frac{q^qk^{2k-2q-2}}{q!(k-q)!^2}\lambda_0^{2k-2q-1}\\
&\prec& N\sum_{q=1}^k \frac{(\lambda_0 e)^k}{\lambda_0 k^2} A_{k-q}, \qquad A_\ell := \left(\frac{k \sqrt{\lambda_0 e}}{\ell}\right)^{2 \ell}\ ,
\end{eqnarray*}
by Stirling's lower bound.  By convention, $A_0 = 1$.
The function $\ell \to A_\ell$ is easily seen to be increasing over $(0, k \sqrt{\lambda_0/e})$ and decreasing over $(k \sqrt{\lambda_0/e}, \infty)$.  Thus, when $\lambda_0 < e$, we have $A_{k-q} \le A_{k \sqrt{\lambda_0/e}}$; and when $\lambda_0 > e$, we have $A_{k-q} \le A_k$; this is for all $q =1, \dots, k$.  Then summing over $q$, we obtain the stated bounds in each case.

\subsection{Proof of \lemref{alternative_tree}} \label{sec:alternative_tree}
 
First, we deal with the expectation.
\begin{eqnarray*}
\mathbb{E}_S[\tree_{k,S,q}]&= & \sum_{C_1\subset S, |C_1|=q} \ \sum_{C_2\subset S^\comp, |C_2|=k-q}\mathbb{P}_S [\cG_{C_1}\text{ and }\cG_{C_1\cup C_2}\text{ are trees}]\ .
\end{eqnarray*}
When $\cG_{C_1}$ and $\cG_{C_1\cup C_2}$ are both trees, there are $q-1$ edges in $\cG_{C_1}$ and $k-q$ additional edges in $\cG_{C_1\cup C_2}$. The number of configurations for $\cG_{C_1}$ is $q^{q-2}$ (Cayley's Identity). By  \lemref{cayley}, when $\cG_{C_1}$ is fixed, there remains $qk^{k-q-1}$ possible configurations for  $\cG_{C_1\cup C_2}$.
As for the previous variance computation, we apply to control this probability. Hence, we get
\beqn
\mathbb{P}_S [\cG_{C_1}\text{ and }\cG_{C_1\cup C_2}\text{ are trees}]
&=&q^{q-2}qk^{k-q-1}p_1^{q-1}p_0^{k-q}(1-p_1)^{q^{(2)}-q+1}(1-p_0)^{k^{(2)}-q^{(2)}-k+q} \\
&\succ& q^{q-1} k^{k-q-1} p_1^{q-1} p_0^{k-q} \ ,
\eeqn
since $(q^{(2)}-q+1) p_1 \le k^2 p_1 \asymp k^2/n = o(1)$ and $(k^{(2)}-q^{(2)}-k+q) p_0 \le k^2 p_0 \asymp k^2/N = o(1)$.  
Hence, using the fact that $nk=o(N)$ and the usual bound $m! \le \sqrt{m} (m/e)^m$, and we derive 
\beqn
\mathbb{E}_S[\tree_{k,S,q}]&\succ& \binom{n}{q}\binom{N-n}{k-q} q^{q-1} k^{k-q-1}p_1^{q-1}p_0^{k-q}
\\
& \succ & n \frac{q^{q-1}k^{k-q-1}\lambda_1^{q-1}\lambda_0^{k-q}}{q!(k-q)!} \\
&\succ &  n \frac{(e\lambda_1)^k}{\lambda_1k^{3}}\left(\frac{\lambda_0k}{\lambda_1(k-q)}\right)^{k-q}\ .
\eeqn
This quantity is maximized with respect to $q$ when $(k-q)/k=\lambda_0/(\lambda_1e)$, and taking  $ q:=k-\lfloor \frac{\lambda_0}{\lambda_1 e}k\rfloor$ leads to
\[
\mathbb{E}_S[\tilde{N}_{k,S,q}^T]\succ
 n \frac{(e\lambda_1)^k}{\lambda_1k^{3}} \exp\left(\frac{\lambda_0}{\lambda_1e}k\right) \ .
\]

Let us turn to the variance. Again, we decompose it as a sum over $(C_1,C_2)\subset S^2$ and $(C_3,C_4)\subset (S^\comp)^2$ depending on the sizes $s = |C_1\cap C_2|$ and $r = |C_3\cap C_4|$. By independence of the edges, only the subsets such $(r,s)\neq (0,0)$ play a role in the variance.  We have
\beqn
\Var_S[\tilde{N}^T_{k,S,q}]&\leq & \sum_{s=1}^{q}\sum_{r=0}^{k-q}\ \sum_{|C_1\cap C_2|=s} \ \sum_{|C_3\cap C_4|=r}\mathbb{P}_S[\cG_{C_1},\  \cG_{C_2},\ \cG_{C_1\cup C_3},\ \cG_{C_2\cup C_4}\text{ are trees} ]\\
&&+ \sum_{r=2}^{k-q} \ \sum_{|C_1\cap C_2|=0} \ \sum_{|C_3\cap C_4|=r}\mathbb{P}_S[\cG_{C_1},\  \cG_{C_2},\ \cG_{C_1\cup C_3},\ \cG_{C_2\cup C_4}\text{ are trees} ]\\ &=& B_1+B_2\ .
\eeqn 
First we consider the sum $B_1$ where $r$ is  positive.
Therefore, fix $C_1,C_2 \subset S$ and $C_3,C_4\subset S^\comp$ with $|C_1| = |C_2| = q$, $|C_3| = |C_4| = k-q$, $|C_1\cap C_2| = s \ge 1$ and $|C_3\cap C_4| = r \ge 1$, and for $1 \le t_1 \le s$ and $1 \le t_2 \le r+s$, define 
\beqn
 \cA_{(t_1,t_2)}
 &:=& \left\{\cG_{C_1},\  \cG_{C_2},\ \cG_{C_1\cup C_3},\ \cG_{C_2\cup C_4}\ \text{are trees}\right\} \\
 &&\cap \left\{\cG_{C_1\cap C_2} \text{ has $t_1$ connected components}\right\} \\
 &&\cap \left\{\cG_{(C_1\cap C_2)\cup (C_3\cap C_4)} \text{ has $t_2$ connected components}\right\}\ .
\eeqn
(The dependency of $\cA_{(t_1,t_2)}$  on $C_1, C_2, C_3, C_4$ is left implicit.)

We first control $\mathbb{P}_S[\cA_{(1,1)}]$.
Under the event $\cA_{(1,1)}$, the graph $\cG_{C_1}\cup \cG_{C_2}$ contains $2q-1-s$ edges and the graph $\cG_{C_1\cup C_3}\cup \cG_{C_2\cup C_4}$ contains $2(k-q)-r$ additional edges.
 Indeed, the number of edges in the last graph is equal to 
$$(|C_1 \cup C_3| -1) + (|C_2 \cup C_4| -1) - (|(C_1\cap C_2)\cup (C_3\cap C_4)| - 1) = k-1 + k-1 - (r+s-1) = 2k-r-s-1$$
Applying \lemref{cayley}, there are $s^{s-2}$ possible configurations for $\cG_{C_1\cap C_2}$ and then $[sq^{q-s-1}]^2$ possible configurations to complete $\cG_{C_1} \cup \cG_{C_2}$. The graph $\cG_{C_1}\cup \cG_{C_2}$ been fixed, there are $s(s+r)^{r-1}$ configurations for $\cG_{(C_1\cap C_2)\cup (C_3\cap C_4)}$, since this is a tree with $s+r$ nodes containing the given tree $\cG_{C_1\cap C_2}$ with $s$ nodes.
By the same token, $\cG_{C_1 \cup C_3}$ is a tree with $k$ nodes that includes the given tree $\cG_{C_1 \cup (C_3 \cap C_4)}$ with $q+r$ nodes, and similarly for $\cG_{C_2 \cup C_4}$, there at most $[(q+r)k^{k-q-r-1}]^2$ possible configurations to complete $\cG_{C_1\cup C_3} \cup \cG_{C_2\cup C_4}$.
Thus, we obtain
\beq
\mathbb{P}_S[\cA_{(1,1)}]
\leq   s^{s-2} [s q^{q-s-1}]^2 s(s+r)^{r-1} [(q+r)k^{k-q-r-1}]^2 p_1^{2q-1-s} p_0^{2(k-q)-r} =: A_{1,1}\ .\label{eq:upper_A11}
\eeq

Let us now control the probability of $\cA_{(t_1,t_2)}$ for $t_1$ or $t_2$ strictly larger than one. First, observe that whenever $t_2<t_1$, $\mathcal{A}_{t_1,t_2}$ is empty. Indeed, if $t_2<t_1$, there is a path in $C_3\cap C_4$ between two connected components of $\cG_{C_1\cap C_2}$. However, these two connected components are also related by a different (since $C_1 \cap C_3 = \emptyset$) path in $C_1$ (since $\cG_{C_1}$ is a tree), and that contradicts the fact that $\cG_{C_1\cup C_3}$ is a tree. Hence, we may assume that $t_2\geq t_1$.
By Lemmas \ref{lem:cayley} and  \ref{lem:combinatorics_forest}, there are at most $s^{s-2}$ possible configurations for the forest $\cG_{C_1\cap C_2}$, and when this is fixed, there are at most 
\[\big[(s/t_1)^{t_1} q^{q-s + t_1-1} (q-s+t_1-1)^{t_1-1} \big]^2  \le [sq^{q-s}q^{3(t_1-1)}]^2 \]
possible configurations to complete $\cG_{C_1} \cup \cG_{C_2}$. 
With $\cG_{C_1} \cup \cG_{C_2}$ being fixed, the number of possible configurations for $\cG_{(C_1\cap C_2)\cup (C_3\cap C_4)}$ is at most the number of trees that contain $\cG_{C_1\cap C_2}$ --- which is at most 
\[(s/t_1)^{t_1} (s+r)^{s+r-s +t_1-1} (s+r-s+t_1-1)^{t_1-1} \le s^{t_1}(s+r)^{r+2t_1-2} \]
by \lemref{cayley} --- times $k^{t_2-1}$, which bounds the number of ways of erasing $t_2-1$ edges in this tree to obtain a forest with $t_2$ components. 
The graph $\cG_{C_1\cup (C_3\cap C_4)}$ contains $t_2-t_1+1$ connected components. By \lemref{cayley}, 
 there are no more than 
\[\left[\left(\frac{q+r}{t_2-t_1+1}\right)^{t_2-t_1+1}k^{k-q-r+t_2-t_1}(k-q-r+t_2-t_1)^{t_2-t_1}\right]^2 \leq [(q+r)k^{k-r-q}k^{3(t_2-t_1)}]^2\] possible configurations to complete $\cG_{C_1\cup C_3} \cup \cG_{C_2\cup C_4}$.
 The number of edges in $\cG_{C_1}\cup \cG_{C_2}$ is $2(q-1) - (s-t_1)$, while the number of edges in $\cG_{C_1\cup C_3}\cup \cG_{C_2\cup C_4}$  is 
\[2(k-q)-(r-(t_2-t_1))= 2(k-q)-r+t_2-t_1\ .\]
  All together, and with some elementary simplifications, we arrive at the following bound
\beqn
\mathbb{P}_S[\cA_{(t_1,t_2)}]&\leq & A_{1,1} k^{7(t_2-t_1)+9t_1-6} p_1^{t_1-1}p_0^{t_2-t_1}\ .
\eeqn
Since $k^{O(1)} (p_0 + p_1) =o(1)$, it follows that 
$$\sum_{t_1=1}^{s}\sum_{t_2=1}^{r+s}\P_{S}(\cA_{t_1,t_2})\prec A_{1,1}.$$

Using the definition of $A_{1,1}$ in \eqref{eq:upper_A11} and the definition of $q$,  we bound $B_1$ 
\beqn
B_1
&\prec & \sum_{s=1}^{q} \sum_{r=0}^{k-q} \binom{n}{q}\binom{N-n}{k-q}\binom{q}{s}\binom{k-q}{r}\binom{n-q}{q-s}\binom{N-n-k+q}{k-q-r}A_{1,1}\\
&\prec& n\sum_{s,r}\frac{s^{s+1}q^{2(q-s-1)}(s+r)^{r-1}k^{2(k-q-r)-2}(q+r)^2}{r!s!(q-s)!^2(k-q-r)!^2}\lambda_1^{2q-s-1}\lambda_0^{2(k-q)-r}\\
&\prec &n \sum_{r,s} e^{2k-r-s} \left(\frac{s+r}{r}\right)^r\left(\frac{q}{q-s}\right)^{2(q-s)}\left(\frac{k}{k-q-r} \right)^{2(k-q-r)}\lambda_1^{2q-s-1}\lambda_0^{2(k-q)-r+1}\\
&\prec & \frac{n}{\lambda_1}\sum_{r,s} e^{4k-2q-3r-s } \left(\frac{s+r}{r}\right)^r\left(\frac{q}{q-s}\right)^{2(q-s)}\left(\frac{k-q}{k-q-r} \right)^{2(k-q-r)}\lambda_1^{2k-2r-s-1}\lambda_0^{r}\\
&\prec & \frac{n}{\lambda_1}\sum_{r,s} e^{4k-2q-3r} \left(\frac{q}{q-s}\right)^{2(q-s)}\left(\frac{k-q}{k-q-r} \right)^{2(k-q-r)}\lambda_1^{2k-2r-s-1}\lambda_0^{r}\ .
\eeqn
We have applied Stirling's lower bound in the fourth line; we haved used the definition of $q$ to control $k/(k-q)$  in the fifth line
\[\left(\frac{k}{k-q}\right)^{2(k-q-r)}= \left(\frac{k}{\lfloor\frac{\lambda_0k}{\lambda_1e} \rfloor }\right)^{2(k-q-r)}\leq \left(\frac{\lambda_1e}{\lambda_0}\right)^{2(k-q-r)}\left(1-\frac{\lambda_1e}{k\lambda_0}\right)^{-2k}=O(1) \left(\frac{\lambda_1e}{\lambda_0}\right)^{2(k-q-r)}\ ;\]
 and we have upper-bounded $(1+s/r)^r$ by $e^s$ in the last line. 
Note that
\[
e^{-3r}\left(\frac{k-q}{k-q-r} \right)^{2(k-q-r)}\lambda_1^{-2r}\lambda_0^{r} =  \left(\frac{(k-q)e^{3/2} \frac{\lambda_1}{\sqrt{\lambda_0}}}{k-q-r} \right)^{2(k-q-r)} \left(e^{3/2}\frac{\lambda_1}{\sqrt{\lambda_0}}\right)^{-2(k-q)}
\]
is decreasing with respect to $r$ since $\lambda_1^2 e>\lambda_0$. 
As a consequence, we have  
\beqn
B_1&\prec & \frac{nk}{\lambda_1}\sum_{\ell=1}^{q} e^{4k-2q}\lambda_1^{2k-q}D_\ell\, \quad \quad D_\ell:=\left(\frac{q}{\ell}\right)^{2\ell}\lambda_1^{\ell}
\eeqn
The function $\ell \to D_\ell$ is easily seen to be maximized at $\ell=q\sqrt{\lambda_1}/e$. This allows us to conclude that 
\beqn
B_1&\prec & \frac{nk^2}{\lambda_1} e^{4k-2q+2q\sqrt{\lambda_1}/e}\lambda_1^{2k-q}\ .
\eeqn

Finally, we bound $B_2$ following a similar strategy. First, we observe that the probability of the event $\cB:=\{\cG_{C_1},\  \cG_{C_2},\ \cG_{C_1\cup C_3},\ \cG_{C_2\cup C_4}\text{ are trees}\}$ is equivalent to the probability of the event $\cB_{1}:=\cB\cap \{\cG_{C_3\cap C_4}\text{ is a tree}\}$. This follows from the fact that the event $\cB_{r}:=\cB\cap \{ \cG_{C_3\cap C_4}\text{ contains $r$ trees}\}$ involves $r-1$ more edges  than $\cB_{1}$ while the number of possible configurations in $\cB_{r}$ is  does not increase more than by a factor $k^{O(1)r}$ compared to $\cB_{1}$.
\beqn
B_2 &= & \sum_{r=2}^{k-q} \ \sum_{|C_1\cap C_2|=0} \ \sum_{|C_3\cap C_4|=r} \mathbb{P}_S[\cG_{C_1},\  \cG_{C_2},\ \cG_{C_1\cup C_3},\ \cG_{C_2\cup C_4}\text{ are trees} ]\\
& \prec& \sum_{r=2}^{k-q} \ \sum_{|C_1\cap C_2|=0} \ \sum_{|C_3\cap C_4|=r} \mathbb{P}_S[\cG_{C_1},\  \cG_{C_2},\ \cG_{C_1\cup C_3},\ \cG_{C_2\cup C_4}\ ,\text{ and } \cG_{C_3\cap C_4}\text{ are trees} ]
\\&\prec& \sum_{r=2}^{k-q} r^{r-2}\left(q^{q-2}\right)^2 \left(\left(\frac{q+r}{2}\right)^{2} k^{k-q-r+1}(k-q-r+1)\right)^2 p_1^{2(q-1)}p_0^{2(k-q)-r+1}\\ &&\times \binom{n}{q}^2\binom{N-n}{r}\binom{N-n}{k-q-r}^2\\
&\prec & \sum_{r=2}^{k-q}\frac{n^2}{N}r^2k^4\lambda_1^{2(q-1)}\lambda_0^{2(k-q)-r+1}e^{2k-r}\left(\frac{k}{k-q-r}\right)^{2(k-q-r)}\\
&\prec & k^6\sum_{r=2}^{k-q}\frac{n^2}{N}\lambda_1^{2k-2r-2 }\lambda_0^{r+1}e^{4k-2q-3r}\left(\frac{k-q}{k-q-r}\right)^{2(k-q-r)}\\
&\prec &\frac{k^7n^2}{N}\lambda_1^{2k-2 }\lambda_0e^{4k-2q}\ .
\eeqn
In the third line, we bound the probability by counting the number of edges involved in the event and the number of possible configurations, as we did before.
 In the fourth line, we use the  bound of  $k/(k-q)$ to obtain a ratio of the form $\frac{k-q}{k-q-r}$. In the last line, we observe that the sum is decreasing with respect to $r$ and is maximized at $r=0$.

\subsection*{Acknowledgements}

We would like to thank Jacques Verstraete and Raphael Yuster for helpful discussions and references on counting $k$-cycles.   
The research of N.~Verzelen is partly supported by the french Agence Nationale
de la Recherche (ANR 2011 BS01 010 01 projet Calibration).
The research of E.~Arias-Castro is partially supported by a grant from the Office of Naval Research (N00014-13-1-0257).

\bibliographystyle{chicago}
\bibliography{subgraph-detection}

\end{document}

%% file: macros-aap.tex
\usepackage{graphicx, subfigure}
\usepackage{amsfonts, amsmath, amssymb, amsthm, constants}

\usepackage{cases}
\usepackage{color}
\definecolor{darkred}{RGB}{100,0,0}
\definecolor{darkgreen}{RGB}{0,100,0}
\definecolor{darkblue}{RGB}{0,0,150}

\usepackage{hyperref}
\hypersetup{colorlinks=true, linkcolor=darkred, citecolor=darkgreen, urlcolor=darkblue}
\usepackage{url}
\usepackage{multirow}
\newtheorem{thm}{Theorem}
\newtheorem{prp}{Proposition}
\newtheorem{lem}{Lemma}

\newtheorem{open}{Open problem}

\def\beq{\begin{equation}}
\def\eeq{\end{equation}}
\def\beqn{\begin{eqnarray*}}
\def\eeqn{\end{eqnarray*}}
\def\bitem{\begin{itemize}}
\def\eitem{\end{itemize}}
\def\benum{\begin{enumerate}}
\def\eenum{\end{enumerate}}
\def\bmult{\begin{multline*}}
\def\emult{\end{multline*}}
\def\bcenter{\begin{center}}
\def\ecenter{\end{center}}

\newcommand{\thmref}[1]{Theorem~\ref{thm:#1}}
\newcommand{\prpref}[1]{Proposition~\ref{prp:#1}}
\newcommand{\lemref}[1]{Lemma~\ref{lem:#1}}
\newcommand{\secref}[1]{Section~\ref{sec:#1}}

\def\cA{\mathcal{A}}
\def\cB{\mathcal{B}}
\def\cC{\mathcal{C}}

\def\cE{\mathcal{E}}
\def\cF{\mathcal{F}}
\def\cG{\mathcal{G}}

\def\cI{\mathcal{I}}
\def\cJ{\mathcal{J}}

\def\cN{\mathcal{N}}

\def\cT{\mathcal{T}}

\def\cV{\mathcal{V}}

\def\bW{\mathbf{W}}

\def\bbG{\mathbb{G}}

\def\bbR{\mathbb{R}}

\newcommand{\E}{\operatorname{\mathbb{E}}}
\renewcommand{\P}{\operatorname{\mathbb{P}}}
\newcommand{\Var}{\operatorname{Var}}

\newcommand{\pr}[1]{\mathbb{P}\left(#1\right)}

\def\Bin{\text{Bin}}

\def\eps{\varepsilon}

\usepackage{dsfont}
\newcommand{\1}{\mathds{1}}

\definecolor{purple}{rgb}{0.4,.1,.9}

\def\tot{W}
\def\scan{W_n^*}
\def\nn{n^{(2)}}
\def\NN{N^{(2)}}
\def\kk{k^{(2)}}

\def\Lt{\tilde{L}}
\def\kmin{k_{\rm min}}

\newcommand{\IND}[1]{\1_{\{ #1 \}}}

\def\Cmax{\cC_{\rm max}}

\def\comp{\mathsf{c}}

\def\tree{N^{\rm tree}}

\usepackage{comment}